\documentclass[reqno,11pt]{amsart}
\usepackage{mathrsfs}
\usepackage{stmaryrd}
\usepackage{cases}
\usepackage{epic}
\usepackage{amsfonts}
\usepackage{graphicx}
\usepackage{amsmath}
\usepackage{amssymb, upgreek}
\usepackage{mathtools}
\usepackage{bm}
\usepackage{array}
\usepackage{chngcntr} 
\counterwithin{figure}{section}
\counterwithin{table}{section}
\usepackage{latexsym,todonotes}
\usepackage{pdflscape}
\usepackage[all]{xypic}
\usepackage[all]{xy}
\usepackage{color}
\usepackage{colordvi}
\usepackage{multicol}
\usepackage{tikz}
\usepackage{tikz-cd}
\usepackage[linktocpage=true]{hyperref}
\hypersetup{colorlinks,linkcolor=blue,urlcolor=cyan,citecolor=blue}

\newcommand{\RN}[1]{%
  \textup{\uppercase\expandafter{\romannumeral#1}}%
}

\usepackage{amsxtra}
\usepackage{amsmath}
\usepackage{amscd}
\usepackage{amssymb}
\usepackage{amsfonts}
\usepackage{mathrsfs}
\usepackage{amsthm}
\usepackage{color}
\usepackage{upgreek}
\usepackage{verbatim}  %%for \begin{comment}
\usepackage{enumerate}
\usepackage{latexsym}
\usepackage{graphicx}
\usepackage{tikz}
\usepackage{todonotes}
\usepackage{setspace}
\usepackage{hyperref}
\usepackage{stmaryrd}
\usepackage{nicefrac}
\usepackage{euscript}
\usetikzlibrary{decorations.pathmorphing}
\usepackage{pifont}

 \definecolor{darkgreen}{HTML}{336633}
 \definecolor{darkred}{HTML}{993333}
\makeatletter

\newcommand{\arxiv}[1]{\href{http://arxiv.org/abs/#1}{\tt
    arXiv:\nolinkurl{#1}}}

\allowdisplaybreaks

\newtheorem{alphatheorem}{Theorem}

\newtheorem{alphaproposition}[alphatheorem]{Proposition}

\theoremstyle{plain}
\newtheorem{thm}{Theorem}[section]
\newtheorem*{thm*}{Theorem}

\newtheorem{lem}[thm]{Lemma}
\newtheorem{prop}[thm]{Proposition}

\newtheorem{df-prop}[thm]{Definition-Proposition}

\theoremstyle{remark}
\newtheorem{rem}[thm]{Remark}

\numberwithin{equation}{section}

\newcommand{\U}{\bold{U}}
\newcommand{\Ui}{\bold{U}^\imath}

\newcommand{\qbinom}[2]{\begin{bmatrix} #1\\#2 \end{bmatrix} }

\newcommand{\bluetext}[1]{\textcolor{blue}{#1}}

\newcommand{\blue}[1]{\bluetext{ #1}}

\begin{document}

\title{$\mathrm{i}$Canonical basis arising from quasi-split rank one $\mathrm{i}$quantum group}

\author[Ziming Chen]{Ziming Chen}
\address{Department of Mathematics, University of Virginia, Charlottesville, VA 22903, USA}\email{jrd6an@virginia.edu}

\subjclass[2020]{Primary 17B37}  
\keywords{iQuantum groups, Canonical bases}

\begin{abstract} 
We compute icanonical basis of the quasi-split rank one modified iquantum group, by obtaining explicit transition matrices among the icanonical basis, monomial basis, and standardized canonical basis; all these bases can be naturally categorified. These transition matrices follow from their counterparts computed in this paper among the icanonical basis, monomial basis, and canonical basis on simple finite-dimensional modules of quantum $\mathfrak{sl}_3$.  
\end{abstract}

\maketitle

\tableofcontents

\thispagestyle{empty}

%%%%%%%%%%%

\section{Introduction}

\subsection{Background}

A cornerstone in the theory of Drinfeld-Jimbo quantum groups $\U$ is the canonical basis introduced by Lusztig \cite{Lus90}, which enjoys striking properties such as integrality and positivity, and plays a fundamental role in representation theory, geometry, and categorification. 

Quantum symmetric pairs $(\U,\Ui)$ were formulated by Letzter \cite{Let02} (and generalized in \cite{Kol02}), where the coideal subalgebra $\Ui$ of $\U$ is nowadays known as an iquantum group. The iquantum groups can be viewed as a natural generalization of quantum groups, and the aim of the $\imath$-Program (see \cite{Wan23} for a survey) is to develop $\imath$-analogues of the fundamental structures in the theory of quantum groups, including icanonical bases, ibraid group action, and icategorification; cf. \cite{BW18b, WZ25, BWW25}. The icanonical basis with positivity properties on the modified iquantum group $\dot{\U}^{\imath}$ of quasi-split type AIII was constructed geometrically in \cite{LiW18} (building on \cite{BKLW18}). While the icanonical basis has been constructed in full generality \cite{BW21}, explicit formulas have not been fully available even in all rank one cases.

Unlike the usual quantum group setting, where there is a unique rank one $\U_{q}(\mathfrak{sl}_{2})$, there are three rank one quasi-split iquantum groups, classified by Satake diagrams, i.e. pairs consisting of a Dynkin diagram together with a diagram involution, as follows. 
\vspace{2mm}
\begin{center}
\renewcommand{\arraystretch}{1.8} % overall row spacing

\tikzset{
  satake node/.style={circle, draw, fill=white, inner sep=2pt},
  satake involution/.style={<->, bend left=30, blue, yshift=2pt}
}

\begin{tabular}{c c c}
\hline
Type & Dynkin type & Satake diagram \\
\hline
Diagonal &
$A_1 \times A_1$ &
\begin{tikzpicture}[baseline=-0.6ex]
  \node[satake node] (a) at (0,0) {};
  \node[satake node] (b) at (1.6,0) {};
  \draw[satake involution]
    (a) to node[above=-1.2pt]{\scriptsize$\tau$} (b);
\end{tikzpicture}
\\%[0.3em] % extra space after first row

Split &
$A_1$ &
\begin{tikzpicture}[baseline=-0.6ex]
  \node[satake node] at (0.8,0) {};
\end{tikzpicture}
\\%[0.1em]

Quasi-split &
$A_2$ &
\begin{tikzpicture}[baseline=-0.6ex]
  \node[satake node] (a) at (0,0) {};
  \node[satake node] (b) at (1.6,0) {};
  \draw (a) -- (b);
  \draw[satake involution]
    (a) to node[above=-1.2pt]{\scriptsize$\tau$} (b);
\end{tikzpicture}
\\
\hline
\end{tabular}
\end{center}
\vspace{2mm}

Let us quickly recall several distinguished bases for rank-one iquantum groups. First, the iquantum group of diagonal type is isomorphic to the usual quantum group $\U_q(\mathfrak{sl}_2)$. In this case, the canonical basis of the modified quantum $\mathfrak{sl}_2$ is known (cf. \cite{Ka93}, \cite[25.3.1]{Lus94}) and the standardized canonical basis (called fused canonical basis back then) was introduced and computed in \cite{Wan25}.
Secondly, for the split type, closed formulas for the icanonical basis on the modified iquantum group (known as idivided powers), as well as explicit expressions of the icanonical basis in terms of the canonical basis at the module level, were obtained in \cite{BeW18, BeW21} (confirming some conjectures in \cite{BW18a}). The transition matrices between standardized canonical basis and idivided powers were given in \cite{BWW23}. 
In contrast, for the third rank one case, namely the quasi-split case, the transition matrices remain largely unknown.

\subsection{What is achieved in this paper?} We focus on the quasi-split rank one case in this paper, which is a subalgebra of a rank two quantum group. From now on, $\U$ denotes the quantum group $U_q(\mathfrak{sl}_3)$, $\mathbf{U}^\imath$ (and $\dot{\U}^{\imath}$) denotes the (modified) quasi-split rank one iquantum group, and $L(m,n)$ denotes the simple finite-dimensional $\U$-module of highest weight $m\omega_1+n\omega_2$, for $m,n\in \mathbb Z_{\ge 0}$. Unless otherwise specified, statements made at the \emph{algebra level} refer to $\dot{\U}^{\imath}$, while statements made at the \emph{module level} refer to $L(m,n)$. 

The standardized canonical basis of $\dot{\U}^{\imath}$ (and for general higher rank as well) has been constructed in \cite{BWW25}. The monomial bases for $\dot{\U}^{\imath}$ and $L(m,n)$ introduced in \cite{WZ25} are explicit and easy to work with, as they are products of divided powers of the form $B_i^{(a)}B_{\tau i}^{(b)}B_i^{(c)}$, for $b\ge a+c$.

The goal of this paper is to determine explicit formulas for the icanonical bases for $\dot{\U}^{\imath}$ and for the modules $L(m,n)$. More explicitly, we shall establish various transition matrices among different bases:
\begin{itemize}
    \item \textbf{On the algebra $\dot{\U}^{\imath}$:} the icanonical basis, the monomial basis, and the standardized canonical basis;
    \item \textbf{On the modules $\bf L(m,n)$:} the icanonical basis, the monomial basis, and Lusztig's canonical basis.
\end{itemize}
 
Explicit transition matrices between the monomial basis and Lusztig’s canonical basis on the module level have already been obtained  in \cite{WZ25}. Consequently, the above problem on module level is largely reduced to determining the transition matrices between the monomial basis and the icanonical basis; the corresponding transition matrices on the algebra level will follow from this.

Many arguments in higher rank reduce to local rank-one calculations (see, e.g., \cite{BW18b, WZ25}). Consequently, a thorough understanding of the rank-one cases provides crucial insight and a solid foundation for the general theory. The bases and their transition matrices studied in this paper can be naturally categorified in the setting of \cite{BWW25}; compare \cite{Lau10, BWW23}. 

\subsection{Main results}
First, we obtain the transition matrices between the standardized canonical basis and the monomial basis on $\dot{\U}^{\imath}$. Note that the standardized canonical basis is not integral \cite{Wan25, BWW25}.

\begin{alphatheorem}[Theorem \ref{thm:stdCB1}]
Fix $\zeta\in X^{\imath}$, for any $a,b,c\geq 0$ such that $b\geq a+c$, we have the following identity on $\dot{\U}^{\imath}1_{\zeta}$:
\begin{align*}
\Delta_{i}^{(a,b,c)}=\sum_{0\leq x\leq a}\sum_{0\leq y\leq c}&(-1)^{x+y}\frac{q^{x(\varsigma_{\tau i}+(-1)^{\tau i}\zeta-2b+3c+a-1)+y(\varsigma_{i}+(-1)^{i}\zeta-2c+b-1)-xy+\binom{x}{2}+\binom{y}{2}}}{\prod\limits_{i=1}^{x}(1-q^{-2i})\prod\limits_{i=1}^{y}(1-q^{-2i})}\nonumber\\
&\times B_{\tau i}^{(a-x)}B_{i}^{(b-x-y)}B_{\tau i}^{(c-y)}1_{\zeta},
\end{align*}
and
\begin{align*}
B_{\tau i}^{(a)}B_{i}^{(b)}B_{\tau i}^{(c)}1_{\zeta}=\sum_{0\leq x \leq a}\sum_{0\leq y \leq c}&\frac{q^{x(\varsigma_{\tau i}+(-1)^{\tau i}\zeta-2b+3c+a-1)+y(\varsigma_{i}+(-1)^{i}\zeta-2c+b-1)-xy-\binom{x}{2}-\binom{y}{2}}}{\prod\limits_{i=1}^{x}(1-q^{-2i})\prod\limits_{i=1}^{y}(1-q^{-2i})}\nonumber\\
&\times \Delta_{i}^{(a-x,b-x-y,c-y)}.
\end{align*}
\end{alphatheorem}

Next we are going to present the transition matrices between icanonical basis and monomial basis on both simple $\U$-module $L(m,n)$ and modified iquantum group level. Recall that $X^\imath\cong\mathbb Z$ under the identification in Section \ref{subsec:iwt}, and that $\varsigma_i$ corresponds to the parameter $\bm{\varsigma}$ defined in Section \ref{subsec:iqg}. Fix an arbitrary $\zeta\in X^{\imath}$, we set 
\[
\alpha^{b,c}_{\zeta, \bm{\varsigma},i}:=(-1)^{i}\zeta+b-2c+\varsigma_{i},
\qquad
\beta^{a,b,c}_{\zeta, \bm{\varsigma},i}:=(-1)^{\tau i}\zeta+a-2b+3c+\varsigma_{\tau i}. 
\]
The form of the transition matrices between the icanonical basis and the monomial basis depends on the signs of $\alpha^{b,c}_{\zeta, \bm{\varsigma},i}$ and $\beta^{a,b,c}_{\zeta, \bm{\varsigma},i}$. Our formulas are obtained case-by-case depending on the signs:
\begin{itemize}
    \item $\alpha^{b,c}_{\zeta, \bm{\varsigma},i}\leq 0$, and $\beta^{a,b,c}_{\zeta, \bm{\varsigma},i}\leq 0$;
    \item exactly one of $\alpha^{b,c}_{\zeta, \bm{\varsigma},i}$ or $\beta^{a,b,c}_{\zeta, \bm{\varsigma},i}$ is positive.
\end{itemize}
The transition matrices between $\mathfrak{B}_{i,\zeta}^{(a,b,c)}$ (icanonical basis) and $B^{(a)}_{\tau i}B^{(b)}_{i}B^{(c)}_{\tau i}1_{\zeta}$ (monomial basis) on $\dot{\U}^{\imath}1_{\zeta}$ are summarized as follows.

\begin{alphaproposition}[Proposition \ref{prop32} and Theorem \ref{thm711}]
\qquad
\begin{enumerate}
\item 
If $\alpha^{b,c}_{\zeta, \bm{\varsigma},i}\leq 0$, $\beta^{a,b,c}_{\zeta, \bm{\varsigma},i}\leq 0$, and $b\geq a+c$, then $B^{(a)}_{\tau i}B^{(b)}_{i}B^{(c)}_{\tau i}1_{\zeta}=\mathfrak{B}^{(a,b,c)}_{i,\zeta}$. 
\item
If $(-1)^{i}(m-n)+b-2c+\varsigma_{i}\leq 0$, $(-1)^{\tau i}(m-n)+a-2b+3c+\varsigma_{\tau i}\leq 0$, and $b\geq a+c$, then $\mathfrak{B}^{(a,b,c)}_{i,\zeta}\eta=B^{(a)}_{\tau i}B^{(b)}_{i}B^{(c)}_{\tau i}\eta$ on $L(m,n)$, where $m,n\in \mathbb{Z}_{\geq 0}$ and $m-n=\zeta$.
\end{enumerate}
\end{alphaproposition}
Actually, it follows from \eqref{mb} that $\mathfrak{B}^{(a,b,c)}_{i,\zeta}\eta\not=0$ if and only if $c\le m_{\tau i}$ and $a \le b-c \le m_{i}$, where we set $m_{1}=m$, $m_{2}=n$. 

\begin{alphatheorem}[Theorem \ref{thm722}]\label{thmm122}
When $\alpha^{b,c}_{\zeta, \bm{\varsigma},i}>0$ and $b\geq a+c$, we get
\begin{align*}
B^{(a)}_{\tau i}B^{(b)}_{i}B^{(c)}_{\tau i}1_{\zeta}&=[\alpha^{b,c}_{\zeta, \bm{\varsigma},i}]\sum_{\substack{k \ge 0 \\ k\equiv \alpha^{b,c}_{\zeta, \bm{\varsigma},i} \pmod 2}}\frac{[\alpha^{b,c}_{\zeta, \bm{\varsigma},i}-k+2]\cdots [\alpha^{b,c}_{\zeta, \bm{\varsigma},i}+k-4][\alpha^{b,c}_{\zeta, \bm{\varsigma},i}+k-2]}{[k]!}\mathfrak{B}_{i,\zeta}^{a,b-k,c-k}\\
&\quad\quad+\sum_{\substack{k\geq 0 \\ k\equiv \alpha^{b,c}_{\zeta, \bm{\varsigma},i}+1 \pmod 2}}\frac{[\alpha^{b,c}_{\zeta, \bm{\varsigma},i}-k+1][\alpha^{b,c}_{\zeta, \bm{\varsigma},i}-k+3]\cdots [\alpha^{b,c}_{\zeta, \bm{\varsigma},i}+k-1]}{[k]!}\mathfrak{B}^{a,b-k,c-k}_{i,\zeta}.
\end{align*}
Furthermore, if $\alpha^{b,c}_{\zeta, \bm{\varsigma},i}$ is even, then
\begin{align*}
\mathfrak{B}_{i,\zeta}^{(a,b,c)}&=\sum_{0\leq k\leq c}(-1)^{k}\frac{[\alpha^{b,c}_{\zeta, \bm{\varsigma},i}][\alpha^{b,c}_{\zeta, \bm{\varsigma},i}+2]\cdots [\alpha^{b,c}_{\zeta, \bm{\varsigma},i}+2k-2]}{[k]!}B^{(a)}_{\tau i}B^{(b-k)}_{i}B^{(c-k)}_{\tau i}1_{\zeta}.
\end{align*} 
If $\alpha^{b,c}_{\zeta, \bm{\varsigma},i}$ is odd, then
\begin{align*}
\mathfrak{B}_{i,\zeta}^{(a,b,c)}&=\sum_{0\leq k\leq c}(-1)^{k}[\alpha^{b,c}_{\zeta, \bm{\varsigma},i}+k-1]\frac{[\alpha^{b,c}_{\zeta, \bm{\varsigma},i}+1][\alpha^{b,c}_{\zeta, \bm{\varsigma},i}+3]\cdots [\alpha^{b,c}_{\zeta, \bm{\varsigma},i}+2k-3]}{[k]!}B^{(a)}_{\tau i}B^{(b-k)}_{i}B^{(c-k)}_{\tau i}1_{\zeta}.
\end{align*} 
\end{alphatheorem}

In fact, Theorem \ref{thmm122} is derived from Theorems \ref{aeqb+c}, \ref{thm5144}, \ref{thm6155}, and \ref{thmm61010}. In these theorems, the explicit transition matrices between the icanonical basis and the monomial basis on the module level are computed. We first outline the main ideas of the proofs of these results and defer the detailed arguments to the corresponding sections.

We begin with a conjectural transition matrix from the icanonical basis to Lusztig’s canonical basis. To verify this conjecture, we use the explicit expansion of the monomial basis in terms of Lusztig’s canonical basis obtained in \cite{WZ25}. This yields the transition matrix from the conjectured icanonical basis to Lusztig’s canonical basis, which we then show to be uni-triangular with off-diagonal entries in $q^{-1}\mathbb{Z}[q^{-1}]$. Consequently, by the defining properties of the icanonical basis as in \cite[Theorem~5.7]{BW18b}, all these conjectures are indeed true statements.

By taking the inverse limit of the icanonical bases along the projective system $\{L(m,n)\}_{m-n=\zeta}$ (see \cite[Section 6]{BW18b}), we obtain Theorem~\ref{thmm122}. Using this result, we further derive Theorem~ \ref{thmstability}, which shows that, when the simple $\U$-module $L(m,n)$ is viewed as a $\dot{\U}^{\imath}$-module, the action of the icanonical basis of $\dot{\U}^{\imath}$ on the highest weight vector $\eta$ yields either an icanonical basis element of $L(m,n)$ or zero. Hence, in our rank one setting this provides an alternative proof of \cite[Corollary~4.6.3]{Wat23} (which was conjectured in \cite[Remark 6.18]{BW18b}).

In parallel, under another assumption when $\beta^{a,b,c}_{\zeta, \bm{\varsigma},i}=(-1)^{\tau i}\zeta+a-2b+3c+\varsigma_{\tau i}>0$,
analogous results to Theorem \ref{thmm122} hold. The detailed statements are given in Theorems \ref{thmm512}, \ref{thm51616}, \ref{thm688}, \ref{thm61212}, \ref{thm744}.

Taken together, these results provide a complete and explicit description of the transition matrices among the icanonical, monomial, standardized canonical, and Lusztig's canonical bases.\\

\subsection{Organization}
The organization of this paper is summarized in the diagram below. (CB is a shorthand for Canonical Basis.)

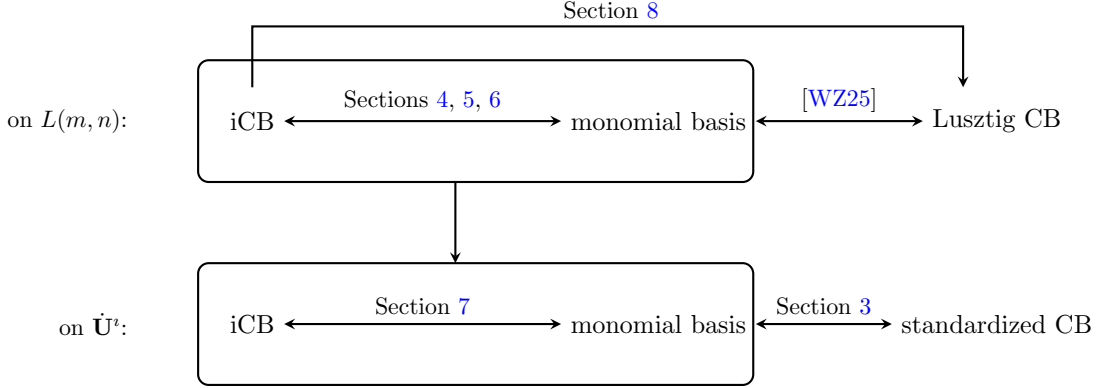
\begin{tikzpicture}[>=stealth, thick, scale=0.9, transform shape]
\begin{scope}[xshift=1.5cm] 

% ---------------- Top row: module level ----------------
\node (iCBm) at (0,4) {iCB};
\node (MBm)  at (6,4) {monomial basis};
\node (LCB)  at (11,4) {Lusztig CB};

% Larger box (module level)
\draw[rounded corners] (-0.8,3.1) rectangle (7.4,4.9);

% Arrows (module level)
\draw[<->] (iCBm) -- node[above] {\small Sections \ref{section4}, \ref{section5}, \ref{section6}} (MBm);
\draw[<->] (MBm) -- node[above] {\small \cite{WZ25}} (LCB);

% -------- Section 8 bent arrow --------
\draw[->] (0,4.5) -- (0,5.4) -- (10.5,5.4) -- (10.5,4.5);
\node at (5.3,5.65) {\small Section \ref{section8}};

% Module-level label
\node[left] at (-1.7,4) {\small on $L(m,n)$:};

% ---------------- Bottom row: algebra level ----------------
\node (iCBa) at (0,1) {iCB};
\node (MBa)  at (6,1) {monomial basis};
\node (SCB)  at (11,1) {standardized CB};

% Larger box (algebra level)
\draw[rounded corners] (-0.8,0.1) rectangle (7.4,1.9);

% Arrows (algebra level)
\draw[<->] (iCBa) -- node[above] {\small Section \ref{section7}} (MBa);
\draw[<->] (MBa) -- node[above] {\small Section \ref{section3}} (SCB);

% Algebra-level label
\node[left] at (-1.7,1) {\small on $\dot{\U}^{\imath}$:};

% ---------------- Vertical arrow ----------------
\draw[->] (3,3.1) -- node[right] {} (3,1.9);
\end{scope}
\end{tikzpicture}

More precisely, we begin in Section \ref{section3} by determining the transition matrices between the monomial basis and the standardized canonical basis. We then study the transition matrices between the icanonical basis and the monomial basis at the module level, which is carried out in Sections \ref{section4}–\ref{section6}. 

Using the results obtained in Sections \ref{section4}–\ref{section6}, we subsequently relate the icanonical basis to the monomial basis at the algebra level in Section \ref{section7}. The vertical arrow in the diagram indicates that the algebra-level results are obtained by generalizing the corresponding module-level constructions via an appropriate asymptotic limit.

In addition, Section \ref{section8} establishes a direct connection between the icanonical basis and Lusztig’s canonical basis. We derive a new explicit formula relating the two bases, going beyond a mere combination of the results obtained in Sections \ref{section4}–\ref{section6} and those of \cite{WZ25}.
\\

\noindent{\bf Acknowledgments:} 
The author thanks his advisor Weiqiang Wang for valuable guidance and insightful discussions. This research was supported in part by the Graduate Research Fellowship from Wang's NSF grant DMS-2401351. The author also appreciates the hospitality of the Department of Mathematics and IMS, National University of Singapore, where part of this work was carried out.
\\

\noindent{\bf Statement on generative AI use:}
No generative AI tools were used to develop or verify the mathematical content of the original submission. During revision, the author used GPT-6 Astra (OpenAI) to assist with the exposition and verification of mathematical arguments. The author reviewed and edited the output as needed and takes full responsibility for the content of the published article.
\\

\noindent{\bf Note on the revised version:}
The previous version used equation (5.15) of \cite[Proposition 5.9]{WZ25}, which contains a sign error in the first arXiv version: the term $-2y$ inside the factor multiplied by $-x/2$ should be $+2y$. This correction yields uniform transition formulas between the monomial and icanonical bases, incorporating the previously separate boundary cases $b=a+c$ and removing the additional boundary correction terms. The transition formulas from the icanonical basis to the canonical basis have also been expanded to cover all cases considered in this paper.

%%%%%%%%%%%

\section{Preliminaries}
In this section, we review basic on the quantum group $\U_q(\mathfrak{sl}_3)$ and the quasi-split rank one iquantum group $\U_q^{\imath}(\mathfrak{sl}_3)$. The purpose of this section is to establish our conventions and fix the notation that will be used throughout the paper.

\subsection{Weight lattice and iweight lattice}\label{subsec:iwt}
Denote $\mathbb{I}:=\{1,2\}$ and a bilinear form $\cdot :\mathbb{I}\times \mathbb{I}\to \mathbb{Z}$ such that $i\cdot i=2$, $\frac{2i \cdot j}{i\cdot i}=-1+3\delta_{i,j}$ for all $i,j\in \mathbb{I}$. We call $(\mathbb{I}, \cdot)$ a Cartan datum (cf. \cite{Lus90}). Furthermore, let $(X,Y, \langle\cdot,\cdot\rangle)$ be a root datum of type $(\mathbb{I}, \cdot)$ (cf. \cite{Lus90}), that is, we have 
\begin{itemize}
	\item two finitely generated abelian groups $Y$, $X$, and a perfect bilinear pairing $\langle\cdot ,\cdot\rangle: Y\times X\to \mathbb{Z}$;
	\item embeddings $\mathbb{I}\to X$ ($i\to \alpha_{i}$) and $\mathbb{I}\to Y$ ($i\to h_{i}$) such that $\langle h_{i},\alpha_{j}\rangle=\frac{2i \cdot j}{i\cdot i}$ for all $i,j\in \mathbb{I}$.
\end{itemize}  
Here $X=\{m\omega_{1}+n\omega_{2} \mid m,n\in \mathbb{Z}\}$ is the weight lattice, where $\omega_i$ are fundamental weights.
%and are related to $\alpha_{1}$, $\alpha_{2}$ by $\omega_1= \frac{2\alpha_1 + \alpha_2}{3}$, $\omega_2 = \frac{\alpha_1 + 2\alpha_2}{3}$.

Let $\tau: \mathbb{I}\to \mathbb{I}$ such that $\tau: i\to j$ for $i\not=j\in \mathbb{I}$ and $\tau(i)\cdot \tau(j)=i\cdot j$ for $i,j \in \mathbb{I}$. For any $\lambda=m\omega_{1}+n\omega_{2}\in X$, $\tau(\lambda)=n\omega_{1}+m\omega_{2}$. Following \cite{BW18b} we introduce the iweight lattice
\begin{align*}
X^{\imath}& :=X\big/\{\lambda+\tau(\lambda): \lambda\in X\}
=X\big/\{a(\omega_{1}+\omega_{2}): a\in \mathbb{Z}\}
\\
& \cong \mathbb{Z},
\end{align*}
where the last isomorphism is given by $\overline{m\omega_{1}+n\omega_{2}}\mapsto m-n$.  Set
\[
Y_{\imath}:=\{h\in Y: \tau(h)=-h\}. 
\]
Given $\zeta\in X$, we denote its image in the projection as $\hat{\zeta}\in X^{\imath}$ and call it iweight. When the context is clear, we slightly abuse of notation and simply write $\zeta\in X^{\imath}$. We have a well-defined bilinear paring $\langle\cdot ,\cdot\rangle: Y_{\imath}\times X^{\imath}\to \mathbb{Z}$ defined by $\langle \omega, \hat{\zeta}\rangle:=\langle \omega, \zeta\rangle$ for any $\omega\in Y_{\imath}$, $\zeta\in X$.

\subsection{Quantum groups and iquantum groups}\label{subsec:iqg}
Let $\mathbb{Q}(q)$ be the field of rational functions of an indeterminate $q$. Let $\mathcal{A}:= \mathbb{Z}[q,q^{-1}]$, for any $a\in \mathbb{Z}$, $b, t\in \mathbb{Z}_{>0}$, we define 
$$
[a]_{q^{t}}:=\frac{q^{ta}-q^{-ta}}{q^{t}-q^{-t}} ,\quad\quad [b]_{q^{t}}!:=[1]_{q^{t}}[2]_{q^{t}}\cdots [b]_{q^{t}} ,\quad\quad \qbinom{a}{b}_{q^{t}}:=\frac{[a]_{q^{t}}[a-1]_{q^{t}}\cdots [a-b+1]_{q^{t}}}{[b]_{q^{t}}!}.
$$
Here we have $\qbinom{a}{b}_{q^t}=0$ when $a\in \mathbb Z_{\ge 0}$, $b<0$ or $b>a$ and $[0]_{q^{t}}!:=1$. In addition, when $t=1$, we omit the subscript $q$ from the notation.

Recall the quantum group $\boldmath{\U}=\boldmath{\U}_{q}(\mathfrak{sl}_{3})$ as the $\mathbb{Q}(q)$-algebra generated by $E_{i}, F_{i}, K_{\mu}$ for $i\in \mathbb{I}=\{1,2\}$, $\mu\in \mathbb{Z}\mathbb{I}$, subject to the following relations:
$$
K_{0}=1,\quad \quad K_{\mu}K_{\nu}=K_{\mu+\nu}\quad\quad \text{for $\mu, \nu \in \mathbb{Z}\mathbb{I}$},
$$
$$
K_{i}E_{j}K^{-1}_{i}=q^{\frac{2i \cdot j}{i\cdot i}}E_{j},\quad\quad K_{i}F_{j}K^{-1}_{i}=q^{-\frac{2i \cdot j}{i\cdot i}}F_{j},
$$
$$
E_{i}F_{j}-F_{j}E_{i}=\delta_{i,j}\frac{K_{i}-K_{i}^{-1}}{q-q^{-1}}\quad\quad \text{for $i, j \in \mathbb{I}$},
$$
$$
E_{j}E_{i}^{(2)}-E_{i}E_{j}E_{i}+E^{(2)}_{i}E_{j}=0,\quad\quad F_{j}F_{i}^{(2)}-F_{i}F_{j}F_{i}+F^{(2)}_{i}F_{j}=0\quad\quad (i\not=j),
$$
where $E_{i}^{(r)}= \frac{E^{r}_{i}}{[r]!}$, $F_{i}^{(r)}= \frac{F^{r}_{i}}{[r]!}$ for any $r\geq 0$, are called divided powers. We set $E^{(r)}_{i} = F^{(r)}_{i} =0$ for $r<0$. Denote by $\U^{+}$ (and respectively, $\U^{-}$) the subalgebra of $\U$ generated by $E_{i}\, (i\in \mathbb{I})$ (and respectively, $F_{i}\, (i\in \mathbb{I})$). Let $\psi$ denote the bar involution on $\boldmath{\U}$. More precisely, $\psi$ is an anti-linear $\mathbb{Q}$-algebra automorphism on $\U$ determined by
$$
\psi: q\to q^{-1},\quad E_{i}\to E_{i},\quad F_{i}\to F_{i},\quad K_{i}\to K_{-i}.
$$\textbf{}

%Denote by $\dot{\U}$ the modified quantum group (cf. \cite[Chapter 23]{Lus94}), which is the $\mathbb{Q}(q)$-algebra generated by idempotent $1_{\lambda}$, $E_{i}1_{\lambda}$, $F_{i}1_{\lambda}$ and $K_{i}1_{\lambda}$, $\lambda\in X$, subject to the relations
%$$
%1_{\lambda}1_{\mu}=\delta_{\lambda,\mu}1_{\lambda},\quad E_{i}1_{\lambda}=1_{\lambda+\alpha_{i}}E_{i},\quad F_{i}1_{\lambda}=1_{\lambda-\alpha_{i}}F_{i},\quad K_{i}1_{\lambda}=q^{\langle h_{i}, \lambda\rangle}1_{\lambda}=1_{\lambda}K_{i},
%$$
%and $x1_{\lambda}-y1_{\lambda}=(x-y)1_{\lambda}$ for any $x,y\in \U$. 
%The quantum group $\U$ admits a Hopf algebra structure with comultiplication:
%$$
%\Delta(E_{i})=1\otimes E_{i}+E_{i}\otimes K^{-1}_{i},\quad \Delta(F_{i})=F_{i}\otimes 1+K_{i}\otimes F_{i}, \quad %\Delta(K_{i})=K_{i}\otimes K_{i}.
%$$

%The quantum group $\U$ (and $\U^{\pm}$) admits an $X$-graded algebra structure with gradation $\deg E_{i}=\alpha_{i}$, $\deg F_{i}=-\alpha_{i}$, $\deg K_{i}=0$ and admits the following decomposition:
%$$
%\U=\bigoplus_{\alpha\in X}\U_{\alpha},\quad \U^{\pm}=\bigoplus_{\alpha\in X}\U^{\pm}_{\alpha}.
%$$

%Furthermore, for each $i\in \mathbb{I}$, let ${}_{i}R$ and $R_{i}$ be the unique $\mathbb{Q}(q)$-linear endomorphisms of $\U^{-}$ such that (cf. 1.2.13) ${}_{i}R(1)=R_{i}(1)=0$, ${}_{i}R(\theta_{j})=R_{i}(\theta_{j})=\delta_{i,j}/(1-q^{-2})$ and
%$$
%{}_{i}R(xy)={}_{i}R(x)y+q^{\langle h_{i}, \alpha\rangle}x{}_{i}R(y),\quad R_{i}(xy)=q^{\langle h_{i}, \beta\rangle}R_{i}(x)y+xR_{i}(y)
%$$
%for any $x\in \U^{-}_{\alpha}$, $y\in \U^{-}_{\beta}$. 

Denote by $\Ui =\U^{\imath}_{\bm{\varsigma}}$ the quasi-split iquantum group associated to $(\mathbb{I},\tau)$ with parameter $\bm{\varsigma}=(\varsigma_{1},\varsigma_{2}) \in \mathbb Z^2$. By definition $\Ui$ is the $\mathbb Q(q)$-subalgebra of $\U$ generated by 
$$
B_{i}=F_{i}+q^{\varsigma_{i}}E_{\tau i}K^{-1}_{i},\quad\quad k_{i}=K_{i}K^{-1}_{\tau i}\quad\quad \forall i\in \mathbb{I}.
$$
In this paper, we always assume $(\varsigma_{1},\varsigma_{2})=(1,0)$ or $(0,1)$ for the study of icanonical basis. Furthermore, let $\psi^{\imath}$ denote the bar involution on $\boldmath{\U}^{\imath}$ (cf. \cite{BW18a}). More precisely, $\psi^\imath$ is an anti-linear $\mathbb{Q}$-algebra automorphism on $\boldmath{\U}^{\imath}$ such that
$$
\psi^\imath: q\to q^{-1},\quad B_{i}\to B_{i},\quad k_{i}\to k^{-1}_{i}.
$$

Next, we recall the modified iquantum group $\dot{\U^{\imath}}$, see \cite{BW18b, BW21}. The $\dot{\U}^{\imath}$ is the $\mathbb{Q}(q)$-algebra generated by $1_{\lambda}$, and $x1_{\lambda}$, where $x\in \U^{\imath}$, $\lambda\in X^{\imath}$; Moreover, $\dot{\U}^{\imath}$ is a left and a right $\Ui$-module such that
\begin{align}
1_{\lambda}1_{\mu}=\delta_{\lambda,\mu}1_{\lambda},\quad B_{i}1_{\lambda}=1_{\lambda-\alpha_{i}}B_{i},\quad k_{i}1_{\lambda}=q^{\langle h_{i}-h_{\tau i},\lambda\rangle}1_{\lambda}=1_{\lambda}k_{i},
%\quad x1_{\lambda}-y1_{\lambda}=(x-y)1_{\lambda}
\end{align}
for any $i\in \mathbb I$, $\lambda,\mu \in X^{\imath}$.  Notice that $\dot{\U^{\imath}}$ admits the following decomposition:
$$
\dot{\U^{\imath}}:=\bigoplus_{\lambda, \zeta\in X^{\imath}}1_{\lambda} \U^{\imath}1_{\zeta}=\bigoplus_{\lambda\in X^{\imath}}1_{\lambda} \U^{\imath}=\bigoplus_{\zeta\in X^{\imath}} \U^{\imath}1_{\zeta}.
$$

\subsection{Bilinear forms} 
First, we introduce 3 different bilinear forms related to the usual quantum group $\U$. From \cite[Chapter 1-3]{Lus94}, denote $r: \U^{-}\to \U^{-}\otimes \U^{-}$ as the $\mathbb{Q}(q)$-algebra homomorphism. We can define the symmetric bilinear form on $\U^{-}$ as follows:
$$
(1,1)=1,\quad (F_{i},F_{j})=\frac{\delta_{i,j}}{1-q^{-2}},\quad (x,yy')=(r(x),y\otimes y'),
$$
where the bilinear form $(\U^{-}\otimes \U^{-})\times (\U^{-}\otimes \U^{-})\to \mathbb{Q}(q)$ is given by $(x\otimes x',y\otimes y')=(x,x')(y,y')$ for any $x,x',y,y'\in \U^{-}$. We further know that the bilinear form is non-degenerate. We can extend the definition above to a symmetric bilinear form on $\dot{\U}$: 
$$
(1_{\lambda}x1_{\mu}, 1_{\lambda'}x'1_{\mu'})=0\quad \text{unless $\lambda=\lambda'$ and $\mu=\mu'$},\quad $$
$$
(ux,x')=(x,\rho(u)x'),\quad (y1_{\mu}, y'1_{\mu})=(y,y').
$$
for $\forall x,x'\in \dot{\U}, y,y'\in \U^{-}, \lambda, \lambda', \mu, \mu'\in X, u\in \U$. 

Furthermore, for $m,n\in \mathbb Z_{\ge 0}$, denote $L(m,n)$ as the (unique) simple $\U$-module with highest weight $\lambda=m\omega_{1}+n\omega_{2}$. Denote the highest weight vector as $\eta$. From \cite[Chapter 19]{Lus94}, there exists a unique  symmetric bilinear form on $L(m,n)$ such that
$$
(\eta,\eta)_{m,n}=1,\quad (ux,y)_{m,n}=(x,\rho(u)y)_{m,n}\quad \forall x,y\in L(m,n), u\in \U,
$$
and if $x\in L(m,n)_{\mu}$, $y\in L(m,n)_{\nu}$, $\mu\not=\nu$, then $(x,y)_{\lambda}=0$. From \cite{Lus94}, we also know
\begin{itemize}
    \item both $(~,~)$ on $\U^{-}$ and $(~,~)_{m,n}$ on $L(m,n)$ are non-degenerate;
    \item $\lim\limits_{m,n\to \infty}(x\eta,y\eta)_{m,n}=(x,y)$ for any $x,y\in U^{-}$.
\end{itemize}

Next, we introduce the $\imath$-analogue of the above bilinear forms. Now fix $\zeta\in X^{\imath}$, we define a bilinear form on $\U^{\imath}1_{\zeta}$ by (cf. \cite{BW18b})
$$
(u,v)_{\zeta}^{\imath}:= \lim_{\substack{m,n\to \infty \\ m-n=\zeta}}(u\eta,v\eta)_{m,n}\quad \forall u,v\in \U^{\imath}1_{\zeta}.
$$

\subsection{Canonical basis, icanonical basis, monomial basis, and standardized canonical basis}\label{section24}
Now we introduce the main objects that we are going to study. 

We start from the canonical basis on $\U^{-}$. From \cite[Chapter 14]{Lus94}, the elements in the set 
$$
\mathcal{B}:=\left\{F^{(a)}_{2}F^{(b)}_{1}F^{(c)}_{2}, F^{(a)}_{1}F^{(b)}_{2}F^{(c)}_{1}: a,b,c\in \mathbb Z_{\ge 0}\quad\text{and}\quad  b\geq a+c\right\},
$$
modulo the identification $F^{(a)}_{2}F^{(b)}_{1}F^{(c)}_{2}=F^{(c)}_{1}F^{(b)}_{2}F^{(a)}_{1}$ if $b=a+c$, forms a basis of $\U^{-}$, called the canonical basis of $\U^{-}$. Furthermore, denote $L(m,n)$ as the simple $\U_{q}(\mathfrak{sl}_{3})$ module with the highest weight $\lambda=m\omega_{1}+n\omega_{2}$ and highest weight vector $\eta$, for $m,n\geq 0$. Now the elements $\mathcal{B}\eta-\{0\}$ forms a basis of $L(m,n)$, called the canonical basis of $L(m,n)$. 

From \cite[Theorem 5.10]{BW18b}, for each $a,b,c\in \mathbb Z_{\ge 0}$ such that $b\geq a+c$ and $F^{(a)}_{\tau i}F^{(b)}_{i}F^{(c)}_{\tau i}\eta\not=0$, there exists a unique $\mathfrak{B}^{(a,b,c)}_{i}\eta\in L(m,n)$ such that $\psi^{\imath}(\mathfrak{B}^{(a,b,c)}_{i}\eta)=\mathfrak{B}^{(a,b,c)}_{i}\eta$ and
\begin{align} \label{icb1}
\mathfrak{B}^{(a,b,c)}_{i}\eta=F^{(a)}_{\tau i}F^{(b)}_{i}F^{(c)}_{\tau i}\eta+\sum_{\substack{x,y\geq 0 \\ (x,y)\not=(0,0)}}q^{-1}\mathbb{Z}[q^{-1}]F^{(a-x)}_{\tau i}F^{(b-x-y)}_{i}F^{(c-y)}_{\tau i}\eta.
\end{align}
Furthermore, the above elements, modulo the identification $\mathfrak{B}^{(a,a+c,c)}_{2}\eta=\mathfrak{B}^{(c,a+c,a)}_{1}\eta$, together form a basis of $L(m,n)$, called the icanonical basis of $L(m,n)$. When the context is clear, $\mathfrak{B}^{(a,b,c)}_{i}\eta$ can be abbreviated as $\mathfrak{B}_{i}^{(a,b,c)}$.

From \cite[Theorem 6.17]{BW18b}, for any $\zeta\in X^{\imath}$, there exists a basis on $\dot{\U}^{\imath}1_{\zeta}$: $\{\mathfrak{B}_{2,\zeta}^{(a,b,c)}, \mathfrak{B}_{1,\zeta}^{(a,b,c)}: b\geq a+c\}$ modulo the identification $\mathfrak{B}_{2,\zeta}^{(a,a+c,c)}=\mathfrak{B}_{1,\zeta}^{(c,a+c,a)}$, obtained by taking the inverse limit of the icanonical bases $\mathfrak{B}^{(a,b,c)}_{i}\eta$, along the projective system $\{L(m,n)\}_{m-n=\zeta}$. From \cite{BW18b}, we further know that $\mathfrak{B}_{i,\zeta}^{(a,b,c)}$ is $\psi^{\imath}$ invariant. We call this basis the icanonical basis on $\dot{\U}^{\imath}1_{\zeta}$. When the context is clear, $\mathfrak{B}_{i,\zeta}^{(a,b,c)}$ can be abbreviated as $\mathfrak{B}_{i}^{(a,b,c)}$.

Next, we introduce the monomial basis. It follows from the following fact:
\begin{thm}[\text{cf. \cite[Theorems 5.11--5.12]{WZ25}}]
For any $i\in I$, $b\ge a+c$, $B_{i}^{(a)}B_{\tau i}^{(b)}B_{i}^{(c)} \eta \neq 0 $ if and only if $F_{i}^{(a)}F_{\tau i}^{(b)}F_{i}^{(c)} \eta \neq 0$, if and only if $c\leq m_i$ and $a\leq b-c\leq m_{\tau i}$. Furthermore, the set 
\begin{align}\label{mb}
\{B_1^{(a)}B_2^{(b)}B_1^{(c)}\eta: c\le m_1, a \le b-c \le m_2\}\cup \{B_2^{(c)} B_1^{(b)} B_2^{(a)} \eta: a\le m_2, c\le b-a\le m_1\}
\end{align} 
forms a basis for $L(m_1,m_2)$, modulo the identification $B_1^{(a)}B_2^{(b)}B_1^{(c)}\eta=B_2^{(c)}B_1^{(b)}B_2^{(a)}\eta$ when $b=a+c$.
\end{thm}
\noindent Here $B^{(r)}_{i}:=\frac{B_{i}^{r}}{[r]!}$ for $r\geq 0$, called the idivided powers. (We also have the convention that $B^{(r)}_{i}=0$ when $r<0$.) The set in \eqref{mb} is called the monomial basis of $L(m,n)$. For any $\zeta\in X^{\imath}$, one can extend the above result to the modified iquantum group $\dot{\U}^{\imath}1_{\zeta}$: $\{B_1^{(a)}B_2^{(b)}B_1^{(c)}1_{\zeta}, B_2^{(c)} B_1^{(b)} B_2^{(a)} 1_{\zeta}: b\geq a+c\}$, modulo the identification $B_1^{(a)}B_2^{(b)}B_1^{(c)}1_{\zeta}=B_2^{(c)} B_1^{(b)} B_2^{(a)} 1_{\zeta}$ when $b=a+c$, forms a basis for $\dot{\U}^{\imath}1_{\zeta}$, called the monomial basis of $\dot{\U}^{\imath}1_{\zeta}$.

Finally, we introduce the standardized canonical basis. We need the following theorem first:
\begin{thm}[\text{cf. \cite[Theorem 5.12]{BWW25}}]\label{bww25}
Fix $\zeta\in X^{\imath}$, there exists a unique $\mathbb{Q}(q)$-linear isomorphism of vector space $\jmath: \dot{\U}^{\imath}1_{\zeta}\to \U^{-}$ such that $\forall x\in \Ui$ and $y\in \U^{-}$:
\begin{align}\label{bww25thm5120}
\lim_{\substack{m,n\to \infty \\ m-n=\zeta}}(x\eta,y\eta)_{m,n}=(\jmath(x1_{\zeta}),y).
\end{align}
Furthermore, $\jmath$ is an isometry, and $\jmath(1_{\zeta})=1$. %and satisfies the following recursive relation:
%\begin{align}\label{bww25thm512}
%\jmath(B_{i}x1_{\zeta})=F_{i}\jmath(x1_{\lambda})+q^{1+\varsigma_{i}-\kappa_{i}}{}_{\tau i}R(\jmath(x1_{\zeta})),
%\end{align}
%for any homogeneous $x\in \Ui$ in the sense that $x1_{\zeta}=1_{\kappa}x$ and $\kappa_{i}:=(h_{i}-h_{\tau i})(\hat{\lambda})$.
\end{thm}
Fix $\zeta\in X^{\imath}$, when $b\geq a+c$, set $\Delta_{i}^{(a,b,c)}:=\jmath^{-1}(F^{(a)}_{\tau i}F^{(b)}_{i}F^{(c)}_{\tau i})$. From Theorem \ref{bww25}, we know that the set $\{\Delta_{2}^{(a,b,c)}, \Delta_{1}^{(a,b,c)}:b\geq a+c\}$, modulo the identification $\Delta_{2}^{(a,b,c)}=\Delta_{1}^{(c,b,a)}$ when $b=a+c$, forms a basis of $\dot{\U}^{\imath}1_{\zeta}$, called standardized canonical basis of $\dot{\U}^{\imath}1_{\zeta}$.  

From \eqref{bww25thm5120}, we further get
\begin{align*}
\lim_{\substack{m,n\to \infty \\ m-n=\zeta}}(\Delta_{i}^{(a,b,c)}\eta,y\eta)_{m,n}=\lim_{m,n\to \infty}(F^{(a)}_{\tau i}F^{(b)}_{i}F^{(c)}_{\tau i}\eta,y\eta)_{m,n}=(F^{(a)}_{\tau i}F^{(b)}_{i}F^{(c)}_{\tau i},y).
\end{align*}
for any $y\in \U^{-}$. Hence, 
$$
\lim_{\substack{m,n\to \infty \\ m-n=\zeta}}(\Delta_{i}^{(a,b,c)}\eta-F^{(a)}_{\tau i}F^{(b)}_{i}F^{(c)}_{\tau i}\eta,y\eta)_{m,n}=0\quad\text{for any $y\in \U^{-}$}.
$$
Writing $\Delta_{i}^{(a,b,c)}$ in terms of monomial basis on $\dot{\U}^{\imath}1_{\zeta}$, and $F^{(a)}_{\tau i}F^{(b)}_{i}F^{(c)}_{\tau i}\eta$ in terms of monomial basis on $L(m,n)$, the above result shows that $\Delta_{i}^{(a,b,c)}$ is the unique element in $\dot{\U}^{\imath}1_{\zeta}$ satisfying 
\begin{align}\label{defscb}
\lim_{\substack{m,n\to \infty \\ m-n=\zeta}}(\Delta_{i}^{(a,b,c)}\eta-F^{(a)}_{\tau i}F^{(b)}_{i}F^{(c)}_{\tau i}\eta)=0.
\end{align}
Here the limit is understood in the sense that, as $m,n\to \infty$, $ m-n=\zeta$, all coefficients of the monomial basis expansion on $L(m,n)$ vanish.

\section{Standardized canonical basis vs monomial basis}\label{section3}
In this section, we determine the explicit transition matrices between the standardized canonical basis and the monomial basis at the level of the modified iquantum group.

We write $L(m,n)$ as the simple $\U$-module with the highest weight $\lambda=m\omega_{1}+n\omega_{2}$ and highest weight vector $\eta$, where $m,n\geq 0$. For the remainder of this section, we fix an arbitrary $\zeta\in X^{\imath}$, and work under the convention that the parameters satisfy $m-n=\zeta$. Before presenting the main theorem, we require the following lemma.

\begin{lem} \cite[Propositions 5.9-5.10]{WZ25}
For any $a,c,r\geq 0$, we have the following relations on $L(m,n)$:
\begin{align}\label{lem311}
F_{1}^{(a)}F_{2}^{(a+c+r)}F_{1}^{(c)}\eta
&=\sum_{x=0}^{a} \sum_{y=0}^{c} (-1)^{x+y}
   q^{x\varsigma_{1}+y\varsigma_{2}+\frac{x(-2m-3-2r+2c-2y+x)}{2}+\frac{y(-2n-3+2r+2a+y)}{2}}\nonumber\\
&\quad\quad\quad \times\qbinom{n-r-a+x}{x} \qbinom{m-c+y}{y}B_{1}^{(a-x)} B_{2}^{(a+c+r-x-y)} B_{1}^{(c-y)}\eta
\end{align}
\begin{align}\label{lem312}
B_{1}^{(a)}B_{2}^{(a+c+r)}B_{1}^{(c)}\eta
&=\sum_{x=0}^{a} \sum_{y=0}^{c}
   q^{x\varsigma_{1}+y\varsigma_{2}+\frac{x(-2m-1-2r+2c-2y-x)}{2}+\frac{y(-2n-1+2r+2a-y)}{2}}\nonumber\\
&\quad\quad\quad \times\qbinom{n-r-a+x}{x} \qbinom{m-c+y}{y}F_{1}^{(a-x)} F_{2}^{(a+c+r-x-y)} F_{1}^{(c-y)}\eta
\end{align}
\end{lem}

Having established the necessary preliminaries, we now present the main theorems of this section.
\begin{thm}\label{thm:stdCB1}
Fix $\zeta\in X^{\imath}$, for any $a,b,c\geq 0$ such that $b\geq a+c$, we have the following identity on $\dot{\U}^{\imath}1_{\zeta}$:
\begin{align}\label{scbvsm1}
\Delta_{i}^{(a,b,c)}=\sum_{0\leq x\leq a}\sum_{0\leq y\leq c}&(-1)^{x+y}\frac{q^{x(\varsigma_{\tau i}+(-1)^{\tau i}\zeta-2b+3c+a-1)+y(\varsigma_{i}+(-1)^{i}\zeta-2c+b-1)-xy+\binom{x}{2}+\binom{y}{2}}}{\prod\limits_{i=1}^{x}(1-q^{-2i})\prod\limits_{i=1}^{y}(1-q^{-2i})}\nonumber\\
&\times B_{\tau i}^{(a-x)}B_{i}^{(b-x-y)}B_{\tau i}^{(c-y)}1_{\zeta},
\end{align}
and
\begin{align}\label{scbvsm2}
B_{\tau i}^{(a)}B_{i}^{(b)}B_{\tau i}^{(c)}1_{\zeta}=\sum_{0\leq x \leq a}\sum_{0\leq y \leq c}&\frac{q^{x(\varsigma_{\tau i}+(-1)^{\tau i}\zeta-2b+3c+a-1)+y(\varsigma_{i}+(-1)^{i}\zeta-2c+b-1)-xy-\binom{x}{2}-\binom{y}{2}}}{\prod\limits_{i=1}^{x}(1-q^{-2i})\prod\limits_{i=1}^{y}(1-q^{-2i})}\nonumber\\
&\times \Delta_{i}^{(a-x,b-x-y,c-y)}.
\end{align}
\end{thm}

\begin{proof}
We focus on the case $i=2$ and the case $i=1$ will follow by symmetry. We first prove \eqref{scbvsm1}. We take the limit $m,n\to \infty$ with $m-n=\zeta$ to the coefficients of $B_{1}^{(a-x)}B_{2}^{(b-x-y)}B_{1}^{(c-y)}\eta$ appearing in the RHS of \eqref{lem311}, and write $r=b-a-c\geq 0$. We have 
\begin{align*}
\lim_{\substack{m,n\to \infty \\ m-n=\zeta}}&\sum_{x=0}^{a}\sum_{y=0}^{c}(-1)^{x+y}\frac{q^{x(\varsigma_{1}-m+a-b+2c)+y(\varsigma_{2}-n+b-c)+\frac{1}{2}(x^2+y^2)-\frac{3}{2}(x+y)-xy}}{\prod\limits_{i=1}^{x} (1-q^{-2i}) \prod\limits_{i=1}^{y} (1-q^{-2i})}\nonumber
\\
&\quad\quad\quad \times\frac{(q-q^{-1})^{x}}{q^{\frac{x(x+1)}{2}}}\frac{[n-r-a+x]!}{[n-r-a]!}\frac{(q-q^{-1})^{y}}{q^{\frac{y(y+1)}{2}}}\frac{[m-c+y]!}{[m-c]!}\nonumber
\\
&=\sum_{x=0}^{a}\sum_{y=0}^{c}(-1)^{x+y}\frac{q^{x(\varsigma_{1}-\zeta+a-2b+3c-1)+y(\varsigma_{2}+\zeta+b-2c-1)+\frac{1}{2}(x^2+y^2)-\frac{1}{2}(x+y)-xy}}{\prod\limits_{i=1}^{x} (1-q^{-2i}) \prod\limits_{i=1}^{y} (1-q^{-2i})}
\\
&\quad\times \lim_{\substack{m,n\to \infty \\ m-n=\zeta}} \prod_{i=1}^x(1-q^{-2(n-r-a)-2i})  \prod_{i=1}^y(1-q^{-2(m-c)-2i})
\\
&=\sum_{x=0}^{a}\sum_{y=0}^{c}(-1)^{x+y}\frac{q^{x(\varsigma_{1}-\zeta+a-2b+3c-1)+y(\varsigma_{2}+\zeta+b-2c-1)+\frac{1}{2}(x^2+y^2)-\frac{1}{2}(x+y)-xy}}{\prod\limits_{i=1}^{x} (1-q^{-2i}) \prod\limits_{i=1}^{y} (1-q^{-2i})}.
\end{align*}
Hence, \eqref{scbvsm1} follows from \eqref{defscb}. The proof of \eqref{scbvsm2} now by taking the limit of \eqref{lem312} is similar and therefore omitted.
\end{proof}

\section{iCanonical basis vs monomial basis: first observation}
\label{section4}
The main goal of the next three sections is to determine explicit transition matrices between the monomial basis and the icanonical basis at the module level. In this section, we begin with some general observations and then treat several special cases, which will prove useful in the analysis of the general situation in the subsequent sections.

\subsection{The first observation}
We write $L(m_{1},m_{2})$ as the simple $\U$-module with the highest weight $m_{1}\omega_{1}+m_{2}\omega_{2}$ and highest weight vector $\eta_{m_{1},m_{2}}$, for $m_{1},m_{2}\geq 0$. Recall \eqref{icb1}, for any simple $\U$-module $L(m_{1},m_{2})$, $m_{1}, m_{2}\geq 0$, the set $\{\mathfrak{B}_{2}^{(a,b,c)}, \mathfrak{B}_{1}^{(a,b,c)}: b\geq a+c\}$ modulo the relation $\mathfrak{B}_{2}^{(a,a+c,c)}=\mathfrak{B}_{1}^{(c,c+a,a)}$ is the icanonical basis on $L(m_{1},m_{2})$. Setting $\bm{\varsigma}= (\varsigma_{1}, \varsigma_{2})$, $\bm{m}=(m_{1},m_{2})$, we define
\begin{align}\label{c2}
c^{a,b,c}_{\tau i,\bm{\varsigma},\bm{m},x,y}:=q^{x\varsigma_{i}+y\varsigma_{\tau i}-\frac{y(y+1+2m_{\tau i}+2c-2b)}{2}-\frac{x(x+1+2m_{i}+2b-4c-2a+2y)}{2}}\qbinom{m_{\tau i}-b+c+x}{x}\qbinom{m_{i}-c+y}{y}.
\end{align}
This section will rely heavily on the following statement. 

\begin{lem}[\text{cf. \cite[Proposition 5.9]{WZ25}}] \label{prop59}
Given any $a,b,c,m_{1},m_{2}\geq 0$ and $b\geq a+c$, we have the following identity on $L(m_{1},m_{2})$:
\begin{align*}
B_{i}^{(a)}B_{\tau i}^{(b)}B_{i}^{(c)}\eta_{m_{1},m_{2}}&=\sum_{x=0}^{a}\sum_{y=0}^{c}c^{a,b,c}_{\tau i,\bm{\varsigma},\bm{m},x,y}F_{i}^{(a-x)}F_{\tau i}^{(b-x-y)}F_{i}^{(c-y)}\eta_{m_{1},m_{2}}.
\end{align*}
\end{lem}

Defining $\deg f$ for any $f\in \mathcal{A}$ as the highest power of $q$ with a nonzero coefficient, we have the following observations from Lemma \ref{prop59}.
\begin{prop} \label{prop32}
When $(\varsigma_{1},\varsigma_{2})=(1,0)$ or $(0,1)$, $\bm{m}=(m_{1},m_{2})$, $m_i\geq c$, $m_{\tau i}\geq b-c$ and $b\geq a+c$:
\begin{enumerate}
    \item For any $0\leq x\leq a$, $0\leq y\leq c$
    \begin{align*}
    \deg c^{a,b,c}_{2,\bm{\varsigma},\bm{m},x,y}&=-\frac{(x+y)^{2}}{2}+(y-x)(m_{1}-m_{2}+b-2c-\frac{1}{2}+\varsigma_{2})-x(b-a-c)\\
    &=-\frac{(x+y)^{2}}{2}+(x-y)(m_{2}-m_{1}+a-2b+3c-\frac{1}{2}+\varsigma_{1})-y(b-a-c).
    \end{align*}
    \item $c^{a,b,c}_{\tau i,\bm{\varsigma},\bm{m},x,y}=c^{a,b,c}_{i,\bm{\varsigma}',\bm{m}',x,y}$ for $\bm{\varsigma}'=1-\bm{\varsigma}$, $\bm{m}'=(m'_{1},m'_{2})$ such that $m'_{i}=m_{\tau i}$.
    \item When $i=1$, $a=0$, $m_{1}-m_{2}+b-2c+\varsigma_{2}\leq 0$, $B_{2}^{(b)}B_{1}^{(c)}\eta_{m_{1},m_{2}}$ belongs to the icanonical basis of $L(m_{1},m_{2})$.
    \item When $i=1$, $c=0$, $m_{2}-m_{1}+a-2b+\varsigma_{1}\leq 0$, $B_{1}^{(a)}B_{2}^{(b)}\eta_{m_{1},m_{2}}$ belongs to the icanonical basis of $L(m_{1},m_{2})$.
    \item When $m_{i}-m_{\tau i}+b-2c+\varsigma_{\tau i}\leq 0$, $m_{\tau i}-m_{i}+a-2b+3c+\varsigma_{i}\leq 0$, $B_{i}^{(a)}B_{\tau i}^{(b)}B_{i}^{(c)}\eta_{m_{1},m_{2}}$ belongs to the icanonical basis of $L(m_{1},m_{2})$.
\end{enumerate}
\end{prop}
\begin{proof}
First, we compute 
\begin{align*}
\deg c^{a,b,c}_{2,\bm{\varsigma},\bm{m},x,y}&=(m_{2}-b+c)x+(m_{1}-c)y+x\varsigma_{1}+y\varsigma_{2}\\
&\quad\quad -\frac{y(y+1+2m_{2}+2c-2b)}{2}-\frac{x(x+1+2m_{1}+2b-4c-2a+2y)}{2}\\
&=-\frac{(x+y)^{2}}{2}+(y-x)(m_{1}-m_{2}+b-2c-\frac{1}{2}+\varsigma_{2})-x(b-a-c).
\end{align*}
Furthermore, notice that 
\begin{align}\label{degc}
\deg c^{a,b,c}_{2,\bm{\varsigma},\bm{m},x,y}&=-\frac{(x+y)^{2}}{2}+y(m_{1}-m_{2}+b-2c-\frac{1}{2}+\varsigma_{2})+x(m_{2}-m_{1}+a-2b+3c-\frac{1}{2}+\varsigma_{1}),
\end{align}
the rest of the statements follow from \eqref{degc} as well as the definition of icanonical basis in \eqref{icb1}.
\end{proof}

\begin{rem}\label{iandtaui}
From now on, we mainly focus on the case $i=1$, as later we will see that all the statements for $i=2$ follows by an entirely analogous argument and will not be repeated. For simplicity reason, we denote $m_{1}=m$, $m_{2}=n$.
\end{rem}

Thanks to Proposition \ref{prop32}, we only need to focus on the case $m-n+b-2c+\varsigma_{2}>0$ or $n-m+a-2b+3c+\varsigma_{1}>0$ in later sections. Our first observation is that when $n-m+a-2b+3c+\varsigma_{1}>0$
$$
m-n+b-2c+\varsigma_{2}=-(n-m+a-2b+3c+\varsigma_{1})+1+a-b+c\leq 0.
$$
When $m-n+b-2c+\varsigma_{2}>0$
$$
n-m+a-2b+3c+\varsigma_{1}=-(m-n+b-2c+\varsigma_{2})+1+a-b+c\leq 0.
$$
Hence, we only need to focus on the case $m-n+b-2c+\varsigma_{2}>0$ XOR $n-m+a-2b+3c+\varsigma_{1}>0$. 

Given $\bm{\varsigma}=(\varsigma_{1},\varsigma_{2})$, we define 
\begin{align}\label{defofa}
\alpha_{b,c,\bm{\varsigma}}:=m-n+b-2c+\varsigma_{2},
\end{align}
\begin{align}\label{defofbb}
\beta_{a,b,c,\bm{\varsigma}}:=n-m+a-2b+3c+\varsigma_{1}. 
\end{align}
Once we restrict to a particular choice of $\bm{\varsigma}$, $\alpha_{b,c,\bm{\varsigma}}$ and $\beta_{a,b,c,\bm{\varsigma}}$ can be abbreviated as $\alpha_{b,c}$, $\beta_{a,b,c}$ respectively. In summary, we will deal with the following cases in the following sections:
\begin{enumerate}
    \item $\alpha_{b,c,\bm{\varsigma}}>0$ or $\beta_{a,b,c,\bm{\varsigma}}>0$,
    \item $\bm{\varsigma}=(0,1)$ or $\bm{\varsigma}=(1,0)$,
    \item when $\alpha_{b,c,\bm{\varsigma}}>0$ (or $\beta_{a,b,c,\bm{\varsigma}}>0$), the parity of $\alpha_{b,c,\bm{\varsigma}}$ (or $\beta_{a,b,c,\bm{\varsigma}}$).
\end{enumerate}

\subsection{Case for $a=0$, $(\varsigma_{1},\varsigma_{2})=(1,0)$, $\alpha_{b,c,\bm{\varsigma}}$ is even} 

Before moving to the case for general $a,b,c$ in next section, we focus on a special case in this subsection: $a=0$, $(\varsigma_{1},\varsigma_{2})=(1,0)$ and $\alpha_{b,c}$ is an even positive integer. Recall that $L(m,n)$ is the simple $\U$-module with highest weight vector $\eta$ and highest weight $m\omega_{1}+n\omega_{2}$, for $m,n\geq 0$. Further recall \eqref{icb1}, for each $b,c\geq 0$, $\mathfrak{B}^{(0,b,c)}_{2}$ is the icanonical basis element in $L(m,n)$ satisfying
\begin{align} \label{defomega}
\mathfrak{B}^{(0,b,c)}_{2}=F^{(b)}_{2}F^{(c)}_{1}\eta+\sum_{\theta> 0}q^{-1}\mathbb{Z}[q^{-1}]F^{(b-\theta)}_{2}F^{(c-\theta)}_{1}\eta.
\end{align}

We want to find the transition matrices between $B^{(b)}_{2}B^{(c)}_{1}\eta$ and $\mathfrak{B}^{(0,b,c)}_{2}$ on $L(m,n)$. %From \eqref{degc}
%\begin{align}
%c^{0,b,c}_{2,\bm{\varsigma},\bm{m},0,y}&=-\frac{y^{2}}{2}+y(m_{1}-m_{2}+b-2c-\frac{1}{2}+\varsigma_{2}).
%\end{align}
%Hence, recall \eqref{defofa}, \eqref{defofbb}, the value of $\beta_{0,b,c,\bm{\varsigma}}$ does not affect the result and we only need to focus on the case $\alpha_{b,c,\bm{\varsigma}}>0$. We further restrict ourselves to the case when $\alpha_{b,c,\bm{\varsigma}}$ is even in this subsection. 
This subsection relies heavily on the following lemma:
\begin{lem} [\text{cf. \cite[Lemma 5.8]{WZ25}}] 
When $(\varsigma_{1},\varsigma_{2})=(1,0)$, we have
\begin{align}\label{wz25}
B^{(b)}_{2}B^{(c)}_{1}\eta=\sum_{y\geq 0}q^{-\frac{y(y+1+2n+2c-2b)}{2}}\qbinom{m-c+y}{y}F^{(b-y)}_{2}F^{(c-y)}_{1}\eta.
\end{align}
\end{lem}

Recall \eqref{defofa},
\begin{align} \label{defalpha}
\alpha_{b,c}=m-n+b-2c,
\end{align}
we define 
\begin{align}\label{deffff}
\Omega^{(b,c)}_{m,n,\theta}&=\sum_{k+y=\theta}(-1)^{k}\frac{[\alpha_{b,c}][\alpha_{b,c}+2]\cdots [\alpha_{b,c}+2k-2]}{[k]!}q^{-\frac{y(y+1+2n+2c-2b)}{2}}\qbinom{m-c+\theta}{y}\in \mathcal{A}
\end{align}
for any $0\leq \theta\leq c$. When $\theta=c$, we also allow the abbreviation $\Omega^{(b,c)}_{m,n}:=\Omega^{(b,c)}_{m,n,c}$. We need a few preparatory lemmas before moving to the main theorem:

\begin{lem} \label{lemrecur}
When $(\varsigma_{1},\varsigma_{2})=(1,0)$, for any $\theta,m,n,b,c\geq 0$ such that $0\leq \theta\leq c$, $m\geq c$, $n-b+c\geq 0$, $\alpha_{b,c}> 0$, and $\alpha_{b,c}$ is even, we have
\begin{align}\label{recur1}
\Omega^{(b,c)}_{m-1,n+1,\theta}=\Omega^{(b,c+1)}_{m,n,\theta}.
\end{align}
\end{lem}
\begin{proof}
The proof can be found in Appendix \ref{applemrecur}.
\end{proof}

\begin{lem} \label{lemrecur2}
Set $(\varsigma_{1},\varsigma_{2})=(1,0)$, and $m,n,b,c\geq 0$ such that $m\geq c$, $n-b+c\geq 0$, $\alpha_{b,c}>0$, and $\alpha_{b,c}$ is even. Given $\Omega^{(b,0)}_{m,n}=1$ and $\Omega^{(b,1)}_{m,n}=q^{-m}[n-b+2]$, when $c\geq 2$, we have
\begin{align}\label{lem34}
\Omega^{(b,c)}_{m,n}&=\frac{q^{-m-1+c}}{[c]}[n-b+c+1]\Omega^{(b,c-1)}_{m,n+2}+\frac{q^{-m+n-b+c+1}}{[c]}[m-n+b-2c]\Omega^{(b,c-2)}_{m,n+2}.
\end{align}
\end{lem}
\begin{proof}
The proof can be found in Appendix \ref{app4}.
\end{proof}

We have the following observation based on \eqref{lem34}, which turns out to be important in later sections.
\begin{lem}\label{lem36}
When $(\varsigma_{1},\varsigma_{2})=(1,0)$, for any $m,n,b\geq 0$ and $c\geq 1$ such that $m\geq c$, $n-b+c\geq 0$, $\alpha_{b,c}>0$, and $\alpha_{b,c}$ is even, we have $c+\deg(\Omega^{(b,c)}_{m,n})<0$. 
\end{lem}
\begin{proof}
From \eqref{lem34}, when $c\geq 2$, we check
$$
\deg\left(\frac{q^{-m-1+c}}{[c]}[n-b+c+1]\Omega^{(b,c-1)}_{m,n+2}\right)+c=-(m-n+b-2c)+\deg(\Omega^{(b,c-1)}_{m,n+2})<0,
$$
$$
\deg\left(\frac{q^{-m+n-b+c+1}}{[c]}[m-n+b-2c]\Omega^{(b,c-2)}_{m,n+2}\right)+c=\deg(\Omega^{(b,c-2)}_{m,n+2})-c+1<0,
$$
and thus $c+\deg(\Omega^{(b,c)}_{m,n})<0$. When $c=1$, $\Omega^{(b,1)}_{m,n}=q^{-m}[n-b+2]$ and $1+\deg(\Omega^{(b,1)}_{m,n})=-(m-n+b-2)<0$.
\end{proof}

\begin{rem}
One can show that Lemma \ref{lem36} also holds when $(\varsigma_{1},\varsigma_{2})=(0,1)$.
\end{rem}

\begin{thm}\label{thm:iCBtoMB}
When $(\varsigma_{1},\varsigma_{2})=(1,0)$, for any $m,n,b,c\geq 0$ such that $m\geq c$, $n-b+c\geq 0$, $\alpha_{b,c}>0$, and $\alpha_{b,c}$ is even, we have
\begin{align} \label{eq:Bbc}
\mathfrak{B}^{(0,b,c)}_{2}&=\sum_{0\leq k\leq c}(-1)^{k}\frac{[\alpha_{b,c}][\alpha_{b,c}+2]\cdots [\alpha_{b,c}+2k-2]}{[k]!}B^{(b-k)}_{2}B^{(c-k)}_{1}\eta.
\end{align}
\end{thm}
\begin{proof}
Substituting \eqref{wz25} into the RHS of \eqref{eq:Bbc}, we get 
\begin{align} \label{Omega}
\text{RHS of \eqref{eq:Bbc}}=F^{(b)}_{2}F^{(c)}_{1}\eta+\sum_{0<\theta\leq c}\Omega^{(b,c)}_{m,n,\theta}F^{(b-\theta)}_{2}F^{(c-\theta)}_{1}\eta,
\end{align} 
where $\Omega^{(b,c)}_{m,n,\theta}$ was defined in \eqref{deffff}. Since $\Omega^{(b,c)}_{m,n,\theta}\in \mathcal{A}$, it suffices to show that $\Omega^{(b,c)}_{m,n,\theta}\in q^{-1}\mathbb{Z}[q^{-1}]$. Once this is established, by the defining properties and uniqueness of the icanonical basis, \eqref{eq:Bbc} follows immediately. Using Lemma \ref{lemrecur} above, we only need to prove $\Omega^{(b,c)}_{m,n}\in q^{-1}\mathbb{Z}[q^{-1}]$, for any $c\geq 1$, and this follows from Lemma \ref{lem36}. So we are done.
\end{proof}

\subsection{Case for $a=0$, $(\varsigma_{1},\varsigma_{2})=(1,0)$, $\alpha_{b,c,\bm{\varsigma}}$ is odd}
We focus on a special case in this subsection: $a=0$, $(\varsigma_{1},\varsigma_{2})=(1,0)$ and $\alpha_{b,c,\bm{\varsigma}}$ is odd. We want to find the transition matrices between $B^{(b)}_{2}B^{(c)}_{1}\eta$ and $\mathfrak{B}^{(0,b,c)}_{2}$ on $L(m,n)$. Recall that $L(m,n)$ is the simple $\U$-module with highest weight vector $\eta$ and highest weight $m\omega_{1}+n\omega_{2}$, for $m,n\geq 0$. Again, we have the same notation $\alpha_{b,c}=m-n+b-2c$ as \eqref{defalpha} and $\mathfrak{B}^{(0,b,c)}_{2}$ as element in icanonical basis of $L(m,n)$.

We define 
\begin{align}\label{deffffo}
O^{(b,c)}_{m,n,\theta}&=\sum_{k+y=\theta}(-1)^{k}[\alpha_{b,c}+k-1]\frac{[\alpha_{b,c}+1][\alpha_{b,c}+3]\cdots [\alpha_{b,c}+2k-3]}{[k]!}q^{-\frac{y(y+1+2n+2c-2b)}{2}}\qbinom{m-c+\theta}{y}
\end{align}
for any $1\leq \theta\leq c$ and $O^{(b,c)}_{m,n,0}=1$. Throughout this paper, whenever an odd-parity coefficient is written
in the form
\[
\frac{(-1)^k[\alpha+k-1]}{[k]!}
\prod_{j=0}^{k-2}[\alpha+1+2j],
\]
the entire expression is interpreted as $1$ when $k=0$. The same convention applies when $\alpha_{b,c}$ is replaced by any other parameter. When $\theta=c$, we also allow the abbreviation $O^{(b,c)}_{m,n}:=O^{(b,c)}_{m,n,c}$. The following lemma unravel the integrality of $O^{(b,c)}_{m,n,\theta}$:
\begin{lem} \label{qmn}
For any odd positive integer $m$ and positive integer $n$, we have 
$\frac{[m+n-1]}{[n]!} \prod_{j=0}^{n-2}[m+1+2j] \in\mathcal{A}$.
\end{lem}
\begin{proof}
The proof can be found in Appendix \ref{app1}.
\end{proof}

We still need a few preparatory lemmas before moving to the main theorem:
\begin{lem}\label{lemoddddd}
When $(\varsigma_{1},\varsigma_{2})=(1,0)$, for any $\theta,m,n,b,c\geq 0$ such that $0\leq \theta\leq c$, $m\geq c$, $n-b+c\geq 0$, $\alpha_{b,c}>0$, and $\alpha_{b,c}$ is odd, we have
\begin{align}\label{recur2}
O^{(b,c)}_{m-1,n+1,\theta}=O^{(b,c+1)}_{m,n,\theta}.
\end{align}
\end{lem}
\begin{proof}
The proof is identical to the proof of Lemma \ref{lemrecur}.
\end{proof}
\begin{lem} \label{lemrecur3}
Set $(\varsigma_{1},\varsigma_{2})=(1,0)$, and $m,n,b,c\geq 0$ such that $m\geq c$, $n-b+c\geq 0$, $\alpha_{b,c}>0$, and $\alpha_{b,c}$ is odd. Given $O^{(b,0)}_{m,n}=1$, $O^{(b,1)}_{m,n}=q^{-m}[n-b+2]$, and $O^{(b,2)}_{m,n}=\frac{q^{-m+1}}{[2]}[n-b+3]O^{(b,1)}_{m,n+2}+\frac{q^{-m+n-b+4}}{[2]}[m-n+b-3]$, when $c\geq 3$, we have
\begin{align*}
O^{(b,c)}_{m,n}&=\frac{q^{-m-1+c}}{[c]}[n-b+c+1]O^{(b,c-1)}_{m,n+2}\\
&\qquad+\frac{q^{-m+n-b+2c}}{[c]}[m-n+b-2c+1](\Omega^{(b,c-2)}_{m,n+3}-q^{-m+n-b+c+1}\Omega^{(b,c-3)}_{m,n+3}).
\end{align*}
\end{lem}
\begin{proof}
The proof can be found in Appendix \ref{app5}.
\end{proof}

\begin{thm}\label{thm:iCBtoMBodd}
When $(\varsigma_{1},\varsigma_{2})=(1,0)$, for any $m,n,b,c\geq 0$ such that $m\geq c$, $n-b+c\geq 0$, $\alpha_{b,c}>0$, and $\alpha_{b,c}$ is odd, we have
\begin{align} \label{eq:oddBbc}
\mathfrak{B}^{(0,b,c)}_{2}=\sum_{0\leq k\leq c}(-1)^{k}[\alpha_{b,c}+k-1]\frac{[\alpha_{b,c}+1][\alpha_{b,c}+3]\cdots[\alpha_{b,c}+2k-3]}{[k]!}B^{(b-k)}_{2}B^{(c-k)}_{1}\eta.
\end{align}
\end{thm}
\begin{proof}
The proof uses the same strategy as Theorem \ref{thm:iCBtoMB}. Substituting \eqref{wz25} into the RHS of \eqref{eq:oddBbc}, we get
\begin{align} \label{O}
\text{RHS of \eqref{eq:oddBbc}}=F^{(b)}_{2}F^{(c)}_{1}\eta+\sum_{0<\theta\leq c}O^{(b,c)}_{m,n,\theta}F^{(b-\theta)}_{2}F^{(c-\theta)}_{1}\eta,
\end{align} 
where $O^{(b,c)}_{m,n,\theta}$ was defined in \eqref{deffffo}. Since $O^{(b,c)}_{m,n,\theta}\in \mathcal{A}$, and $O^{(b,c)}_{m,n,0}=1$, it suffices to show that $O^{(b,c)}_{m,n,\theta}\in q^{-1}\mathbb{Z}[q^{-1}]$ when $c\geq 1$ and $\theta\geq 1$. Once this is established, by the defining properties and uniqueness of the icanonical basis, \eqref{eq:oddBbc} follows immediately. Using Lemma \ref{lemoddddd} above, we only need to prove $O^{(b,c)}_{m,n}\in q^{-1}\mathbb{Z}[q^{-1}]$, for any $c\geq 1$. 

From Lemma \ref{lemrecur3}, we found (recall the condition that $m-n+b-2c>0$) $q^{-m-1+c}[n-b+c+1]\in q^{-1}\mathbb{Z}[q^{-1}]$ when $c\geq 1$ and $\frac{q^{-m+n-b+2c}}{[c]}[m-n+b-2c+1]\in q^{-1}\mathbb{Z}[[q^{-1}]]$ when $c\geq 2$. Hence, $O^{(b,c)}_{m,n}\in q^{-1}\mathbb{Z}[q^{-1}]$ by using induction on $c$. So we are done.
\end{proof}

\subsection{The inverse transition matrix}
In the previous two subsections, we have already found the expression of icanonical basis $\mathfrak{B}^{(0,b,c)}_{2}$ (defined in \eqref{icb1} ) in terms of monomial basis $B^{(b)}_{2}B^{(c)}_{1}\eta$ when $a=0$, $(\varsigma_{1},\varsigma_{2})=(1,0)$, $\alpha_{b,c}=m-n+b-2c>0$. In this subsection, we want to prove the inverse formula. We need the following $q$-binomial identity first:
\begin{lem}\label{lem1}
For any $k,m,c\geq 1$, 
\begin{align*}
\qbinom{k+1}{m+1}&\big([c][k-m]-[m+1][c+k-1]\big)\\
&=\qbinom{k-1}{m+1}[k+1][c-(m+1)]-\qbinom{k-1}{m-1}[k+1][c+2k-(m+1)].
\end{align*}
\end{lem}
\begin{proof}
The proof can be found in Appendix \ref{app2}.
\end{proof}

The inverse matrix is given by:
\begin{prop}\label{inverse}
When $(\varsigma_{1},\varsigma_{2})=(1,0)$, for any $m,n,b,c\geq 0$ such that $m\geq c$, $n-b+c\geq 0$, and $\alpha_{b,c}>0$, we have:
\begin{align*}
B^{(b)}_{2}B^{(c)}_{1}\eta=&[\alpha_{b,c}]\sum_{\substack{k \ge 0 \\ k\equiv \alpha_{b,c} \pmod 2}}\frac{[\alpha_{b,c}-k+2][\alpha_{b,c}-k+4]\cdots [\alpha_{b,c}+k-4][\alpha_{b,c}+k-2]}{[k]!}\mathfrak{B}_{2}^{(0,b-k,c-k)}\\
&+\sum_{\substack{k\geq 0 \\ k\equiv \alpha_{b,c}+1 \pmod 2}}\frac{[\alpha_{b,c}-k+1][\alpha_{b,c}-k+3]\cdots [\alpha_{b,c}+k-3][\alpha_{b,c}+k-1]}{[k]!}\mathfrak{B}_{2}^{(0,b-k,c-k)}.
\end{align*}
\end{prop}

\begin{rem}
Throughout this paper, whenever an inverse transition coefficient
is written in the form
\[
\frac{[\gamma]}{[k]!}
\prod_{j=1}^{k-1}[\gamma-k+2j],
\]
the entire expression, including the factor $[\gamma]$,
is interpreted as $1$ when $k=0$. For $k\geq1$, it has its usual meaning, with an empty product equal to $1$. This convention applies to every parameter $\gamma$ occurring in the inverse transition formulas, including the $\alpha$- and $\beta$-parameters and their hatted and algebra-level versions in later chapters. 
\end{rem}

\begin{proof}
The above identity can be understood as the triangular transition matrix from a basis $C:=\{\mathfrak{B}_{2}^{(0,b-k,c-k)}|0\leq k\leq c\}$ to another basis $P:=\{B_{2}^{(b-k)}B_{1}^{(c-k)}\eta |0\leq k\leq c\}$. Meanwhile, the second identity gives a transition matrix from $P$ to $C$. Now the remaining thing is to prove that there two transition matrices are inverses to each other. By Theorem \ref{thm:iCBtoMB} and \ref{thm:iCBtoMBodd}, we can replace all $\mathfrak{B}_{2}^{(0,b-k,c-k)}$ appearing in the RHS of the first identity by the summation in terms of elements in $P$. Now, it remains to prove that the coefficients of each $B_{2}^{(b-k)}B_{1}^{(c-k)}\eta$, $k>0$ is $0$. We divide into four cases.

When $\alpha_{b,c}$ is even, $k$ is even, the coefficient of $B_{2}^{(b-k)}B_{1}^{(c-k)}\eta$, $k>0$ is given by
\begin{align*}
&\frac{[\alpha_{b,c}][\alpha_{b,c}+2]\cdots [\alpha_{b,c}+2k-2]}{[k]!}
\\
&\quad\quad +[\alpha_{b,c}]\sum_{\substack{1\leq m\leq k 
\\ 
\text{$m$ is even}}}\frac{[\alpha_{b,c}-m+2][\alpha_{b,c}-m+4]\cdots [\alpha_{b,c}+m]\cdots [\alpha_{b,c}+2k-m-2]}{[m]![k-m]!}
\\
&\quad\quad -[\alpha_{b,c}+k-1]\sum_{\substack{1\leq m\leq k 
\\ 
\text{$m$ is odd}}}\frac{[\alpha_{b,c}-m+1][\alpha_{b,c}-m+3]\cdots [\alpha_{b,c}+m+1]\cdots [\alpha_{b,c}+2k-m-3]}{[m]![k-m]!}
\\
&\quad =\frac{1}{[k+1]!}\bigg[[\alpha_{b,c}][\alpha_{b,c}+2]\cdots [\alpha_{b,c}+2k-2][k+1]
\\
&\quad\quad\quad\quad +\sum_{\substack{1\leq m\leq k 
\\ 
\text{$m$ is odd}}}\bigg([\alpha_{b,c}-m+1][\alpha_{b,c}-m+3]\cdots [\alpha_{b,c}+m+1]\cdots [\alpha_{b,c}+2k-m-3]
\\
&\quad\quad\quad\quad \quad\quad\times \qbinom{k+1}{m+1}\big([\alpha_{b,c}][k-m]-[m+1][\alpha_{b,c}+k-1]\big)\bigg)\bigg]
\\
&=0.
\end{align*}
For the last step, we used Lemma \ref{lem1}.

When $\alpha_{b,c}$ is even, $k$ is odd, the coefficient of $B_{2}^{(b-k)}B_{1}^{(c-k)}\eta$, for $k>0$, is given by
\begin{align*}
&-\frac{[\alpha_{b,c}][\alpha_{b,c}+2]\cdots [\alpha_{b,c}+2k-2]}{[k]!}\\
&\quad\quad -[\alpha_{b,c}]\sum_{\substack{1\leq m\leq k-1 \\ \text{$m$ is even}}}\frac{[\alpha_{b,c}-m+2][\alpha_{b,c}-m+4]\cdots [\alpha_{b,c}+m]\cdots [\alpha_{b,c}+2k-m-2]}{[m]![k-m]!}\\
&\quad\quad +[\alpha_{b,c}+k-1]\sum_{\substack{1\leq m\leq k-1 \\ \text{$m$ is odd}}}\frac{[\alpha_{b,c}-m+1][\alpha_{b,c}-m+3]\cdots [\alpha_{b,c}+m+1]\cdots [\alpha_{b,c}+2k-m-3]}{[m]![k-m]!}\\
&\quad\quad + \frac{[\alpha_{b,c}-k+1][\alpha_{b,c}-k+3]\cdots [\alpha_{b,c}+k-1]}{[k]!}\\
&\quad =\frac{1}{[k+1]!}\bigg[-[\alpha_{b,c}][\alpha_{b,c}+2]\cdots [\alpha_{b,c}+2k-2][k+1]\\
&\quad\quad\quad\quad +\sum_{\substack{1\leq m\leq k \\ \text{$m$ is odd}}}\bigg([\alpha_{b,c}-m+1][\alpha_{b,c}-m+3]\cdots [\alpha_{b,c}+m+1]\cdots [\alpha_{b,c}+2k-m-3]\\
&\quad\quad\quad\quad \quad\quad\times \qbinom{k+1}{m+1}\big([m+1][\alpha_{b,c}+k-1]-[\alpha_{b,c}][k-m]\big)\bigg)\\
&\quad\quad\quad\quad +[\alpha_{b,c}-k+1][\alpha_{b,c}-k+3]\cdots [\alpha_{b,c}+k-1][k+1]\bigg]=0.
\end{align*}

When $\alpha_{b,c}$ is odd, $k$ is odd, the coefficient of $B_{2}^{(b-k)}B_{1}^{(c-k)}\eta$, $k>0$ is given by
\begin{align*}
&-[\alpha_{b,c}+k-1]\sum_{\substack{0\leq m\leq k \\ \text{$m$ is even}}}\frac{[\alpha_{b,c}-m+1][\alpha_{b,c}-m+3]\cdots [\alpha_{b,c}+m+1]\cdots [\alpha_{b,c}+2k-m-3]}{[m]![k-m]!}\\
&\quad\quad +[\alpha_{b,c}]\sum_{\substack{0\leq m\leq k\\ \text{$m$ is odd}}}\frac{[\alpha_{b,c}-m+2][\alpha_{b,c}-m+4]\cdots [\alpha_{b,c}+m]\cdots [\alpha_{b,c}+2k-m-2]}{[m]![k-m]!}\\
&\quad =\frac{1}{[k+1]!}\bigg[\sum_{\substack{0\leq m\leq k \\ \text{$m$ is even}}}\bigg([\alpha_{b,c}-m+1][\alpha_{b,c}-m+3]\cdots [\alpha_{b,c}+m+1]\cdots [\alpha_{b,c}+2k-m-3]\\
&\quad\quad\quad\quad \quad\quad\times \qbinom{k+1}{m+1}\big([\alpha_{b,c}][k-m]-[m+1][\alpha_{b,c}+k-1]\big)\bigg)\bigg]=0.
\end{align*}

When $\alpha_{b,c}$ is odd, $k$ is even, the coefficient of $B_{2}^{(b-k)}B_{1}^{(c-k)}\eta$, $k>0$ is given by
\begin{align*}
&-[\alpha_{b,c}]\sum_{\substack{0\leq m\leq k \\ \text{$m$ is odd}}}\frac{[\alpha_{b,c}-m+2][\alpha_{b,c}-m+4]\cdots [\alpha_{b,c}+m]\cdots [\alpha_{b,c}+2k-m-2]}{[m]![k-m]!}\\
&\quad\quad +[\alpha_{b,c}+k-1]\sum_{\substack{0\leq m\leq k \\ \text{$m$ is even}}}\frac{[\alpha_{b,c}-m+1][\alpha_{b,c}-m+3]\cdots [\alpha_{b,c}+m+1]\cdots [\alpha_{b,c}+2k-m-3]}{[m]![k-m]!}\\
&\quad =\frac{1}{[k+1]!}\bigg[\sum_{\substack{0\leq m\leq k \\ \text{$m$ is even}}}\bigg([\alpha_{b,c}-m+1][\alpha_{b,c}-m+3]\cdots [\alpha_{b,c}+m+1]\cdots [\alpha_{b,c}+2k-m-3]\\
&\quad\quad\quad\quad \quad\quad\times \qbinom{k+1}{m+1}\big([m+1][\alpha_{b,c}+k-1]-[\alpha_{b,c}][k-m]\big)\bigg)\bigg]=0.
\end{align*}
\end{proof}

\begin{rem}
It is worth noting that the above proof of Proposition \ref{inverse} does not rely on the specific value of $\alpha_{b,c}$.  
\end{rem}

\subsection{Case for $a=0$, $(\varsigma_{1},\varsigma_{2})=(0,1)$} 
We focus on a special case in this subsection: $a=0$, $(\varsigma_{1},\varsigma_{2})=(0,1)$. We want to find the transition matrices between $B^{(b)}_{2}B^{(c)}_{1}\eta$ and $\mathfrak{B}^{(0,b,c)}_{2}$. Recall that $L(m,n)$ is the simple $\U$-module with highest weight vector $\eta$ and highest weight $m\omega_{1}+n\omega_{2}$, for $m,n\geq 0$. Again, we have the same notation $\mathfrak{B}^{(0,b,c)}_{2}$ as element in icanonical basis of $L(m,n)$ and $\alpha'_{b,c}=m-n+b-2c+1$. We need the following lemma first.
\begin{lem} [\text{cf. \cite[Lemma 5.8]{WZ25}}] 
When $(\varsigma_{1},\varsigma_{2})=(0,1)$, taking $b'=b+1$, we have
\begin{align}\label{wz255}
B^{(b)}_{2}B^{(c)}_{1}\eta=\sum_{y\geq 0}q^{-\frac{y(y+1+2n+2c-2b')}{2}}\qbinom{m-c+y}{y}F^{(b-y)}_{2}F^{(c-y)}_{1}\eta.
\end{align}
\end{lem}
Recall when $(\varsigma_{1},\varsigma_{2})=(1,0)$, $\alpha_{b,c}=m-n+b-2c$ defined in \eqref{defalpha}, we have $\alpha'_{b,c}=\alpha_{b',c}$. By further comparing the coefficient of $F^{(b-y)}_{2}F^{(c-y)}_{1}\eta$ in \eqref{wz255} with that in \eqref{wz25}, and taking $b'=b+1$, when $n-b+c\geq1$, this reduces the coefficient calculation to the case $\bm{\varsigma}=(1,0)$ treated above. When $n-b+c=0$, the shifted parameters lie outside the hypotheses of the preceding module statements. Nevertheless, the same finite-sum calculation used to prove \eqref{lem34} gives the corresponding recurrence for $\bm{\varsigma}=(0,1)$, including this boundary case. Its first term vanishes, since $[n-\varsigma_2-b+c+1]=[0]=0$. The same degree induction therefore applies at this boundary. The odd case follows from the odd--even coefficient identity.

Consequently, Theorems \ref{thm:iCBtoMB}, \ref{thm:iCBtoMBodd} and Proposition \ref{inverse} can be generalized to the following:

\begin{thm}\label{thm41717}
When $(\varsigma_{1},\varsigma_{2})=(1,0)$ or $(0,1)$, for any $m,n,b,c\geq 0$ such that $m\geq c$, $n-b+c\geq 0$, and $\alpha_{b,c,\bm{\varsigma}}> 0$, if $\alpha_{b,c,\bm{\varsigma}}$ is even, then we have 
\begin{align} 
\mathfrak{B}^{(0,b,c)}_{2}&=\sum_{0\leq k\leq c}(-1)^{k}\frac{[\alpha_{b,c,\bm{\varsigma}}][\alpha_{b,c,\bm{\varsigma}}+2]\cdots [\alpha_{b,c,\bm{\varsigma}}+2k-2]}{[k]!}B^{(b-k)}_{2}B^{(c-k)}_{1}\eta.
\end{align}
If $\alpha_{b,c,\bm{\varsigma}}$ is odd, then we have
\begin{align} 
\mathfrak{B}^{(0,b,c)}_{2}=\sum_{0\leq k\leq c}(-1)^{k}[\alpha_{b,c,\bm{\varsigma}}+k-1]\frac{[\alpha_{b,c,\bm{\varsigma}}+1][\alpha_{b,c,\bm{\varsigma}}+3]\cdots[\alpha_{b,c,\bm{\varsigma}}+2k-3]}{[k]!}B^{(b-k)}_{2}B^{(c-k)}_{1}\eta.
\end{align}
Furthermore, the inverse formula is given by
\begin{align*}
B^{(b)}_{2}B^{(c)}_{1}\eta=&[\alpha_{b,c,\bm{\varsigma}}]\sum_{\substack{k \ge 0 \\ k\equiv \alpha_{b,c,\bm{\varsigma}} \pmod 2}}\frac{[\alpha_{b,c,\bm{\varsigma}}-k+2][\alpha_{b,c,\bm{\varsigma}}-k+4]\cdots [\alpha_{b,c,\bm{\varsigma}}+k-2]}{[k]!}\mathfrak{B}_{2}^{(0,b-k,c-k)}\\
&+\sum_{\substack{k\geq 0 \\ k\equiv \alpha_{b,c,\bm{\varsigma}}+1 \pmod 2}}\frac{[\alpha_{b,c,\bm{\varsigma}}-k+1][\alpha_{b,c,\bm{\varsigma}}-k+3]\cdots [\alpha_{b,c,\bm{\varsigma}}+k-1]}{[k]!}\mathfrak{B}_{2}^{(0,b-k,c-k)}.
\end{align*}
\end{thm}

\subsection{Case for $c=0$}
We focus on a special case in this subsection: $c=0$. Recall that $L(m,n)$ is the simple $\U$-module with highest weight vector $\eta$ and highest weight $m\omega_{1}+n\omega_{2}$, for $m,n\geq 0$. %From \eqref{degc}
%\begin{align}
%c^{a,b,0}_{2,\bm{\varsigma},\bm{m},x,0}=-\frac{x^{2}}{2}+x(n-m+a-2b-\frac{1}{2}+\varsigma_{1}).
%\end{align}
%Hence, recall \eqref{defofa}, \eqref{defofbb}, the value of $\alpha_{b,c,\bm{\varsigma}}$ does not affect the result and we only need to focus on the case $\beta_{a,b,0,\bm{\varsigma}}=n-m+a-2b+\varsigma_{1}>0$. 
Further recall \eqref{icb1}, for each $a,b\geq 0$, $\mathfrak{B}^{(a,b,0)}_{2}$ is the icanonical basis element in $L(m,n)$ such that 
\begin{align} 
\mathfrak{B}^{(a,b,0)}_{2}=F^{(a)}_{1}F^{(b)}_{2}\eta+\sum_{\theta> 0}q^{-1}\mathbb{Z}[q^{-1}]F^{(a-\theta)}_{1}F^{(b-\theta)}_{2}\eta.
\end{align}
Next, we have the following lemma:
\begin{lem} [\text{cf. \cite[Proposition 5.9]{WZ25}}] 
When $(\varsigma_{1},\varsigma_{2})=(1,0)$, taking $m'=m-1$, we have
\begin{align}\label{bbtoff1}
B^{(a)}_{1}B^{(b)}_{2}\eta=\sum_{x\geq 0}q^{-\frac{x(x+1+2m'+2b-2a)}{2}}\qbinom{n-b+x}{x}F^{(a-x)}_{1}F^{(b-x)}_{2}\eta.
\end{align}
When $(\varsigma_{1},\varsigma_{2})=(0,1)$, we have
\begin{align}\label{bbtoff2}
B^{(a)}_{1}B^{(b)}_{2}\eta=\sum_{x\geq 0}q^{-\frac{x(x+1+2m+2b-2a)}{2}}\qbinom{n-b+x}{x}F^{(a-x)}_{1}F^{(b-x)}_{2}\eta.
\end{align}
\end{lem}
After comparing \eqref{bbtoff1}, \eqref{bbtoff2} with \eqref{wz25}, Theorems \ref{thm:iCBtoMB}, \ref{thm:iCBtoMBodd} and Proposition \ref{inverse} can be extended to the following theorems. ($\beta_{a,b,0,\bm{\varsigma}}=n-m+a-2b+\varsigma_{1}$ from \eqref{defofbb})
\begin{thm}
When $(\varsigma_{1},\varsigma_{2})=(1,0)$ or $(0,1)$, for any $m,n,a,b\geq 0$ such that $n\geq b$, and $\beta_{a,b,0,\bm{\varsigma}}>0$, if $\beta_{a,b,0,\bm{\varsigma}}$ is even, then we have 
\begin{align} 
\mathfrak{B}_{2}^{(a,b,0)}&=\sum_{0\leq k\leq a}(-1)^{k}\frac{[\beta_{a,b,0,\bm{\varsigma}}][\beta_{a,b,0,\bm{\varsigma}}+2]\cdots [\beta_{a,b,0,\bm{\varsigma}}+2k-2]}{[k]!}B^{(a-k)}_{1}B^{(b-k)}_{2}\eta.
\end{align}
If $\beta_{a,b,0,\bm{\varsigma}}$ is odd, then we have
\begin{align} 
\mathfrak{B}_{2}^{(a,b,0)}=\sum_{0\leq k\leq a}(-1)^{k}[\beta_{a,b,0,\bm{\varsigma}}+k-1]\frac{[\beta_{a,b,0,\bm{\varsigma}}+1][\beta_{a,b,0,\bm{\varsigma}}+3]\cdots[\beta_{a,b,0,\bm{\varsigma}}+2k-3]}{[k]!}B^{(a-k)}_{1}B^{(b-k)}_{2}\eta.
\end{align}
Furthermore, the inverse formula is given by
\begin{align*}
B^{(a)}_{1}B^{(b)}_{2}\eta=&[\beta_{a,b,0,\bm{\varsigma}}]\sum_{\substack{k \ge 0 \\ k\equiv \beta_{a,b,0,\bm{\varsigma}} \pmod 2}}\frac{[\beta_{a,b,0,\bm{\varsigma}}-k+2][\beta_{a,b,0,\bm{\varsigma}}-k+4]\cdots [\beta_{a,b,0,\bm{\varsigma}}+k-2]}{[k]!}\mathfrak{B}_{2}^{(a-k,b-k,0)}\\
&+\sum_{\substack{k\geq 0 \\ k\equiv \beta_{a,b,0,\bm{\varsigma}}+1 \pmod 2}}\frac{[\beta_{a,b,0,\bm{\varsigma}}-k+1][\beta_{a,b,0,\bm{\varsigma}}-k+3]\cdots [\beta_{a,b,0,\bm{\varsigma}}+k-1]}{[k]!}\mathfrak{B}_{2}^{(a-k,b-k,0)}.
\end{align*}
\end{thm}

\section{iCanonical basis vs monomial basis I: even parity case}\label{section5}

In this section we determine the explicit transition matrices between the monomial basis $B_{i}^{(a)}B_{\tau i}^{(b)}B_{i}^{(c)}\eta$ and icanonical basis $\mathfrak{B}^{(a,b,c)}_{\tau i}$ on module level under the condition that $a,c\geq 1$ and the even parity of $\alpha_{b,c,\bm{\varsigma}}$ and $\beta_{a,b,c,\bm{\varsigma}}$. According to Remark \ref{iandtaui}, we focus on the case $i=1$ in the first two subsections and move to the case $i=2$ in the last subsection. 
\subsection{Case for $\alpha_{b,c,\bm{\varsigma}}>0$, $\alpha_{b,c,\bm{\varsigma}}$ is even} Recall $L(m,n)$ denotes the simple $\U$-module with highest weight vector $\eta$ and highest weight $m\omega_{1}+n\omega_{2}$, for $m,n\geq 0$. Further recall \eqref{icb1}, for each $a,b,c\geq 0$, $\mathfrak{B}^{(a,b,c)}_{2}$ is the icanonical basis element in $L(m,n)$ such that 
\begin{align} 
\mathfrak{B}^{(a,b,c)}_{2}=F^{(a)}_{1}F^{(b)}_{2}F^{(c)}_{1}\eta+\sum_{\substack{x,y\geq 0 \\ (x,y)\not=(0,0)}}q^{-1}\mathbb{Z}[q^{-1}]F^{(a-x)}_{1}F^{(b-x-y)}_{2}F^{(c-y)}_{1}\eta.
\end{align}

In this subsection, $\alpha_{b,c,\bm{\varsigma}}=m-n+b-2c+\varsigma_{2}$ (recall \eqref{defofa}) is always an even natural number. For any $m,n\geq 0$, $0\leq x\leq a$, $0\leq \theta\leq c$, we define 
\begin{align}\label{defomegaabc}
\Omega^{(a,b,c)}_{m,n,x,\theta,\bm{\varsigma}}&=\sum_{k+y=\theta}(-1)^{k}\frac{[\alpha_{b,c,\bm{\varsigma}}][\alpha_{b,c,\bm{\varsigma}}+2]\cdots [\alpha_{b,c,\bm{\varsigma}}+2k-2]}{[k]!}q^{\varsigma_{2}(y-x)}\nonumber\\
&\quad\quad\times q^{-\frac{y(y+1+2n+2c-2b)}{2}-\frac{x(x-1+2m+2b-4c-2a+2y+2k)}{2}}\qbinom{n-b+c+x}{x}\qbinom{m-c+\theta}{y}\in \mathcal{A}.
\end{align}
Once we focus on a fixed $\bm{\varsigma}$, $\Omega^{(a,b,c)}_{m,n,x,\theta,\bm{\varsigma}}$ can be abbreviated as $\Omega^{(a,b,c)}_{m,n,x,\theta}$. When $(x,\theta)=(a,c)$, $\Omega^{(a,b,c)}_{m,n,x,\theta}$ can be abbreviated as $\Omega^{(a,b,c)}_{m,n}$. We will show in Theorem \ref{aeqb+c} that $\Omega^{(a,b,c)}_{m,n,x,\theta,\bm{\varsigma}}$
is the coefficient of $F_1^{(a-x)}F_2^{(b-x-\theta)}F_1^{(c-\theta)}\eta$ in the canonical basis expansion \eqref{icb1} of
$\mathfrak{B}_2^{(a,b,c)}$. Next, we list some recursions on $\Omega^{(a,b,c)}_{m,n,x,\theta}$.

\begin{lem} \label{lemrecursive}
For any $m,n,a,b,c\in \mathbb{Z}_{\geq 0}$ such that $b\geq a+c$, $0\leq x\leq a$, and $0\leq \theta\leq c$, we have 
$$
\Omega^{(a+1,b,c-1)}_{m-1,n+1,x,\theta}=\Omega^{(a,b,c)}_{m,n,x,\theta},
\qquad
\Omega^{(a,b,c)}_{m,n,x,c}=\Omega^{(a+1,b+1,c)}_{m,n+1,x,c},
\qquad
\Omega^{(a,b,c)}_{m,n,x,\theta}=\Omega^{(x,b-a-c+\theta+x,\theta)}_{m-c+\theta,n-a+x,x,\theta}.
$$
\end{lem}
\noindent Recall the condition $n-b+c\geq 0$, together with $b\geq a+c$, imply that $n-a+x\geq0$. 

From now on, we allow the following abbreviation: 
\begin{align}\label{defome1}
\Omega^{a,b,c}_{m,n}=\Omega^{(a,b,c)}_{m,n,a,c}=\sum_{k+y=c}\omega^{(a,b,c)}_{m,n,k,y},
\end{align}
where
\begin{align}\label{defome2}
\omega^{(a,b,c)}_{m,n,k,y}&:=(-1)^{k}\frac{[\alpha_{b,c,\bm{\varsigma}}][\alpha_{b,c,\bm{\varsigma}}+2]\cdots [\alpha_{b,c,\bm{\varsigma}}+2k-2]}{[k]!}q^{\varsigma_{2}(y-a)}\nonumber\\
&\quad\quad\times q^{-\frac{y(y+1+2n+2c-2b)}{2}-\frac{a(-1+2m+2b-4c-a+2y+2k)}{2}}\qbinom{n-b+c+a}{a}\qbinom{m}{y}.
\end{align}
We have the following observation:
\begin{lem}\label{degomega}
When $b\geq a+c$ and $a\geq c\geq 1$, we have $\deg(\omega^{(a,b,c)}_{m,n,k,y})<0$ for any $k+y=c$. 
\end{lem}
\begin{proof}
When $\varsigma_{2}=0$, 
\begin{align*}
\deg(\omega^{(a,b,c)}_{m,n,k,y})&=-\frac{a(a-1)}{2}+(a-c)(n-m+a-2b+3c)-c(b-c)-y^2+y+\frac{c^2-3c}{2}\\
&\leq -\frac{a(a-1)}{2}+(a-c)(n-m+a-2b+3c)-c(b-c)+\frac{c^2-3c}{2}\\
&\leq -\frac{c(c-1)}{2}-c(b-c)+\frac{c^2-3c}{2}=c^2-cb-c\leq -c^2-c<0.
\end{align*}
Here we use the fact that $b-c\ge a\ge c$, and thus $-c(b-c)\le-c^2$. When $\varsigma_{2}=1$, the argument is analogous to that of the previous case and is therefore omitted.
\end{proof}

\begin{lem}\label{lemrecursive1}
Let $m,n,a,b,c\in \mathbb{Z}_{\geq 0}$ such that $m\geq c$, $n-b+c\geq 0$, and $\alpha_{b,c,\bm{\varsigma}}>0$. If $\alpha_{b,c,\bm{\varsigma}}$ is even, we have
\begin{align*}
\Omega^{a,b,c}_{m,n}&=q^{-\varsigma_2a-\frac{a(-1+2m+2b-4c-a+2c)}{2}}\qbinom{n-b+c+a}{a}\\
&\quad\quad\times \sum_{k+y=c}(-1)^{k}\frac{[\alpha_{b,c,\bm{\varsigma}}][\alpha_{b,c,\bm{\varsigma}}+2]\cdots [\alpha_{b,c,\bm{\varsigma}}+2k-2]}{[k]!}q^{\varsigma_{2}y-\frac{y(y+1+2n+2c-2b)}{2}}\qbinom{m}{y}.
\end{align*}
\end{lem}
\begin{proof}
It follows from \eqref{defome1} and \eqref{defome2}.
\end{proof}

Finally, we can introduce the main results in this section:
\begin{thm}\label{aeqb+c}
Let $m,n,a,b,c\in \mathbb{Z}_{\geq 0}$ such that $m\geq c$, $n-b+c\geq 0$, and $\alpha_{b,c,\bm{\varsigma}}> 0$. When $\alpha_{b,c,\bm{\varsigma}}$ is even, $b\geq a+c$, we have
\begin{align}\label{generalabc1}
B^{(a)}_{1}B^{(b)}_{2}B^{(c)}_{1}\eta=&[\alpha_{b,c,\bm{\varsigma}}]\sum_{\substack{k \ge 0 \\ k\equiv \alpha_{b,c,\bm{\varsigma}} \pmod 2}}\frac{[\alpha_{b,c,\bm{\varsigma}}-k+2][\alpha_{b,c,\bm{\varsigma}}-k+4]\cdots[\alpha_{b,c,\bm{\varsigma}}+k-2]}{[k]!}\mathfrak{B}^{(a,b-k,c-k)}_{2}\nonumber\\
&\quad\quad+\sum_{\substack{k\geq 0 \\ k\equiv \alpha_{b,c,\bm{\varsigma}}+1 \pmod 2}}\frac{[\alpha_{b,c,\bm{\varsigma}}-k+1][\alpha_{b,c,\bm{\varsigma}}-k+3]\cdots [\alpha_{b,c,\bm{\varsigma}}+k-1]}{[k]!}\mathfrak{B}^{(a,b-k,c-k)}_{2},
\end{align}
and
\begin{align}\label{generalabc2}
\mathfrak{B}_{2}^{(a,b,c)}=\sum_{0\leq k\leq c}(-1)^{k}\frac{[\alpha_{b,c,\bm{\varsigma}}][\alpha_{b,c,\bm{\varsigma}}+2]\cdots [\alpha_{b,c,\bm{\varsigma}}+2k-2]}{[k]!}B_{1}^{(a)}B^{(b-k)}_{2}B^{(c-k)}_{1}\eta. 
\end{align}
\end{thm}
\begin{proof}
First, observe that we only need to prove \eqref{generalabc2}, and \eqref{generalabc1} follow from the proof of Proposition \ref{inverse}. Combining the RHS of \eqref{generalabc2} with Lemma \ref{prop59}, we get
$$
\text{RHS of \eqref{generalabc2}}=\sum_{0\leq x\leq a}\sum_{0\leq \theta\leq c}\Omega^{(a,b,c)}_{m,n,x,\theta}F^{(a-x)}_{1}F^{(b-x-\theta)}_{2}F_{1}^{(c-\theta)}\eta,
$$
where $\Omega^{(a,b,c)}_{m,n,x,\theta}$ was defined in \eqref{defomegaabc}.

Now we want to show that $\Omega^{(a,b,c)}_{m,n,x,\theta}\in q^{-1}\mathbb{Z}[q^{-1}]$ for any $a,b,c\geq 0$, $0\leq x\leq a$, $0\leq \theta\leq c$, and $(x,\theta)\neq (0,0)$. Using Lemma \ref{lemrecursive} above,  we only need to prove $\Omega^{(a,b,c)}_{m,n,a,c}=\Omega^{a,b,c}_{m,n}\in q^{-1}\mathbb{Z}[q^{-1}]$. When $a=0$ or $c=0$, the assertion follows from the two-factor results in Section \ref{section4}. We therefore assume $a,c\ge1$.

For $a\geq c$, we have $\Omega^{a,b,c}_{m,n}\in q^{-1}\mathbb{Z}[q^{-1}]$, by Lemma \ref{degomega}. For $1\leq a\leq c-1$, Lemma \ref{lemrecursive1} gives
\begin{align}\label{3vs2-factors}
\Omega^{a,b,c}_{m,n}
=
q^{-\varsigma_2a-\frac{a(-1+2m+2b-2c-a)}{2}}
\qbinom{n-b+c+a}{a}
\Omega^{(0,b,c)}_{m,n,0,c}.
\end{align}
According to the two-factor results in Section \ref{section4}, $\Omega^{(0,b,c)}_{m,n,0,c}\in q^{-1}\mathbb{Z}[q^{-1}]$, while
\begin{align}
&\deg\left(
q^{-\varsigma_2a-\frac{a(-1+2m+2b-2c-a)}{2}}
\qbinom{n-b+c+a}{a}
\right)=-\varsigma_2a-\frac{a(-1+2m+2b-2c-a)}2+a(n-b+c)\notag
\\
&\quad\quad\quad=a\beta_{a,b,c,\bm{\varsigma}}
-\frac{a(a+1)}{2}-ac<0. \label{degree-3-factors}
\end{align}
Thus $\Omega^{a,b,c}_{m,n}\in q^{-1}\mathbb{Z}[q^{-1}]$
also in this case.

Since the RHS of \eqref{generalabc2} is $\psi^\imath$-invariant and has leading coefficient $1$, \eqref{generalabc2} follows from the defining properties and uniqueness of the icanonical basis. So we are done.
\end{proof}

\subsection{Case for $\beta_{a,b,c,\bm{\varsigma}}$ is even, $\beta_{a,b,c,\bm{\varsigma}}>0$} In this subsection, $\beta_{a,b,c,\bm{\varsigma}}=n-m+a-2b+3c+\varsigma_{1}$ is always an even natural number. For any $m,n\geq 0$, $0\leq y\leq c$, $0\leq \theta\leq a$, we define 
\begin{align}\label{defoabc}
O^{(a,b,c)}_{m,n,\theta,y,\bm{\varsigma}}&=\sum_{x+k=\theta}(-1)^{k}\frac{[\beta_{a,b,c,\bm{\varsigma}}][\beta_{a,b,c,\bm{\varsigma}}+2]\cdots [\beta_{a,b,c,\bm{\varsigma}}+2k-2]}{[k]!}q^{\varsigma_{1}(x-y)}\nonumber\\
&\quad\quad\times q^{-\frac{y(y-1+2n+2c-2b+2k)}{2}-\frac{x(x+1+2m+2b-4c-2a+2y)}{2}}\qbinom{n-b+c+\theta}{x}\qbinom{m-c+y}{y}\in \mathcal{A}.
\end{align}
Once we focus on a fixed $\bm{\varsigma}$, $O^{(a,b,c)}_{m,n,\theta,y,\bm{\varsigma}}$ can be abbreviated as $O^{(a,b,c)}_{m,n,\theta,y}$. When $(\theta,y)=(a,c)$, $O^{(a,b,c)}_{m,n,\theta,y}$ can be abbreviated as $O^{(a,b,c)}_{m,n}$. We will show in Theorem \ref{thmm512} that $O^{(a,b,c)}_{m,n,\theta,y,\bm{\varsigma}}$
is the coefficient of $F_1^{(a-\theta)}F_2^{(b-\theta-y)}F_1^{(c-y)}\eta$ in the canonical basis expansion \eqref{icb1} of
$\mathfrak{B}_2^{(a,b,c)}$. Next, we list some recursions about $O_{m,n,\theta,y}$ that we just defined.
\begin{lem} \label{lemrecursiveo}
When $b\geq a+c$, for any $m,n,a,b,c\in \mathbb{Z}_{\geq 0}$, $0\leq \theta\leq a$, $0\leq y\leq c$, we have 
$$
O^{(a-1,b,c+1)}_{m+1,n-1,\theta,y}=O^{(a,b,c)}_{m,n,\theta,y},
\qquad
O^{(a,b,c)}_{m,n,a,y}=O^{(a,b+1,c+1)}_{m+1,n,a,y},
\qquad
O^{(a,b,c)}_{m,n,\theta,y}=O^{(\theta,b-a-c+\theta+y,y)}_{m-c+y,n-a+\theta,\theta,y},
$$
and
\begin{align}\label{oandomega}
{\qbinom{m+b-a-c}{c}O^{a,b,c}_{m,n,\bm{\varsigma}}=\qbinom{m}{c}}\Omega^{c,a+c,a}_{n-b+c+a,m+b-a-c,1-\bm{\varsigma}}.
\end{align}
\end{lem}
\begin{proof}
The proof can be found in Appendix \ref{app6}.
\end{proof}
Recall condition $n-b+c\geq 0$, together with $b\geq a+c$, tell us that $n-a+\theta\geq0$. 
From now on, we allow the following abbreviation: 
\begin{align}\label{defoome1}
O^{a,b,c}_{m,n}=O^{(a,b,c)}_{m,n,a,c}=\sum_{x+k=a}o^{(a,b,c)}_{m,n,x,k},
\end{align}
where
\begin{align}\label{defoome2}
o^{(a,b,c)}_{m,n,x,k}&:=(-1)^{k}\frac{[\beta_{a,b,c,\bm{\varsigma}}][\beta_{a,b,c,\bm{\varsigma}}+2]\cdots [\beta_{a,b,c,\bm{\varsigma}}+2k-2]}{[k]!}q^{\varsigma_{1}(x-c)}\nonumber\\
&\quad\quad\times q^{-\frac{c(-1+2n+3c-2b+2k)}{2}-\frac{x(x+1+2m+2b-4c-2a+2c)}{2}}\qbinom{n-b+c+a}{x}\qbinom{m}{c}.
\end{align}

From \eqref{oandomega}, if $b\geq a+c$, we have $O^{(a,b,c)}_{m,n}\in q^{-1}\mathbb{Z}[q^{-1}]$. This leads to the following theorem:
\begin{thm}\label{thmm512}
Let $m,n,a,b,c\in \mathbb{Z}_{\geq 0}$ such that $m\geq c$, $n-b+c\geq 0$, $\beta_{a,b,c,\bm{\varsigma}}>0$, $\beta_{a,b,c,\bm{\varsigma}}$ is even, and $b\geq a+c$, then
\begin{align}\label{generalabc55}
B^{(a)}_{1}B^{(b)}_{2}B^{(c)}_{1}\eta&=[\beta_{a,b,c,\bm{\varsigma}}]\sum_{\substack{k \ge 0 \\ k\equiv \beta_{a,b,c,\bm{\varsigma}} \pmod 2}}\frac{[\beta_{a,b,c,\bm{\varsigma}}-k+2][\beta_{a,b,c,\bm{\varsigma}}-k+4]\cdots [\beta_{a,b,c,\bm{\varsigma}}+k-2]}{[k]!}\mathfrak{B}_{2}^{a-k,b-k,c}\nonumber\\
&\quad\quad+\sum_{\substack{k \ge 0 \\ k\equiv \beta_{a,b,c,\bm{\varsigma}}+1 \pmod 2}}\frac{[\beta_{a,b,c,\bm{\varsigma}}-k+1][\beta_{a,b,c,\bm{\varsigma}}-k+3]\cdots [\beta_{a,b,c,\bm{\varsigma}}+k-1]}{[k]!}\mathfrak{B}_{2}^{a-k,b-k,c},
\end{align}
\begin{align}\label{generalabc66}
\mathfrak{B}_{2}^{(a,b,c)}=\sum_{0\leq k\leq a}(-1)^{k}\frac{[\beta_{a,b,c,\bm{\varsigma}}][\beta_{a,b,c,\bm{\varsigma}}+2]\cdots [\beta_{a,b,c,\bm{\varsigma}}+2k-2]}{[k]!}B_{1}^{(a-k)}B^{(b-k)}_{2}B^{(c)}_{1}\eta. 
\end{align}
\end{thm}

\subsection{Case for $B_{2}^{(a)}B_{1}^{(b)}B_{2}^{(c)}\eta$ and $\mathfrak{B}^{(a,b,c)}_{1}$}
Recall \eqref{icb1}, for each $a,b,c\geq 0$, denote $\mathfrak{B}^{(a,b,c)}_{1}$ as the unique element in $L(m,n)$ such that $\psi^{\imath}(\mathfrak{B}^{(a,b,c)}_{1})=\mathfrak{B}^{(a,b,c)}_{1}$ and
\begin{align} 
\mathfrak{B}^{(a,b,c)}_{1}=F^{(a)}_{2}F^{(b)}_{1}F^{(c)}_{2}\eta+\sum_{\substack{x,y\geq 0 \\ (x,y)\not=(0,0)}}q^{-1}\mathbb{Z}[q^{-1}]F^{(a-x)}_{2}F^{(b-x-y)}_{1}F^{(c-y)}_{2}\eta.
\end{align}
All $\mathfrak{B}^{(a,b,c)}_{1}$ belong to the icanonical basis of $L(m,n)$. 
In this subsection, define 
\[
\hat{\alpha}_{b,c,\bm{\varsigma}}:=n-m+b-2c+\varsigma_{1},
\qquad
\hat{\beta}_{a,b,c,\bm{\varsigma}}:=m-n+a-2b+3c+\varsigma_{2},
\]
and we only consider the case when $\hat{\alpha}_{b,c,\bm{\varsigma}}$ or $\hat{\beta}_{a,b,c,\bm{\varsigma}}$ are even. By symmetry, Theorems \ref{aeqb+c} and \ref{thmm512} can be extended to the following:

\begin{thm}\label{thm5144}
Let $m,n,a,b,c\in \mathbb{Z}_{\geq 0}$ such that $n\geq c$, $m-b+c\geq 0$, $\hat{\alpha}_{b,c,\bm{\varsigma}}>0$, $\hat{\alpha}_{b,c,\bm{\varsigma}}$ is even, and $b\geq a+c$, then
\begin{align*}
B^{(a)}_{2}B^{(b)}_{1}B^{(c)}_{2}\eta&=[\hat{\alpha}_{b,c,\bm{\varsigma}}]\sum_{\substack{k \ge 0 \\ k\equiv \hat{\alpha}_{b,c,\bm{\varsigma}} \pmod 2}}\frac{[\hat{\alpha}_{b,c,\bm{\varsigma}}-k+2]\cdots [\hat{\alpha}_{b,c,\bm{\varsigma}}+k-4][\hat{\alpha}_{b,c,\bm{\varsigma}}+k-2]}{[k]!}\mathfrak{B}_{1}^{(a,b-k,c-k)}\\
&\quad\quad+\sum_{\substack{k\geq 0 \\ k\equiv \hat{\alpha}_{b,c,\bm{\varsigma}}+1 \pmod 2}}\frac{[\hat{\alpha}_{b,c,\bm{\varsigma}}-k+1][\hat{\alpha}_{b,c,\bm{\varsigma}}-k+3]\cdots [\hat{\alpha}_{b,c,\bm{\varsigma}}+k-1]}{[k]!}\mathfrak{B}^{(a,b-k,c-k)}_{1}.
\end{align*}
\begin{align*}
\mathfrak{B}_{1}^{(a,b,c)}&=\sum_{0\leq k\leq c}(-1)^{k}\frac{[\hat{\alpha}_{b,c,\bm{\varsigma}}][\hat{\alpha}_{b,c,\bm{\varsigma}}+2]\cdots [\hat{\alpha}_{b,c,\bm{\varsigma}}+2k-2]}{[k]!}B_{2}^{(a)}B^{(b-k)}_{1}B^{(c-k)}_{2}\eta.
\end{align*} 
\end{thm}

\begin{thm}\label{thm51616}
Let $m,n,a,b,c\in \mathbb{Z}_{\geq 0}$ such that $n\geq c$, $m-b+c\geq 0$, $\hat{\beta}_{a,b,c,\bm{\varsigma}}>0$, $\hat{\beta}_{a,b,c,\bm{\varsigma}}$ is even, and $b\geq a+c$, then
\begin{align*}
B^{(a)}_{2}B^{(b)}_{1}B^{(c)}_{2}\eta&=[\hat{\beta}_{a,b,c,\bm{\varsigma}}]\sum_{\substack{k \ge 0 \\ k\equiv \hat{\beta}_{a,b,c,\bm{\varsigma}} \pmod 2}}\frac{[\hat{\beta}_{a,b,c,\bm{\varsigma}}-k+2]\cdots [\hat{\beta}_{a,b,c,\bm{\varsigma}}+k-4][\hat{\beta}_{a,b,c,\bm{\varsigma}}+k-2]}{[k]!}\mathfrak{B}_{1}^{(a-k,b-k,c)}\\
&\quad\quad+\sum_{\substack{k\geq 0 \\ k\equiv \hat{\beta}_{a,b,c,\bm{\varsigma}}+1 \pmod 2}}\frac{[\hat{\beta}_{a,b,c,\bm{\varsigma}}-k+1][\hat{\beta}_{a,b,c,\bm{\varsigma}}-k+3]\cdots [\hat{\beta}_{a,b,c,\bm{\varsigma}}+k-1]}{[k]!}\mathfrak{B}_{1}^{(a-k,b-k,c)}.
\end{align*}
\begin{align*}
\mathfrak{B}_{1}^{(a,b,c)}&=\sum_{0\leq k\leq a}(-1)^{k}\frac{[\hat{\beta}_{a,b,c,\bm{\varsigma}}][\hat{\beta}_{a,b,c,\bm{\varsigma}}+2]\cdots [\hat{\beta}_{a,b,c,\bm{\varsigma}}+2k-2]}{[k]!}B_{2}^{(a-k)}B^{(b-k)}_{1}B^{(c)}_{2}\eta.
\end{align*}
\end{thm}

\section{iCanonical basis vs monomial basis II: odd parity case} \label{section6}
Parallel to Section \ref{section5}, in this section we determine the explicit transition matrices between the monomial basis $B_{i}^{(a)}B_{\tau i}^{(b)}B_{i}^{(c)}\eta$ and icanonical basis $\mathfrak{B}^{(a,b,c)}_{\tau i}$ on module level under the condition that $a,c\geq 1$ and the odd parity of $\alpha_{b,c,\bm{\varsigma}}$ and $\beta_{a,b,c,\bm{\varsigma}}$. According to Remark \ref{iandtaui}, we focus on the case $i=1$ in the first two subsections and move to the case $i=2$ in the last subsection.

\subsection{Case for $\alpha_{b,c,\bm{\varsigma}}$ or $\beta_{a,b,c,\bm{\varsigma}}$ is odd}
In this subsection, $\alpha_{b,c,\bm{\varsigma}}=m-n+b-2c+\varsigma_{2}$ or $\beta_{a,b,c,\bm{\varsigma}}=n-m+a-2b+3c+\varsigma_{1}$ are always odd natural numbers. First, recall $L(m,n)$ is the simple $\U$-module with the highest weight $m\omega_{1}+n\omega_{2}$ and highest weight vector $\eta$. Further recall \eqref{icb1}, for each $a,b,c\geq 0$, $\mathfrak{B}^{(a,b,c)}_{i}$ is the icanonical basis element in $L(m,n)$ such that 
\begin{align} 
\mathfrak{B}^{(a,b,c)}_{i}=F^{(a)}_{\tau i}F^{(b)}_{i}F^{(c)}_{\tau i}\eta+\sum_{\substack{x,y\geq 0 \\ (x,y)\not=(0,0)}}q^{-1}\mathbb{Z}[q^{-1}]F^{(a-x)}_{\tau i}F^{(b-x-y)}_{i}F^{(c-y)}_{\tau i}\eta.
\end{align}

We focus on the case $\alpha_{b,c,\bm{\varsigma}}=m-n+b-2c+\varsigma_{2}$ is odd first. For any $m,n\geq 0$, $0\leq x\leq a$, $0\leq \theta\leq c$, we define 
\begin{align}\label{defxiabcodd}
\Xi^{(a,b,c)}_{m,n,x,\theta,\bm{\varsigma}}&=\sum_{k+y=\theta}(-1)^{k}[\alpha_{b,c,\bm{\varsigma}}+k-1]\frac{[\alpha_{b,c,\bm{\varsigma}}+1][\alpha_{b,c,\bm{\varsigma}}+3]\cdots [\alpha_{b,c,\bm{\varsigma}}+2k-3]}{[k]!}q^{\varsigma_{2}(y-x)}\nonumber\\
&\quad\quad\times q^{-\frac{y(y+1+2n+2c-2b)}{2}-\frac{x(x-1+2m+2b-4c-2a+2y+2k)}{2}}\qbinom{n-b+c+x}{x}\qbinom{m-c+\theta}{y}\in \mathcal{A}.
\end{align}

Once we focus on a fixed $\bm{\varsigma}$, $\Xi^{(a,b,c)}_{m,n,x,\theta,\bm{\varsigma}}$ can be abbreviated as $\Xi^{(a,b,c)}_{m,n,x,\theta}$. When $(x,\theta)=(a,c)$, $\Xi^{(a,b,c)}_{m,n,x,\theta}$ can be abbreviated as $\Xi^{(a,b,c)}_{m,n}$. We will show in Theorem \ref{thm6155} that $\Xi^{(a,b,c)}_{m,n,x,\theta,\bm{\varsigma}}$
is the coefficient of $F_1^{(a-x)}F_2^{(b-x-\theta)}F_1^{(c-\theta)}\eta$ in the canonical basis expansion \eqref{icb1} of
$\mathfrak{B}_2^{(a,b,c)}$. It is easy to check that $\Xi^{(a,b,c)}_{m,n,x,\theta}$ satisfies the similar recursions as Lemma \ref{lemrecursive}:
\begin{lem} \label{lemrecursiveodd}
When $b\geq a+c$, for any $m,n,a,b,c\in \mathbb{Z}_{\geq 0}$, $0\leq x\leq a$, $0\leq \theta\leq c$, we have 
$$
\Xi^{(a+1,b,c-1)}_{m-1,n+1,x,\theta}=\Xi^{(a,b,c)}_{m,n,x,\theta},
\qquad
\Xi^{(a,b,c)}_{m,n,x,c}=\Xi^{(a+1,b+1,c)}_{m,n+1,x,c},
\qquad
\Xi^{(a,b,c)}_{m,n,x,\theta}=\Xi^{(x,b-a-c+\theta+x,\theta)}_{m-c+\theta,n-a+x,x,\theta}.
$$
\end{lem}

From now on, we allow the following abbreviation: 
\begin{align}\label{defxi1}
\Xi^{a,b,c}_{m,n}=\Xi^{(a,b,c)}_{m,n,a,c}=\sum_{k+y=c}\xi^{(a,b,c)}_{m,n,k,y},
\end{align}
where
\begin{align}\label{defxi2}
\xi^{(a,b,c)}_{m,n,k,y}&:=(-1)^{k}[\alpha_{b,c,\bm{\varsigma}}+k-1]\frac{[\alpha_{b,c,\bm{\varsigma}}+1][\alpha_{b,c,\bm{\varsigma}}+3]\cdots [\alpha_{b,c,\bm{\varsigma}}+2k-3]}{[k]!}q^{\varsigma_{2}(y-a)}\nonumber\\
&\quad\quad\times q^{-\frac{y(y+1+2n+2c-2b)}{2}-\frac{a(-1+2m+2b-4c-a+2y+2k)}{2}}\qbinom{n-b+c+a}{a}\qbinom{m}{y}.
\end{align}
Now we have the following lemma:
\begin{lem}\label{lem417}
Let $m,n,a,b,c\in\mathbb Z_{\geq0}$ satisfy $b\geq a+c$, $m\geq c$, and $n-b+c\geq0$. Suppose that $\alpha_{b,c,\bm{\varsigma}}$ is positive odd. Then, for $0\leq x\leq a$ and $0\leq\theta\leq c$,
\begin{align*}
\Xi^{(a,b,c)}_{m,n,x,\theta}
=q^{x+\theta}\Omega^{(a,b,c)}_{m+1,n,x,\theta}.
\end{align*}
Specifically, 
\begin{align}\label{lem4177}
\Xi^{a,b,c}_{m,n}
=q^{a+c}\Omega^{a,b,c}_{m+1,n}.
\end{align}
\end{lem}
\begin{proof}
Using \eqref{defomegaabc}, \eqref{defxiabcodd}, and the fact that $\alpha_{b,c,\bm{\varsigma}}=m-n+b-2c+\varsigma_2$, we compute
\begin{align*}
q^{x+\theta}\Omega^{(a,b,c)}_{m+1,n,x,\theta}
&=\sum_{k+y=\theta}
(-1)^k
\frac{\prod_{j=0}^{k-1}
[\alpha_{b,c,\bm{\varsigma}}+1+2j]}{[k]!}
q^{x+\theta+\varsigma_2(y-x)}
\\[-2pt]
&\qquad\times
q^{-\frac{y(y+1+2n+2c-2b)}2
-\frac{x(x-1+2(m+1)+2b-4c-2a+2\theta)}2}
\begin{bmatrix}n-b+c+x\\x\end{bmatrix}
\begin{bmatrix}m+1-c+\theta\\y\end{bmatrix}
\\[-2pt]
&=\sum_{k+y=\theta}
(-1)^k
\frac{\prod_{j=0}^{k-1}
[\alpha_{b,c,\bm{\varsigma}}+1+2j]}{[k]!}
q^{\theta+\varsigma_2(y-x)}
\\[-2pt]
&\qquad\times
q^{-\frac{y(y+1+2n+2c-2b)}2
-\frac{x(x-1+2m+2b-4c-2a+2\theta)}2}
\begin{bmatrix}n-b+c+x\\x\end{bmatrix}
\begin{bmatrix}m+1-c+\theta\\y\end{bmatrix}.
\end{align*}
Apply quantum Pascal
\[
\begin{bmatrix}m+1-c+\theta\\y\end{bmatrix}
=
q^{-y}\begin{bmatrix}m-c+\theta\\y\end{bmatrix}
+
q^{m+1-c+\theta-y}
\begin{bmatrix}m-c+\theta\\y-1\end{bmatrix}.
\]
We then get
\begin{align*}
q^{x+\theta}
\Omega^{(a,b,c)}_{m+1,n,x,\theta}
&=
\sum_{\substack{k,y\geq 0\\k+y=\theta}}
(-1)^k
\frac{\prod_{j=0}^{k-1}
[\alpha_{b,c,\bm{\varsigma}}+1+2j]}{[k]!}
q^{k+\varsigma_2(y-x)}
\\
&\qquad\times
q^{-\frac{y(y+1+2n+2c-2b)}2
-\frac{x(x-1+2m+2b-4c-2a+2\theta)}2}
\begin{bmatrix}n-b+c+x\\x\end{bmatrix}
\begin{bmatrix}m-c+\theta\\y\end{bmatrix}
\\
&\quad+\sum_{\substack{k\geq 0,y\geq 1\\k+y=\theta}}
(-1)^k
\frac{\prod_{j=0}^{k-1}
[\alpha_{b,c,\bm{\varsigma}}+1+2j]}{[k]!}
q^{\blue{\theta+\varsigma_2(y-x)+m+1-c+\theta-y}}
\\
&\qquad\times
q^{\blue{-\frac{y(y+1+2n+2c-2b)}2}
-\frac{x(x-1+2m+2b-4c-2a+2\theta)}2}
\begin{bmatrix}n-b+c+x\\x\end{bmatrix}
\begin{bmatrix}m-c+\theta\\y-1\end{bmatrix}
\\
&\overset{\substack{k\mapsto k-1\\y\mapsto y+1}}{=}\sum_{\substack{k,y\geq 0\\k+y=\theta}}
(-1)^k
\frac{\prod_{j=0}^{k-1}
[\alpha_{b,c,\bm{\varsigma}}+1+2j]}{[k]!}
q^{k+\varsigma_2(y-x)}
\\
&\qquad\times
q^{-\frac{y(y+1+2n+2c-2b)}2
-\frac{x(x-1+2m+2b-4c-2a+2\theta)}2}
\begin{bmatrix}n-b+c+x\\x\end{bmatrix}
\begin{bmatrix}m-c+\theta\\y\end{bmatrix}
\\
&\quad+\sum_{\substack{k\geq 1,y\geq 0\\k+y=\theta}}
(-1)^{k-1}
\frac{\prod_{j=0}^{k-2}
[\alpha_{b,c,\bm{\varsigma}}+1+2j]}{[k-1]!}
q^{\blue{\alpha_{b,c,\bm{\varsigma}}+2k-1+\varsigma_{2}(y-x)}}
\\
&\qquad\times
q^{\blue{-\frac{y(y+1+2n+2c-2b)}2}
-\frac{x(x-1+2m+2b-4c-2a+2\theta)}2}
\begin{bmatrix}n-b+c+x\\x\end{bmatrix}
\begin{bmatrix}m-c+\theta\\y\end{bmatrix}
\\
&=
q^{\varsigma_2(\theta-x)-\frac{\theta(\theta+1+2n+2c-2b)}2
-\frac{x(x-1+2m+2b-4c-2a+2\theta)}2}
\begin{bmatrix}n-b+c+x\\x\end{bmatrix}
\begin{bmatrix}m-c+\theta\\\theta\end{bmatrix}
\\
&\quad+\sum_{\substack{k\geq 1,y\geq 0\\k+y=\theta}}
(-1)^{k}
\frac{\prod_{j=0}^{k-2}
[\alpha_{b,c,\bm{\varsigma}}+1+2j]}{[k]!}q^{\varsigma_2(y-x)}
\\
&\qquad\times
q^{-\frac{y(y+1+2n+2c-2b)}2
-\frac{x(x-1+2m+2b-4c-2a+2\theta)}2}
\begin{bmatrix}n-b+c+x\\x\end{bmatrix}
\begin{bmatrix}m-c+\theta\\y\end{bmatrix}
\\
&\qquad\times 
\left(
q^k[\alpha_{b,c,\bm{\varsigma}}+2k-1]
-q^{\alpha_{b,c,\bm{\varsigma}}+2k-1}[k]
\right)
\\
&\overset{\eqref{quantumaplusb}}{=}
\sum_{\substack{k\geq 0,y\geq 0\\k+y=\theta}}
(-1)^{k}
[\alpha_{b,c,\bm{\varsigma}}+k-1]\frac{\prod_{j=0}^{k-2}
[\alpha_{b,c,\bm{\varsigma}}+1+2j]}{[k]!}q^{\varsigma_2(y-x)}
\\
&\qquad\times
q^{-\frac{y(y+1+2n+2c-2b)}2
-\frac{x(x-1+2m+2b-4c-2a+2\theta)}2}
\begin{bmatrix}n-b+c+x\\x\end{bmatrix}
\begin{bmatrix}m-c+\theta\\y\end{bmatrix}
\\
&=\Xi^{(a,b,c)}_{m,n,x,\theta},
\end{align*}
where
\begin{equation}\label{quantumaplusb}
q^k[\alpha_{b,c,\bm{\varsigma}}+2k-1]
-q^{\alpha_{b,c,\bm{\varsigma}}+2k-1}[k]
=
[\alpha_{b,c,\bm{\varsigma}}+k-1],\tag{\ding{164}}
\end{equation}
and the substitution $k_{\text{old}}=k-1$, $y_{\text{old}}=y+1$ only applies to the second summand, with the computation 
\[
\begin{aligned}
&\theta+\varsigma_2(y+1-x)+m+1-c+\theta-(y+1)
-\frac{(y+1)(y+2+2n+2c-2b)}2\\
&\qquad=
\varsigma_2(y-x)-\frac{y(y+1+2n+2c-2b)}2
+m-n+b-2c+\varsigma_2+2(\theta-y)-1\\
&\qquad=
\varsigma_2(y-x)-\frac{y(y+1+2n+2c-2b)}2
+\alpha_{b,c,\bm{\varsigma}}+2k-1.
\end{aligned}
\]
\end{proof}

\begin{lem}\label{degxiodddd1}
Under the hypotheses of Lemma \ref{lem417}, suppose further
that $(a,c)\neq(0,0)$. Then $\Xi^{a,b,c}_{m,n}\in q^{-1}\mathbb{Z}[q^{-1}]$.
\end{lem}
\begin{proof}
Using \eqref{3vs2-factors} and \eqref{degree-3-factors}, one can compute
\begin{align*}
\deg\Omega^{a,b,c}_{m+1,n}
={}&a\bigl(\beta_{a,b,c,\bm{\varsigma}}-1\bigr)-\frac{a(a+1)}2-ac+\deg\Omega^{(0,b,c)}_{m+1,n,0,c}.
\end{align*}
Here increasing $m$ by $1$ decreases $\beta_{a,b,c,\bm{\varsigma}}$ by $1$. Therefore,
\begin{align*}
\deg\!\left(q^{a+c}\Omega^{a,b,c}_{m+1,n}\right)
={}&a+c+a\bigl(\beta_{a,b,c,\bm{\varsigma}}-1\bigr)
-\frac{a(a+1)}2-ac+\deg\Omega^{(0,b,c)}_{m+1,n,0,c}\\
={}&a\beta_{a,b,c,\bm{\varsigma}}
-\frac{a(a+1)}2-ac+\left(c+\deg\Omega^{(0,b,c)}_{m+1,n,0,c}\right).
\end{align*}
For $c=0$ and $a\geq 1$, $\deg\Xi^{a,b,c}_{m,n}=a\beta_{a,b,0,\bm{\varsigma}}-\frac{a(a+1)}{2}<0$. For $c\ge1$, Lemma \ref{lem36} gives $c+\deg\Omega^{(0,b,c)}_{m+1,n,0,c}<0$. Meanwhile, we also have $\beta_{a,b,c,\bm{\varsigma}}\leq 0$. Consequently,
\[
\deg\Xi^{a,b,c}_{m,n}
<
a\beta_{a,b,c,\bm{\varsigma}}
-\frac{a(a+1)}2-ac
\le0.
\]
\end{proof}

With these preparations, we can now present the main theorem of this section.
\begin{thm}\label{thm6155}
Let $m,n,a,b,c\in \mathbb{Z}_{\geq 0}$ such that $m\geq c$, $n-b+c\geq 0$, $\alpha_{b,c,\bm{\varsigma}}>0$, $\alpha_{b,c,\bm{\varsigma}}$ is odd, and $b\geq a+c$, then
\begin{align}\label{generalabc1odd}
B^{(a)}_{1}B^{(b)}_{2}B^{(c)}_{1}\eta=&[\alpha_{b,c,\bm{\varsigma}}]\sum_{\substack{k \ge 0 \\ k\equiv \alpha_{b,c,\bm{\varsigma}} \pmod 2}}\frac{[\alpha_{b,c,\bm{\varsigma}}-k+2][\alpha_{b,c,\bm{\varsigma}}-k+4]\cdots[\alpha_{b,c,\bm{\varsigma}}+k-2]}{[k]!}\mathfrak{B}^{(a,b-k,c-k)}_{2}\nonumber\\
&\quad\quad+\sum_{\substack{k\geq 0 \\ k\equiv \alpha_{b,c,\bm{\varsigma}}+1 \pmod 2}}\frac{[\alpha_{b,c,\bm{\varsigma}}-k+1][\alpha_{b,c,\bm{\varsigma}}-k+3]\cdots [\alpha_{b,c,\bm{\varsigma}}+k-1]}{[k]!}\mathfrak{B}^{(a,b-k,c-k)}_{2},
\end{align}
\begin{align}\label{generalabc2odd}
\mathfrak{B}_{2}^{(a,b,c)}=\sum_{0\leq k\leq c}(-1)^{k}[\alpha_{b,c,\bm{\varsigma}}+k-1]\frac{[\alpha_{b,c,\bm{\varsigma}}+1][\alpha_{b,c,\bm{\varsigma}}+3]\cdots [\alpha_{b,c,\bm{\varsigma}}+2k-3]}{[k]!}B_{1}^{(a)}B^{(b-k)}_{2}B^{(c-k)}_{1}\eta. 
\end{align}
\end{thm}

\begin{proof}
The proof is almost the same as Theorem \ref{aeqb+c}. First, observe that we only need to prove \eqref{generalabc2odd}, and \eqref{generalabc1odd} follow from the proof of Proposition \ref{inverse}. Combining the RHS of \eqref{generalabc2odd} with \eqref{prop59}, we get
$$
\text{RHS of \eqref{generalabc2odd}}=\sum_{0\leq x\leq a}\sum_{0\leq \theta\leq c}\Xi^{(a,b,c)}_{m,n,x,\theta}F^{(a-x)}_{1}F^{(b-x-\theta)}_{2}F_{1}^{(c-\theta)}\eta,
$$
where $\Xi^{(a,b,c)}_{m,n,x,\theta}$ was defined in \eqref{defxiabcodd}.

Now we want to show that $\Xi^{(a,b,c)}_{m,n,x,\theta}\in q^{-1}\mathbb{Z}[q^{-1}]$ for any $a,b,c\geq 0$, $0\leq x\leq a$, $0\leq \theta\leq c$, and $(x,\theta)\neq (0,0)$. Using Lemma \ref{lemrecursiveodd} above,  we only need to prove $\Xi^{(a,b,c)}_{m,n,a,c}=\Xi^{a,b,c}_{m,n}\in q^{-1}\mathbb{Z}[q^{-1}]$, which can be proved by Lemma \ref{degxiodddd1}. We are done.
\end{proof}

Next, we focus on the case $\beta_{a,b,c,\bm{\varsigma}}=n-m+a-2b+3c+\varsigma_{1}>0$, $\beta_{a,b,c,\bm{\varsigma}}$ is odd. For any $m,n\geq 0$, $0\leq y\leq c$, $0\leq \theta\leq a$, we define 
\begin{align}\label{defoabcoddddddd}
Z^{(a,b,c)}_{m,n,\theta,y,\bm{\varsigma}}&=\sum_{x+k=\theta}(-1)^{k}[\beta_{a,b,c,\bm{\varsigma}}+k-1]\frac{[\beta_{a,b,c,\bm{\varsigma}}+1][\beta_{a,b,c,\bm{\varsigma}}+3]\cdots [\beta_{a,b,c,\bm{\varsigma}}+2k-3]}{[k]!}q^{\varsigma_{1}(x-y)}\nonumber\\
&\quad\quad\times q^{-\frac{y(y-1+2n+2c-2b+2k)}{2}-\frac{x(x+1+2m+2b-4c-2a+2y)}{2}}\qbinom{n-b+c+\theta}{x}\qbinom{m-c+y}{y}\in \mathcal{A}.
\end{align}
Once we focus on a fixed $\bm{\varsigma}$, $Z^{(a,b,c)}_{m,n,\theta,y,\bm{\varsigma}}$ can be abbreviated as $Z^{(a,b,c)}_{m,n,\theta,y}$. When $(\theta,y)=(a,c)$, $Z^{(a,b,c)}_{m,n,\theta,y}$ can be abbreviated as $Z^{a,b,c}_{m,n}$. We will show in Theorem \ref{thm688} that $\Xi^{(a,b,c)}_{m,n,\theta,y,\bm{\varsigma}}$
is the coefficient of $F_1^{(a-\theta)}F_2^{(b-\theta-y)}F_1^{(c-y)}\eta$ in the canonical basis expansion \eqref{icb1} of
$\mathfrak{B}_2^{(a,b,c)}$. Next, we want to list some properties about $Z^{(a,b,c)}_{m,n,\theta,y}$.

\begin{lem} \label{lemrecursiveoZ}
When $b\geq a+c$, for any $m,n,a,b,c\in \mathbb{Z}_{\geq 0}$, $0\leq \theta\leq a$, $0\leq y\leq c$, we have 
$$
Z^{(a-1,b,c+1)}_{m+1,n-1,\theta,y}=Z^{(a,b,c)}_{m,n,\theta,y},
\qquad
Z^{(a,b,c)}_{m,n,a,y}=Z^{(a,b+1,c+1)}_{m+1,n,a,y},
\qquad
Z^{(a,b,c)}_{m,n,\theta,y}=Z^{(\theta,b-a-c+\theta+y,y)}_{m-c+y,n-a+\theta,\theta,y},
$$
and $Z^{a,b,c}_{m,n,\bm{\varsigma}}=\qbinom{m}{c}\qbinom{m+b-a-c}{c}^{-1}\Xi^{c,a+c,a}_{n-b+c+a,m+b-a-c,\bm{\varsigma}'}$, where $\bm{\varsigma}'=1-\bm{\varsigma}$.
\end{lem}
\begin{proof}
The proofs are identical to those of Lemma \ref{lemrecursive}, and \eqref{oandomega} and are therefore omitted.
\end{proof}

From Lemma \ref{lemrecursiveoZ}, if $b\geq a+c$, we have $Z^{(a,b,c)}_{m,n}\in q^{-1}\mathbb{Z}[q^{-1}]$. We are thus led to the following theorem:
\begin{thm}\label{thm688}
Let $m,n,a,b,c\in \mathbb{Z}_{\geq 0}$ such that $m\geq c$, $n-b+c\geq 0$, $\beta_{a,b,c,\bm{\varsigma}}>0$, $\beta_{a,b,c,\bm{\varsigma}}$ is odd, and $b\geq a+c$, then
\begin{align}\label{generalabc55odd}
B^{(a)}_{1}B^{(b)}_{2}B^{(c)}_{1}\eta&=[\beta_{a,b,c,\bm{\varsigma}}]\sum_{\substack{k \ge 0 \\ k\equiv \beta_{a,b,c,\bm{\varsigma}} \pmod 2}}\frac{[\beta_{a,b,c,\bm{\varsigma}}-k+2][\beta_{a,b,c,\bm{\varsigma}}-k+4]\cdots [\beta_{a,b,c,\bm{\varsigma}}+k-2]}{[k]!}\mathfrak{B}_{2}^{(a-k,b-k,c)}\nonumber\\
&\quad\quad+\sum_{\substack{k \ge 0 \\ k\equiv \beta_{a,b,c,\bm{\varsigma}}+1 \pmod 2}}\frac{[\beta_{a,b,c,\bm{\varsigma}}-k+1][\beta_{a,b,c,\bm{\varsigma}}-k+3]\cdots [\beta_{a,b,c,\bm{\varsigma}}+k-1]}{[k]!}\mathfrak{B}_{2}^{(a-k,b-k,c)},
\end{align}
\begin{align}\label{generalabc66odd}
\mathfrak{B}_{2}^{(a,b,c)}=\sum_{0\leq k\leq a}(-1)^{k}[\beta_{a,b,c,\bm{\varsigma}}+k-1]\frac{[\beta_{a,b,c,\bm{\varsigma}}+1][\beta_{a,b,c,\bm{\varsigma}}+3]\cdots [\beta_{a,b,c,\bm{\varsigma}}+2k-3]}{[k]!}B_{1}^{(a-k)}B^{(b-k)}_{2}B^{(c)}_{1}\eta. 
\end{align}
\end{thm}

\subsection{Case for $B^{(a)}_{2}B^{(b)}_{1}B^{(c)}_{2}\eta$ and $\mathfrak{B}_{1}^{(a,b,c)}$}
Define $\hat{\alpha}_{b,c,\bm{\varsigma}}:=n-m+b-2c+\varsigma_{1}$, $\hat{\beta}_{a,b,c,\bm{\varsigma}}:=m-n+a-2b+3c+\varsigma_{2}$. Again, we focus on the case $\hat{\alpha}_{b,c,\bm{\varsigma}}>0$ XOR $\hat{\beta}_{a,b,c,\bm{\varsigma}}>0$. By symmetry, the results in previous subsection can be extended to the following:
\begin{thm}\label{thmm61010}
Let $m,n,a,b,c\in \mathbb{Z}_{\geq 0}$ such that $n\geq c$, $m-b+c\geq 0$, $\hat{\alpha}_{b,c,\bm{\varsigma}}>0$, $\hat{\alpha}_{b,c,\bm{\varsigma}}$ is odd, and $b\geq a+c$, then
\begin{align*}
B^{(a)}_{2}B^{(b)}_{1}B^{(c)}_{2}\eta&=[\hat{\alpha}_{b,c,\bm{\varsigma}}]\sum_{\substack{k \ge 0 \\ k\equiv \hat{\alpha}_{b,c,\bm{\varsigma}} \pmod 2}}\frac{[\hat{\alpha}_{b,c,\bm{\varsigma}}-k+2]\cdots [\hat{\alpha}_{b,c,\bm{\varsigma}}+k-4][\hat{\alpha}_{b,c,\bm{\varsigma}}+k-2]}{[k]!}\mathfrak{B}_{1}^{(a,b-k,c-k)}\\
&\quad\quad+\sum_{\substack{k\geq 0 \\ k\equiv \hat{\alpha}_{b,c,\bm{\varsigma}}+1 \pmod 2}}\frac{[\hat{\alpha}_{b,c,\bm{\varsigma}}-k+1][\hat{\alpha}_{b,c,\bm{\varsigma}}-k+3]\cdots [\hat{\alpha}_{b,c,\bm{\varsigma}}+k-1]}{[k]!}\mathfrak{B}^{(a,b-k,c-k)}_{1},
\end{align*}
and
\begin{align*}
\mathfrak{B}_{1}^{(a,b,c)}&=\sum_{0\leq k\leq c}(-1)^{k}[\hat{\alpha}_{b,c,\bm{\varsigma}}+k-1]\frac{[\hat{\alpha}_{b,c,\bm{\varsigma}}+1][\hat{\alpha}_{b,c,\bm{\varsigma}}+3]\cdots [\hat{\alpha}_{b,c,\bm{\varsigma}}+2k-3]}{[k]!}B_{2}^{(a)}B^{(b-k)}_{1}B^{(c-k)}_{2}\eta.
\end{align*} 
\end{thm}

\begin{thm}\label{thm61212}
Let $m,n,a,b,c\in \mathbb{Z}_{\geq 0}$ such that $n\geq c$, $m-b+c\geq 0$, $\hat{\beta}_{a,b,c,\bm{\varsigma}}>0$, $\hat{\beta}_{a,b,c,\bm{\varsigma}}$ is odd, and $b\geq a+c$, then
\begin{align*}
B^{(a)}_{2}B^{(b)}_{1}B^{(c)}_{2}\eta&=[\hat{\beta}_{a,b,c,\bm{\varsigma}}]\sum_{\substack{k \ge 0 \\ k\equiv \hat{\beta}_{a,b,c,\bm{\varsigma}} \pmod 2}}\frac{[\hat{\beta}_{a,b,c,\bm{\varsigma}}-k+2]\cdots [\hat{\beta}_{a,b,c,\bm{\varsigma}}+k-4][\hat{\beta}_{a,b,c,\bm{\varsigma}}+k-2]}{[k]!}\mathfrak{B}_{1}^{(a-k,b-k,c)}\\
&\quad\quad+\sum_{\substack{k\geq 0 \\ k\equiv \hat{\beta}_{a,b,c,\bm{\varsigma}}+1 \pmod 2}}\frac{[\hat{\beta}_{a,b,c,\bm{\varsigma}}-k+1][\hat{\beta}_{a,b,c,\bm{\varsigma}}-k+3]\cdots [\hat{\beta}_{a,b,c,\bm{\varsigma}}+k-1]}{[k]!}\mathfrak{B}_{1}^{(a-k,b-k,c)},
\end{align*}
and
\begin{align*}
\mathfrak{B}_{1}^{(a,b,c)}&=\sum_{0\leq k\leq a}(-1)^{k}[\hat{\beta}_{a,b,c,\bm{\varsigma}}+k-1]\frac{[\hat{\beta}_{a,b,c,\bm{\varsigma}}+1][\hat{\beta}_{a,b,c,\bm{\varsigma}}+3]\cdots [\hat{\beta}_{a,b,c,\bm{\varsigma}}+2k-3]}{[k]!}B_{2}^{(a-k)}B^{(b-k)}_{1}B^{(c)}_{2}\eta.
\end{align*}
\end{thm}

\section{iCanonical basis vs monomial basis on algebra level}\label{section7}
In this section, we first extend the results obtained in the previous sections on the transition matrices between the icanonical basis and the monomial basis at the module level to the algebra level. Next, we show that the icanonical basis is preserved by an involution.

\subsection{iCanonical basis vs monomial basis on $\dot{\U}^{\imath}$}
Recall $L(m,n)$ is the simple $\U$-module with the highest weight $m\omega_{1}+n\omega_{2}$ and highest weight vector $\eta$, for $m,n\geq 0$. Further recall that for any $\zeta\in X^{\imath}$, the icanonical basis on $\dot{\U}^{\imath}1_{\zeta}$ is the set $\{\mathfrak{B}_{2}^{(a,b,c)}1_{\zeta}, \mathfrak{B}_{1}^{(a,b,c)}1_{\zeta}: b\geq a+c\}$ modulo the identification $\mathfrak{B}_{2}^{(a,b,c)}1_{\zeta}=\mathfrak{B}_{1}^{c,b,a}1_{\zeta}$ when $b=a+c$.

For any fixed $\zeta\in X^{\imath}=\mathbb{Z}$, we define 
\begin{align}\label{defabalg}
\alpha^{b,c}_{\zeta, \bm{\varsigma},i}:=(-1)^{i}\zeta+b-2c+\varsigma_{i}\quad\text{and}\quad \beta^{a,b,c}_{\zeta, \bm{\varsigma},i}:=(-1)^{\tau i}\zeta+a-2b+3c+\varsigma_{\tau i}.
\end{align}

For any $\zeta\in X^{\imath}$, letting $m,n\to \infty$, $m-n=\zeta$, the relations between $\mathfrak{B}^{(a,b,c)}_{\tau i}$ and $B^{(a)}_{i}B^{(b)}_{\tau i}B^{(c)}_{i}$ on $L(m,n)$ can be generalized to the modified iquantum group $\dot{\U}^{\imath}1_{\zeta}$. 
\begin{thm}\label{thm711}
When $\alpha^{b,c}_{\zeta, \bm{\varsigma},i}\leq 0$, $\beta^{a,b,c}_{\zeta, \bm{\varsigma},i}\leq 0$, and $b\geq a+c$, we get $B^{(a)}_{\tau i}B^{(b)}_{i}B^{(c)}_{\tau i}1_{\zeta}=\mathfrak{B}^{(a,b,c)}_{i}$.
\end{thm}
\begin{proof}
By imposing the condition $m,n\to \infty$ while $m-n=\zeta$ to Proposition \ref{prop32}, we obtain the desired result.
\end{proof}

\begin{thm}\label{thm722}
When $\alpha^{b,c}_{\zeta, \bm{\varsigma},i}>0$ and $b\geq a+c$, we get
\begin{align*}
B^{(a)}_{\tau i}B^{(b)}_{i}B^{(c)}_{\tau i}1_{\zeta}&=[\alpha^{b,c}_{\zeta, \bm{\varsigma},i}]\sum_{\substack{k \ge 0 \\ k\equiv \alpha^{b,c}_{\zeta, \bm{\varsigma},i} \pmod 2}}\frac{[\alpha^{b,c}_{\zeta, \bm{\varsigma},i}-k+2]\cdots [\alpha^{b,c}_{\zeta, \bm{\varsigma},i}+k-4][\alpha^{b,c}_{\zeta, \bm{\varsigma},i}+k-2]}{[k]!}\mathfrak{B}_{i}^{a,b-k,c-k}\\
&\quad\quad+\sum_{\substack{k\geq 0 \\ k\equiv \alpha^{b,c}_{\zeta, \bm{\varsigma},i}+1 \pmod 2}}\frac{[\alpha^{b,c}_{\zeta, \bm{\varsigma},i}-k+1][\alpha^{b,c}_{\zeta, \bm{\varsigma},i}-k+3]\cdots [\alpha^{b,c}_{\zeta, \bm{\varsigma},i}+k-1]}{[k]!}\mathfrak{B}^{a,b-k,c-k}_{i}.
\end{align*}
Furthermore, if $\alpha^{b,c}_{\zeta, \bm{\varsigma},i}$ is even, then
\begin{align}\label{theorem721}
\mathfrak{B}_{i}^{(a,b,c)}&=\sum_{0\leq k\leq c}(-1)^{k}\frac{[\alpha^{b,c}_{\zeta, \bm{\varsigma},i}][\alpha^{b,c}_{\zeta, \bm{\varsigma},i}+2]\cdots [\alpha^{b,c}_{\zeta, \bm{\varsigma},i}+2k-2]}{[k]!}B^{(a)}_{\tau i}B^{(b-k)}_{i}B^{(c-k)}_{\tau i}1_{\zeta}.
\end{align} 
If $\alpha^{b,c}_{\zeta, \bm{\varsigma},i}$ is odd, then
\begin{align}\label{theorem722}
\mathfrak{B}_{i}^{(a,b,c)}&=\sum_{0\leq k\leq c}(-1)^{k}[\alpha^{b,c}_{\zeta, \bm{\varsigma},i}+k-1]\frac{[\alpha^{b,c}_{\zeta, \bm{\varsigma},i}+1][\alpha^{b,c}_{\zeta, \bm{\varsigma},i}+3]\cdots [\alpha^{b,c}_{\zeta, \bm{\varsigma},i}+2k-3]}{[k]!}B^{(a)}_{\tau i}B^{(b-k)}_{i}B^{(c-k)}_{\tau i}1_{\zeta}.
\end{align} 
\end{thm}
\begin{proof}
By imposing the condition $m,n\to \infty$ while $m-n=\zeta$ to Theorems \ref{aeqb+c}, \ref{thm5144}, \ref{thm6155}, and \ref{thmm61010}, we obtain the desired result.
\end{proof}

\begin{thm}\label{thm744}
When $\beta^{a,b,c}_{\zeta, \bm{\varsigma},i}>0$ and $b\geq a+c$, we get
\begin{align*}
B^{(a)}_{\tau i}B^{(b)}_{i}B^{(c)}_{\tau i}1_{\zeta}&=[\beta^{a,b,c}_{\zeta, \bm{\varsigma},i}]\sum_{\substack{k \ge 0 \\ k\equiv \beta^{a,b,c}_{\zeta, \bm{\varsigma},i} \pmod 2}}\frac{[\beta^{a,b,c}_{\zeta, \bm{\varsigma},i}-k+2]\cdots [\beta^{a,b,c}_{\zeta, \bm{\varsigma},i}+k-4][\beta^{a,b,c}_{\zeta, \bm{\varsigma},i}+k-2]}{[k]!}\mathfrak{B}_{i}^{a-k,b-k,c}\\
&\quad\quad+\sum_{\substack{k\geq 0 \\ k\equiv \beta^{a,b,c}_{\zeta, \bm{\varsigma},i}+1 \pmod 2}}\frac{[\beta^{a,b,c}_{\zeta, \bm{\varsigma},i}-k+1][\beta^{a,b,c}_{\zeta, \bm{\varsigma},i}-k+3]\cdots [\beta^{a,b,c}_{\zeta, \bm{\varsigma},i}+k-1]}{[k]!}\mathfrak{B}_{i}^{a-k,b-k,c}.
\end{align*}
Furthermore, if $\beta^{a,b,c}_{\zeta, \bm{\varsigma},i}$ is even, then
\begin{align}\label{theorem741}
\mathfrak{B}_{i}^{(a,b,c)}&=\sum_{0\leq k\leq a}(-1)^{k}\frac{[\beta^{a,b,c}_{\zeta, \bm{\varsigma},i}][\beta^{a,b,c}_{\zeta, \bm{\varsigma},i}+2]\cdots [\beta^{a,b,c}_{\zeta, \bm{\varsigma},i}+2k-2]}{[k]!}B_{\tau i}^{(a-k)}B^{(b-k)}_{i}B^{(c)}_{\tau i}1_{\zeta}.
\end{align}
If $\beta^{a,b,c}_{\zeta, \bm{\varsigma},i}$ is odd, then
\begin{align}\label{theorem742}
\mathfrak{B}_{i}^{(a,b,c)}&=\sum_{0\leq k\leq a}(-1)^{k}[\beta^{a,b,c}_{\zeta, \bm{\varsigma},i}+k-1]\frac{[\beta^{a,b,c}_{\zeta, \bm{\varsigma},i}+1][\beta^{a,b,c}_{\zeta, \bm{\varsigma},i}+3]\cdots [\beta^{a,b,c}_{\zeta, \bm{\varsigma},i}+2k-3]}{[k]!}B_{\tau i}^{(a-k)}B^{(b-k)}_{i}B^{(c)}_{\tau i}1_{\zeta}.
\end{align}
\end{thm}
\begin{proof}
By imposing the condition $m,n\to \infty$ while $m-n=\zeta$ to Theorems \ref{thmm512}, \ref{thm51616}, \ref{thm688}, and \ref{thm61212}, we obtain the desired result.  
\end{proof}

\subsection{Stability of icanonical basis}
%In this subsection, denote $L(m_{1},m_{2})$ as the simple $\U$-module with the highester weight $m_{1}\omega_{1}+m_{2}\omega_{2}$ and highest weight vector $\eta$. 
To avoid confusion, recall that for any $\zeta\in X^{\imath}$ and $m_{1},m_{2}\in \mathbb{Z}_{\geq 0}$, we denote the icanonical basis on $\dot{\U}^{\imath}1_{\zeta}$ as $\mathfrak{B}_{i,\zeta}^{(a,b,c)}$ and the icanonical basis on $L(m_{1},m_{2})$ as $\mathfrak{B}_{i,m_{1},m_{2}}^{(a,b,c)}$. 

The following theorem unravel the stability of icanonical basis, that is, upon viewing $L(m,n)$ as a $\dot{\U}^{\imath}$-module, the action of the icanonical basis of $\dot{\U}^{\imath}$ on $\eta$ yields either an icanonical basis element of $L(m,n)$ or zero.

\begin{thm}\label{thmstability}
For any $\zeta\in X^{\imath}$, $m_{1},m_{2}\in \mathbb{Z}_{\geq 0}$ with $m_{1}-m_{2}=\zeta$, we have the following identity on $L(m_{1},m_{2})$:
\begin{align}
\mathfrak{B}_{i,\zeta}^{(a,b,c)}\eta=\mathfrak{B}_{i,m_{1},m_{2}}^{(a,b,c)}~\text{or 0}.
\end{align}
Furthermore, $\mathfrak{B}_{i}^{(a,b,c)}\eta\not=0$ if and only if $c\leq m_{\tau i}$ and $a\leq b-c\leq m_{i}$.
\end{thm}
\begin{proof}
We restrict our attention to the case $i=2$, and the case for $i=1$ follows by symmetry. Recalling \eqref{defabalg}, we have
$$
\alpha^{b,c}_{\zeta, \bm{\varsigma},2}=m_{1}-m_{2}+b-2c+\varsigma_{2}\quad\text{and}\quad \beta^{a,b,c}_{\zeta, \bm{\varsigma},i}=m_{2}-m_{1}+a-2b+3c+\varsigma_{1}.
$$
We now divide the argument into three cases.

\textbf{Case 1:
When $\alpha^{b,c}_{\zeta, \bm{\varsigma},2}\leq 0$, $\beta^{a,b,c}_{\zeta, \bm{\varsigma},2}\leq 0$.} From Theorem \ref{thm711}, $\mathfrak{B}_{2,\zeta}^{(a,b,c)}\eta=B^{(a)}_{1}B^{(b)}_{2}B^{(c)}_{1}\eta\not=0$ if and only if $c\leq m_{1}$ and $a\leq b-c\leq m_{2}$ (cf. \cite[Theorem 5.12]{WZ25}). 

\textbf{Case 2: 
When $\alpha^{b,c}_{\zeta, \bm{\varsigma},2}>0$ and $b\geq a+c$.} Apply $\eta$ on both sides of \eqref{theorem721}--\eqref{theorem722}, we first know $B^{(a)}_{1}B^{(b-k)}_{2}B^{(c-k)}_{1}\eta\not=0$ if $c\leq m_{1}$ and $a\leq b-c\leq m_{2}$, for each $0\leq k\leq c$. Hence, it remains to show that, if $c>m_{1}$ or $b-c>m_{2}$, $B^{(a)}_{1}B^{(b-k)}_{2}B^{(c-k)}_{1}\eta=0$. If $b-c>m_{2}$, $B^{(a)}_{1}B^{(b-k)}_{2}B^{(c-k)}_{1}\eta=0$ follows from \cite[Theorem 5.12]{WZ25}. Next, we prove that the conditions $c>m_{1}$ and $b-c\leq m_{2}$ are not compatible. On one hand, $b-c\leq m_{2}$ tells us that
$$
m_{1}-m_{2}+b-2c\leq m_{1}-(b-c)+b-2c=m_{1}-c.
$$
If we further have $c>m_{1}$, it will contradict with the condition $\alpha^{b,c}_{\zeta, \bm{\varsigma},2}>0$.

\textbf{Case 3: 
When $\beta^{a,b,c}_{\zeta, \bm{\varsigma},2}>0$ and $b\geq a+c$.} Applying both sides of \eqref{theorem741}--\eqref{theorem742} to $\eta$, we first know $B^{(a-k)}_{1}B^{(b-k)}_{2}B^{(c)}_{1}\eta\not=0$ if $c\leq m_{1}$ and $a\leq b-c\leq m_{2}$, for each $0\leq k\leq a$. Hence, it remains to show that, if $c>m_{1}$ or $b-c>m_{2}$, $B^{(a-k)}_{1}B^{(b-k)}_{2}B^{(c)}_{1}\eta=0$. If $c>m_{1}$ or $b-a-c>m_{2}$, then $B^{(a-k)}_{1}B^{(b-k)}_{2}B^{(c)}_{1}\eta=0$ follows from \cite[Theorem 5.12]{WZ25}. Next, we prove that the conditions $c\leq m_{1}$, $b-c>m_{2}$, $b-a-c\leq m_{2}$ are not compatible. From $b-c>m_{2}$, $b-a-c\leq m_{2}$, we know $b-a-c\leq m_{2}\leq b-c-1$. Hence,
$$
m_{2}-m_{1}+a-2b+3c\leq (b-c-1)-m_{1}+a-2b+3c=-b+a+2c-1-m_{1}
$$
If we further have $c\leq m_{1}$, we get $m_{2}-m_{1}+a-2b+3c\leq -b+a+2c-1-c=a+c-b-1$, it will contradicts to the condition $\beta^{a,b,c}_{\zeta, \bm{\varsigma},2}>0$.

\end{proof}

\subsection{Symmetry on $\dot{\U}^{\imath}$}

First, we introduce an anti-involution $\sigma_{\tau}$ on $\dot{\U}^{\imath}$:
\begin{lem}\label{tau}
There exists an anti-involution $\sigma_{\tau}$ on $\dot{\U}^{\imath}$ such that $\sigma_{\tau}(B_{i}1_{\zeta})=1_{\zeta}B_{\tau i}$, for $i\in I$, $\zeta\in X^{\imath}$.
\end{lem}
\begin{proof} 
It follows by composing \cite[(3.8)]{BWW25} and \cite[(3.10)]{BWW25}. This is a modified version of the anti-involution given in \cite{WZ23}. 
\end{proof}

With the above preparations, we can now present the main theorem of this section.
\begin{thm}
Let $\sigma_{\tau}$ be the anti-involution on $\dot{\U}^{\imath}_{\bm{\varsigma}}$ defined in Lemma \ref{tau}, then $\sigma_{\tau}(\mathfrak{B}^{(a,b,c)}_{2}1_{\zeta})=\mathfrak{B}^{(c,b,a)}_{1}1_{\zeta+3b-3a-3c}$ for any fixed $\zeta\in X^{\imath}$.
\end{thm}
\begin{proof}
First, we focus on the even parity case. When $\alpha^{b,c}_{\zeta, \bm{\varsigma},2}=\zeta+b-2c+\varsigma_{2}>0$,  $\alpha^{b,c}_{\zeta, \bm{\varsigma},2}$ is even, and $b\geq a+c$, we compute 
\begin{align*}
\sigma_{\tau}(\mathfrak{B}^{(a,b,c)}_{2}1_{\zeta})&=\sum_{0\leq k\leq c}(-1)^{k}\frac{[\alpha^{b,c}_{\zeta, \bm{\varsigma},2}][\alpha^{b,c}_{\zeta, \bm{\varsigma},2}+2]\cdots [\alpha^{b,c}_{\zeta, \bm{\varsigma},2}+2k-2]}{[k]!}1_{\zeta}B_{2}^{(c-k)}B^{(b-k)}_{1}B^{(a)}_{2}\\
&=\sum_{0\leq k\leq c}(-1)^{k}\frac{[\beta^{c,b,a}_{\zeta+3b-3a-3c, \bm{\varsigma},1}][\beta^{c,b,a}_{\zeta+3b-3a-3c, \bm{\varsigma},1}+2]\cdots [\beta^{c,b,a}_{\zeta+3b-3a-3c, \bm{\varsigma},1}+2k-2]}{[k]!}\\
&\quad\quad\quad\quad\times B_{2}^{(c-k)}B^{(b-k)}_{1}B^{(a)}_{2}1_{\zeta+3b-3a-3c}.\\
&=\mathfrak{B}^{(c,b,a)}_{1}1_{\zeta+3b-3a-3c}.
\end{align*}
Here we use the fact that $\alpha^{b,c}_{\zeta, \bm{\varsigma},2}=\zeta+b-2c+\varsigma_{2}=(\zeta+3b-3a-3c)-2b+3a+c+\varsigma_{2}=\beta^{c,b,a}_{\zeta+3b-3a-3c, \bm{\varsigma},1}$. Meanwhile, when $\beta^{a,b,c}_{\zeta, \bm{\varsigma},2}=-\zeta+a-2b+3c+\varsigma_{1}>0$, and $\beta^{a,b,c}_{\zeta, \bm{\varsigma},2}$ is even, and $b\geq a+c$, we compute
\begin{align*}
\sigma_{\tau}(\mathfrak{B}^{(a,b,c)}_{2}1_{\zeta})&=\sum_{0\leq k\leq a}(-1)^{k}\frac{[\beta^{a,b,c}_{\zeta, \bm{\varsigma},2}][\beta^{a,b,c}_{\zeta, \bm{\varsigma},2}+2]\cdots [\beta^{a,b,c}_{\zeta, \bm{\varsigma},2}+2k-2]}{[k]!}1_{\zeta}B_{2}^{(c)}B^{(b-k)}_{1}B^{(a-k)}_{2}\\
&=\sum_{0\leq k\leq a}(-1)^{k}\frac{[\alpha^{b,a}_{\zeta+3b-3a-3c, \bm{\varsigma},1}][\alpha^{b,a}_{\zeta+3b-3a-3c, \bm{\varsigma},1}+2]\cdots [\alpha^{b,a}_{\zeta+3b-3a-3c, \bm{\varsigma},1}+2k-2]}{[k]!}\\
&\quad\quad\quad\quad\times B_{2}^{(c)}B^{(b-k)}_{1}B^{(a-k)}_{2}1_{\zeta+3b-3a-3c}\\
&=\mathfrak{B}^{(c,b,a)}_{1}1_{\zeta+3b-3a-3c}.
\end{align*}
Here we use the fact that $\beta^{a,b,c}_{\zeta, \bm{\varsigma},2}=-\zeta+a-2b+3c+\varsigma_{1}=-(\zeta+3b-3a-3c)+b-2a+\varsigma_{1}=\alpha^{b,a}_{\zeta+3b-3a-3c, \bm{\varsigma},1}$. The proofs of the remaining cases are entirely analogous and are therefore omitted.
\end{proof}

\section{iCanonical basis to canonical basis}\label{section8}
In this section, we first derive a new closed formula for the coefficients appearing in the transition matrices from the icanonical basis to Lusztig's canonical basis on the simple $\U$-module $L(m,n)$ under a special case, where the positivity of the coefficients manifest. We then summarize the closed formulas for the coefficients in the transition matrices from the icanonical basis to Lusztig's canonical basis on $L(m,n)$ obtained in Sections \ref{section5}--\ref{section6}.
\subsection{Case for $a=0$ or $c=0$} 
For a positive integer $\alpha$, define $d_0(\alpha)=1$ and, for $k\geq1$,
\begin{align}\label{sec8-d}
d_k(\alpha)=\begin{cases}
\displaystyle\frac{(-1)^k}{[k]!}\prod_{j=0}^{k-1}[\alpha+2j]&\text{ when }\alpha\text{ is even},\\[4pt]
\displaystyle\frac{(-1)^k[\alpha+k-1]}{[k]!}\prod_{j=0}^{k-2}[\alpha+1+2j]&\text{ when }\alpha\text{ is odd}.
\end{cases}
\end{align}
Suppose $b\geq c$, $m\geq c$, $n-b+c\geq0$, and $\alpha_{b,c,\bm{\varsigma}}=m-n+b-2c+\varsigma_2>0$. Define
\begin{align}\label{sec8-T}
T^{(b,c),\bm{\varsigma}}_{m,n,\theta}
=\sum_{k+y=\theta}d_k(\alpha_{b,c,\bm{\varsigma}})
q^{\varsigma_2 y-y(y+1+2n+2c-2b)/2}\qbinom{m-c+\theta}{y},
\qquad 0\leq\theta\leq c.
\end{align}

Meanwhile, recall $\eqref{defomega}$, $\mathfrak{B}^{(0,b,c)}_{2}$ is the icanonical basis in simple $\U$-module $L(m,n)$ such that $\mathfrak{B}^{(0,b,c)}_{2}$ is $\psi^{\imath}$-invariant and
\begin{align}\label{811}
\mathfrak{B}^{(0,b,c)}_{2}=F^{(b)}_{2}F^{(c)}_{1}\eta+\sum_{0<\theta\leq c}T^{(b,c),\bm{\varsigma}}_{m,n,\theta}F^{(b-\theta)}_{2}F^{(c-\theta)}_{1}\eta,
\end{align}
where $T^{(b,c),\bm{\varsigma}}_{m,n,\theta}\in q^{-1}\mathbb{Z}[q^{-1}]$. Recall the expressions \eqref{deffff}, \eqref{deffffo} and Theorem \ref{thm41717}, we know 
\[
T^{(b,c),\bm{\varsigma}}_{m,n,\theta}
=\begin{cases}
\Omega^{(b,c),\bm{\varsigma}}_{m,n,\theta}&\text{ when }\alpha_{b,c,\bm{\varsigma}}\text{ is even},\\
O^{(b,c),\bm{\varsigma}}_{m,n,\theta}&\text{ when }\alpha_{b,c,\bm{\varsigma}}\text{ is odd}.
\end{cases}
\] In particular, $T^{(b,c),\bm{\varsigma}}_{m,n,0}=1$. Substitution in \eqref{recur1} and \eqref{recur2} gives the full-to-partial reduction
\begin{align}\label{833}
T^{(b,c),\bm{\varsigma}}_{m,n,\theta}
=T^{(b,\theta),\bm{\varsigma}}_{m-c+\theta,n+c-\theta,\theta}.
\end{align}

\begin{lem}\label{sec8-odd-even}
For positive odd $\alpha=m-n+b-2c+\varsigma_2$ and $0\leq\theta\leq c$,
\begin{align}\label{sec8-shift}
O^{(b,c),\bm{\varsigma}}_{m,n,\theta}
=q^\theta\Omega^{(b,c),\bm{\varsigma}}_{m+1,n,\theta}.
\end{align}
\end{lem}

\begin{prop}\label{thm811}
Take $N=n-\varsigma_2$, $\epsilon\in\{0,1\}$ with
$\epsilon\equiv \alpha_{b,c,\bm{\varsigma}}\pmod2$, and $g=(\alpha_{b,c,\bm{\varsigma}}+\epsilon)/2$. Then
\begin{align}\label{844}
T^{(b,c),\bm{\varsigma}}_{m,n,c}
=\sum_{j=0}^{\lfloor c/2\rfloor}
\qbinom{g+j-1}{j}_{q^2}
\qbinom{N-b+2c-2j}{c-2j}
q^{\frac{(c-2j)(c-1-2m)}2+j(-m+N-b+c+1+\epsilon)}.
\end{align}
Together with \eqref{833}, this determines every coefficient in
\eqref{811}. 
\end{prop}

\begin{rem}
In particular, $T_{m,n,0}^{(b,c),\bm{\varsigma}}=1$ and
$T_{m,n,\theta}^{(b,c),\bm{\varsigma}}\in q^{-1}\mathbb Z_{\geq0}[q^{-1}]$ for $\theta>0$ follows from \eqref{833} and \eqref{844}.
\end{rem}

\begin{proof}
The long and straightforward proof for Proposition \ref{thm811} when
$\bm{\varsigma}=(1,0)$ and $\alpha_{b,c,\bm{\varsigma}}$ is even can be found in Appendix \ref{app8}. The proofs of the rest cases are entirely similar.
\end{proof}

\begin{rem}
In particular, for $\bm{\varsigma}=(1,0)$ and positive odd $\alpha$, the odd formula is
\[
O^{(b,c),\bm{\varsigma}}_{m,n,c}
=\sum_{j=0}^{\lfloor c/2\rfloor}
\qbinom{(\alpha_{b,c,\bm{\varsigma}}+1)/2+j-1}{j}_{q^2}
\qbinom{n-b+2c-2j}{c-2j}
q^{\frac{(c-2j)(c-1-2m)}2+j(-m+n-b+c+2)}.
\]
\end{rem}

\subsection{General three-factor transition coefficients}
Assume throughout this subsection that
\[
a,b,c,m,n\geq0,\qquad b\geq a+c,\qquad m\geq c,\qquad n\geq b-c.
\]
Put $\alpha=m-n+b-2c+\varsigma_2$ and
$\beta=n-m+a-2b+3c+\varsigma_1$. Since
$\alpha+\beta=1+a-b+c\leq1$, at most one of $\alpha,\beta$ is
positive. Define the actual transition coefficients by
\begin{align}\label{866}
\mathfrak B_2^{(a,b,c)}
=\sum_{x=0}^a\sum_{\theta=0}^c T^{(a,b,c),\bm{\varsigma}}_{m,n,x,\theta}
F_1^{(a-x)}F_2^{(b-x-\theta)}F_1^{(c-\theta)}\eta.
\end{align}

\begin{thm}\label{sec8-general}
The coefficients in \eqref{866} are as follows.
\begin{enumerate}
\item If $\alpha>0$, then
\begin{align}\label{sec8-alpha-factor}
T^{(a,b,c),\bm{\varsigma}}_{m,n,x,\theta}&=P_{x,\theta}T^{(b,c),\bm{\varsigma}}_{m,n,\theta},\\
P_{x,\theta}&=q^{-\varsigma_2 x-x(x-1+2m+2b-4c-2a+2\theta)/2}
\qbinom{n-b+c+x}{x}.\nonumber
\end{align}
\item If $\beta>0$, put
\[
\widetilde m=n-b+c+a,\quad \widetilde n=m+b-a-c,
\quad \widetilde b=a+c,\quad \widetilde c=a,\quad\bm{\varsigma}^\prime=(\varsigma_2,\varsigma_1).
\]
Then
\begin{align}\label{sec8-beta-factor}
T^{(a,b,c),\bm{\varsigma}}_{m,n,x,\theta}&=R_{x,\theta}T^{(\widetilde b,\widetilde c),\bm{\varsigma}^\prime}_{
\widetilde m,\widetilde n,x},\\
R_{x,\theta}&=q^{-\varsigma_1\theta-\theta(\theta-1+2n+2c-2b+2x)/2}
\qbinom{m-c+\theta}{\theta}.\nonumber
\end{align}
The parameter of the two-factor coefficient is $\beta$.
\item If $\alpha\leq0$ and $\beta\leq0$, then
\begin{align}\label{sec8-nonpositive}
T^{(a,b,c),\bm{\varsigma}}_{m,n,x,\theta}={}&q^{\varsigma_1x+\varsigma_2 \theta-\theta(\theta+1+2n+2c-2b)/2}
q^{-x(x+1+2m+2b-4c-2a+2\theta)/2}\nonumber\\
&\hspace{18mm}\times\qbinom{n-b+c+x}{x}\qbinom{m-c+\theta}{\theta}.
\end{align}
\end{enumerate}
\end{thm}

\begin{proof}
We use the coefficients $d_k$ defined in \eqref{sec8-d}, with $d_0=1$. First suppose $\alpha_{b,c,\bm{\varsigma}}>0$. Theorems \ref{aeqb+c} and \ref{thm6155}, together with the two-factor
formulas when $a=0$ or $c=0$, give
\[
\mathfrak B_2^{(a,b,c)}
=\sum_{k=0}^c d_k(\alpha_{b,c,\bm{\varsigma}})
B_1^{(a)}B_2^{(b-k)}B_1^{(c-k)}\eta,
\]
and
\begin{align*}
T^{(a,b,c),\bm{\varsigma}}_{m,n,x,\theta}
={}&\sum_{k+y=\theta}d_k(\alpha_{b,c,\bm{\varsigma}})q^{\varsigma_2(y-x)}
\qbinom{n-b+c+x}{x}\qbinom{m-c+\theta}{y}\\
&\quad\times q^{-\frac{y(y+1+2n+2c-2b)}2
-\frac{x(x-1+2m+2b-4c-2a+2y+2k)}2}\\
={}&q^{-\varsigma_2x-\frac{x(x-1+2m+2b-4c-2a+2\theta)}2}
\qbinom{n-b+c+x}{x}\\
&\quad\times\sum_{k+y=\theta}d_k(\alpha_{b,c,\bm{\varsigma}})
q^{\varsigma_2y-\frac{y(y+1+2n+2c-2b)}2}
\qbinom{m-c+\theta}{y}\\
={}&P_{x,\theta}T^{(b,c),\bm{\varsigma}}_{m,n,\theta}.
\end{align*}
Here the second equality uses $k+y=\theta$, and the last equality
is precisely \eqref{sec8-T}. This proves \eqref{sec8-alpha-factor}
for both parities of $\alpha_{b,c,\bm{\varsigma}}$.

Next suppose $\beta_{a,b,c,\bm{\varsigma}}>0$. Theorems \ref{thmm512} and \ref{thm688}, with the two-factor cases
included, give
\[
\mathfrak B_2^{(a,b,c)}
=\sum_{k=0}^a d_k(\beta_{a,b,c,\bm{\varsigma}})
B_1^{(a-k)}B_2^{(b-k)}B_1^{(c)}\eta,
\]
and
\begin{align*}
T^{(a,b,c),\bm{\varsigma}}_{m,n,x,\theta}
={}&\sum_{j+k=x}d_k(\beta_{a,b,c,\bm{\varsigma}})q^{\varsigma_1(j-\theta)}
\qbinom{n-b+c+x}{j}\qbinom{m-c+\theta}{\theta}\\
&\quad\times q^{-\frac{\theta(\theta-1+2n+2c-2b+2k)}2
-\frac{j(j+1+2m+2b-4c-2a+2\theta)}2}\\
={}&q^{-\varsigma_1\theta-\frac{\theta(\theta-1+2n+2c-2b+2x)}2}
\qbinom{m-c+\theta}{\theta}\\
&\quad\times\sum_{j+k=x}d_k(\beta_{a,b,c,\bm{\varsigma}})
q^{\varsigma_1j-\frac{j(j+1+2m+2b-4c-2a)}2}
\qbinom{n-b+c+x}{j}.
\end{align*}
Furthermore, one can compute
\begin{align*}
\widetilde m-\widetilde c+x&=n-b+c+x,\\
\widetilde n+\widetilde c-\widetilde b&=m+b-a-2c,\\
\widetilde m-\widetilde n+\widetilde b-2\widetilde c
+\varsigma'_2&=\beta_{a,b,c,\bm{\varsigma}},\qquad \varsigma'_2=\varsigma_1.
\end{align*}
Thus the remaining sum is
$T^{(\widetilde b,\widetilde c),\bm{\varsigma}'}_{\widetilde m,\widetilde n,x}$ by \eqref{sec8-T}, proving \eqref{sec8-beta-factor}. The transformed parameters are admissible: $\widetilde b\geq\widetilde c$, $\widetilde m-\widetilde c=n-b+c\geq0$, and $\widetilde n-\widetilde b+\widetilde c=m+b-a-2c\geq0$.

Finally, if $\alpha_{b,c,\bm{\varsigma}}\leq0$ and $\beta_{a,b,c,\bm{\varsigma}}\leq0$, Proposition \ref{prop32} gives $\mathfrak B_2^{(a,b,c)}=B_1^{(a)}B_2^{(b)}B_1^{(c)}\eta$. Now \eqref{sec8-nonpositive} follows from Lemma \ref{prop59} directly.
\end{proof}

\begin{rem}
In all cases $T^{(a,b,c),\bm{\varsigma}}_{m,n,0,0}=1$ and
$T^{(a,b,c),\bm{\varsigma}}_{m,n,x,\theta}\in q^{-1}\mathbb Z_{\geq0}[q^{-1}]$ for $(x,\theta)\ne(0,0)$.
Equations \eqref{833} and \eqref{844} make the first two formulas
explicit positive finite sums, for either parity of the positive parameter.
\end{rem}

The formulas include $a=0$, $c=0$, and $b=a+c$ without an additional
boundary correction. Interchanging $1$ and $2$, $m$ and $n$, and
$\varsigma_1$ and $\varsigma_2$ gives the corresponding expansions of
$\mathfrak B_1^{(a,b,c)}$.

\appendix

\section{Proofs of several identities and recursive relations}

\subsection{Proofs of $q$-binomial identities}

\subsubsection{Proof of Lemma \ref{qmn}}
\label{app1}
Denote  $Q_{m,n} =\frac{[m+n-1]}{[n]!} \prod_{j=0}^{n-2}[m+1+2j]$, and $m=2t+1$, for $t\in\mathbb Z_{\geq 0}$.
We have
\begin{align*}
	Q_{m,n}&=\frac{[2t+n]\displaystyle\prod_{j=1}^{n-1}[2(t+j)]}{[n]!}=\frac{[2t+n]\displaystyle\prod_{j=1}^{n-1}[t+j]\displaystyle\prod_{j=1}^{n-1}(q^{t+j}+q^{-t-j})}{[n]!}\\
	&=\frac{[2t+n]}{[t+n]}\qbinom{t+n}{n}\prod_{j=1}^{n-1}(q^{t+j}+q^{-t-j}).
\end{align*}
Meanwhile, since 
\begin{align*}
	\frac{[2t+n]}{[t+n]}\qbinom{t+n}{n}&=\frac{q^{-t}[t+n]+q^{t+n}[t]}{[t+n]}\qbinom{t+n}{n}=q^{-t}\qbinom{t+n}{n}+q^{t+n}\qbinom{t-1+n}{n}\in \mathcal{A},
\end{align*}
we get $Q_{m,n}\in \mathcal{A}$ as desired. This completes the proof of Lemma \ref{qmn}.

\subsubsection{Proof of Lemma \ref{lem1}}
\label{app2}
It is enough to prove 
$$
[k][k-m][c]-[k][m+1][c+k-1]=[k-m][k-m-1][c-m-1]-[m][m+1][c+2k-m-1].
$$
Denote $s(x):=q^x-q^{-x}$, the above identity is equivalent to 
\begin{align*}
	s(k)s(k-m)s(c)-&s(k)s(m+1)s(c+k-1)\\
	&=s(k-m)s(k-m-1)s(c-m-1)-s(m)s(m+1)s(c+2k-m-1).
\end{align*}
Now we use the fact that for any $a,b,c\in \mathbb{Z}$, 
$$
s(a)s(b)s(c) = s(a+b+c)-s(a+b-c)-s(a-b+c)+s(a-b-c).
$$
Expanding both sides will give us the desired result. This proves Lemma \ref{lem1}.

\subsection{Proof of Lemma \ref{lemrecur}}\label{applemrecur}
Recall \eqref{deffff}, we have 
\begin{align*}
\Omega^{(b,c+1)}_{m,n,\theta}&=\sum_{0\leq k\leq \theta}(-1)^{k}\frac{[m-n+b-2c-2][m-n+b-2c]\cdots [m-n+b-2c+2k-4]}{[k]!}\\
&\quad\quad\times q^{-\frac{(\theta-k)(\theta-k+3+2n+2c-2b)}{2}}\qbinom{m-c-1+\theta-k}{\theta-k}\\
&=\sum_{0\leq k\leq \theta}(-1)^{k}\frac{[(m-1)-(n+1)+b-2c]\cdots [(m-1)-(n+1)+b-2c+2k-2]}{[k]!}\\
&\quad\quad\times q^{-\frac{(\theta-k)(\theta-k+1+2(n+1)+2c-2b)}{2}}\qbinom{(m-1)-c+\theta-k}{\theta-k}\\
&=\Omega^{(b,c)}_{m-1,n+1,\theta}.
\end{align*}
This proves Lemma \ref{lemrecur}.

\subsection{Proof of Lemma \ref{lemrecur2}}\label{app4}
%\begin{proof}
We compute 
\begin{align*}
	\Omega^{(b,c)}_{m,n}&=\sum_{k+y=c}(-1)^{k}\frac{[\alpha_{b,c}][\alpha_{b,c}+2]\cdots [\alpha_{b,c}+2k-2]}{[k]!}q^{-\frac{y(y+1+2n+2c-2b)}{2}}\frac{[m]!}{[y]![m-y]!}\\
	&=\sum_{d=0}^{c}(-1)^{d}\frac{[\alpha_{b,c}][\alpha_{b,c}+2]\cdots [\alpha_{b,c}+2d-2]}{[c-d]![d]!}\\
	&\quad\quad \times \Big(q^{-(1+c+n-b)}[m-c+d+1]\Big)\Big(q^{-(2+c+n-b)}[m-c+d+2]\Big)\cdots \Big(q^{-(2c+n-b-d)}[m]\Big).
\end{align*}
Meanwhile, denote the sequence $\{\epsilon_{n}\}_{1\leq n\leq c}$ such that $\epsilon_{1}=1$, 
$$
\epsilon_{n+1}=q^{-2n}\left(\epsilon_{n}+(-1)^{n}q^{n(c-n)}\qbinom{c}{n}\right).
$$
It is direct to show that $\epsilon_{n}=(-1)^{n-1}q^{(n-1)(c-n)}\qbinom{c-1}{n-1}$. %This fact follows from \cite[0.2.2]{Jan95}. 
Now we have
\begin{align*}
	\Omega^{(b,c)}_{m,n}&=\frac{1}{[c]!}\sum_{d=1}^{c}q^{-(n-b+2c)}[m]q^{-(n-b+2c-1)}[m-1]\cdots q^{-(n-b+c+1+d)}[m-c+1+d]\\
	&\quad\quad\times [\alpha_{b,c}][\alpha_{b,c}+2]\cdots [\alpha_{b,c}+2(d-2)]q^{-m+c-d}[n-b+c+2-d]\epsilon_{d}\\
	&=\frac{q^{-m+c-1}}{[c]}\sum_{d=1}^{c}\frac{1}{[c-1]!}q^{-(n-b+2c)}[m]q^{-(n-b+2c-1)}[m-1]\cdots q^{-(n-b+c+1+d)}[m-c+1+d]\\
	&\quad\quad\times [\alpha_{b,c}][\alpha_{b,c}+2]\cdots [\alpha_{b,c}+2(d-2)]q^{-d+1}[n-b+c+2-d]\epsilon_{d}\\
	&=\frac{q^{-m+c-1}}{[c]}\sum_{d=1}^{c}\frac{q^{-(n-b+2c-d+1)}[m]q^{-(n-b+2c-d)}[m-1]\cdots q^{-(n-b+c+2)}[m-c+1+d]}{[c-d]![d-1]!}\\
	&\quad\quad\times [\alpha_{b,c}][\alpha_{b,c}+2]\cdots [\alpha_{b,c}+2(d-2)]q^{-d+1}[n-b+c+2-d]
\end{align*}
Denote $m'=m+1$, $n'=n+1$,  
\begin{align*}
	\Omega^{(b,c)}_{m,n}&=\frac{q^{-m'+c}}{[c]}\Omega^{(b,c)}_{m',n',c-1}[n-b+c+1]\\
	&\quad\quad+\frac{q^{-m'+c}}{[c]}[\alpha_{b,c}]\bigg(\sum_{d=2}^{c}(-1)^{d}\frac{q^{-(n'-b+2c-d)}[m'-1]\cdots q^{-(n'-b+c+1)}[m'-c+d]}{[c-d]![d-1]!}\\
	&\quad\quad\quad\quad\times[\alpha_{b,c}+2]\cdots [\alpha_{b,c}+2(d-2)] (q^{n-b+c}+q^{n-b+c-2}+\cdots +q^{n-b+c-2d+4})\bigg)\\
	&=\frac{q^{-m'+c}}{[c]}\Omega^{(b,c)}_{m',n',c-1}[n-b+c+1]\\
	&\quad\quad+\frac{q^{-m'+c}}{[c]}q^{n-b+c}[\alpha_{b,c}]\bigg(\sum_{d=2}^{c}(-1)^{d}\frac{q^{-(n'-b+2c-d)}[m'-1]\cdots q^{-(n'-b+c+1)}[m'-c+d]}{[c-d]![d-2]!}\\
	&\quad\quad\quad\quad\times[\alpha_{b,c}+2]\cdots [\alpha_{b,c}+2(d-2)] q^{-(d-2)}\bigg)\\
	&=\frac{q^{-m'+c}}{[c]}\Omega^{(b,c)}_{m',n',c-1}[n-b+c+1]\\
	&\quad\quad+\frac{q^{-m'+c}}{[c]}q^{n-b+c}q^{-(c-2)}[\alpha_{b,c}]\bigg(\sum_{d=2}^{c}(-1)^{d}\frac{q^{-(n'-b+2c-d-1)}[m'-1]\cdots q^{-(n'-b+c)}[m'-c+d]}{[c-d]![d-2]!}\\
	&\quad\quad\quad\quad\times[\alpha_{b,c}+2]\cdots [\alpha_{b,c}+2(d-2)]\bigg)
\end{align*}
Letting $m''=m'+1$, $n''=n'-1$, $\alpha''_{c}=m''-n''+b-2c=\alpha_{c}+2$, we finally get
$$
\Omega^{(b,c)}_{m,n,c}=\frac{q^{-m'+c}}{[c]}[n-b+c+1]\Omega^{(b,c)}_{m',n',c-1}+\frac{q^{-m'+n-b+c+2}}{[c]}[\alpha_{b,c}]\Omega^{(b,c)}_{m'',n'',c-2}.
$$
This proves Lemma \ref{lemrecur2}.
%\end{proof}

\subsection{Proof of Lemma \ref{lemrecur3}}\label{app5}

%\begin{proof} 
We compute 
\begin{align*}
	O^{(b,c)}_{m,n}&=\sum_{k+y=c}(-1)^{k}[\alpha_{b,c}+k-1]\frac{[\alpha_{b,c}+1][\alpha_{b,c}+3]\cdots[\alpha_{b,c}+2k-3]}{[k]!}q^{-\frac{y(y+1+2n+2c-2b)}{2}}\frac{[m]!}{[y]![m-y]!}\\
	&=\sum_{d=0}^{c}(-1)^{d}[\alpha_{b,c}+d-1]\frac{[\alpha_{b,c}+1][\alpha_{b,c}+3]\cdots[\alpha_{b,c}+2d-3]}{[c-d]![d]!}\\
	&\quad\quad \times \Big(q^{-(1+c+n-b)}[m-c+d+1]\Big)\Big(q^{-(2+c+n-b)}[m-c+d+2]\Big)\cdots \Big(q^{-(2c+n-b-d)}[m]\Big)\\
	&=\frac{1}{[c]!}\sum_{d=1}^{c}q^{-(n-b+2c)}[m]q^{-(n-b+2c-1)}[m-1]\cdots q^{-(n-b+c+1+d)}[m-c+1+d]\\
	&\quad\quad\times [\alpha_{b,c}+d-1][\alpha_{b,c}+1][\alpha_{b,c}+3]\cdots [\alpha_{b,c}+1+2(d-3)]q^{-m+c-d}[n-b+c+2-d]\epsilon_{d}\\
	&\quad+\frac{1}{[c]!}\sum_{d=2}^{c}q^{-(n-b+2c)}[m]q^{-(n-b+2c-1)}[m-1]\cdots q^{-(n-b+c+1+d)}[m-c+1+d]\\
	&\quad\quad \times [\alpha_{b,c}+1][\alpha_{b,c}+3]\cdots [\alpha_{b,c}+2d-5][d-1](-1)^{d}q^{(d-1)(c-d)}\qbinom{c-1}{d-1}\\
	&=\frac{1}{[c]!}\sum_{d=1}^{c}q^{-(n-b+2c)}[m]q^{-(n-b+2c-1)}[m-1]\cdots q^{-(n-b+c+1+d)}[m-c+1+d]\\
	&\quad\quad\times [\alpha_{c}+d-1][\alpha_{b,c}+1][\alpha_{b,c}+3]\cdots [\alpha_{b,c}+1+2(d-3)]q^{-m+c-d}[n-b+c+2-d]\epsilon_{d}\\
	&\quad +\frac{1}{[c]}\sum_{d=2}^{c}(-1)^{d}q^{-(n-b+2c-d+1)}[m]q^{-(n-b+2c-d)}[m-1]\cdots q^{-(n-b+c+2)}[m-c+1+d]\\
	&\quad\quad \times [\alpha_{b,c}+1][\alpha_{b,c}+3]\cdots [\alpha_{b,c}+2d-5]\frac{1}{[c-d]![d-2]!}.
\end{align*}
The rest is to deal with the first summand. Similarly, we have
\begin{align*}
	O^{(b,c)}_{m,n}&=\frac{q^{-m-1+c}}{[c]}\Omega^{(b,c)}_{m+1,n+1,c-1}[n-b+c+1]\\
	&\quad+\frac{q^{-m-1+c}}{[c]}\bigg(\sum_{d=2}^{c}(-1)^{d}\frac{q^{-(n+1-b+2c-d)}[m]\cdots q^{-(n+1-b+c+1)}[m+1-c+d]}{[c-d]![d-1]!}[\alpha_{b,c}+d-2]\\
	&\quad\quad\times[\alpha_{b,c}+1][\alpha_{b,c}+3]\cdots [\alpha_{b,c}+1+2(d-3)] (q^{n-b+c}+q^{n-b+c-2}+\cdots +q^{n-b+c-2d+4})\bigg)\\
	&\quad +\frac{1}{[c]}\sum_{d=2}^{c}(-1)^{d}q^{-(n-b+2c-d+1)}[m]q^{-(n-b+2c-d)}[m-1]\cdots q^{-(n-b+c+2)}[m-c+1+d]\\
	&\quad\quad \times [\alpha_{b,c}+1][\alpha_{b,c}+3]\cdots [\alpha_{b,c}+2d-5][d-1]\frac{1}{[c-d]![d-2]!}\\
	&=\frac{q^{-m-1+c}}{[c]}\Omega^{(b,c)}_{m+1,n+1,c-1}[n-b+c+1]\\
	&\quad+\frac{q^{-m-1+c}}{[c]}q^{n-b+c}\bigg(\sum_{d=2}^{c}(-1)^{d}\frac{q^{-(n+1-b+2c-d)}[m]\cdots q^{-(n-b+c+2)}[m+1-c+d]}{[c-d]![d-2]!}[\alpha_{b,c}+d-2]\\
	&\quad\quad\times[\alpha_{b,c}+1][\alpha_{b,c}+3]\cdots [\alpha_{b,c}+2d-5] q^{-(d-2)}\bigg)\\
	&\quad +\frac{1}{[c]}\sum_{d=2}^{c}(-1)^{d}q^{-(n-b+2c-d+1)}[m]q^{-(n-b+2c-d)}[m-1]\cdots q^{-(n-b+c+2)}[m-c+1+d]\\
	&\quad\quad \times [\alpha_{b,c}+1][\alpha_{b,c}+3]\cdots [\alpha_{b,c}+2d-5]\frac{1}{[c-d]![d-2]!}\\
	&=\frac{q^{-m-1+c}}{[c]}\Omega^{(b,c)}_{m+1,n+1,c-1}[n-b+c+1]\\
	&\quad+\frac{1}{[c]}\sum_{d=2}^{c}(-1)^{d}\frac{q^{-(n+1-b+2c-d)}[m]\cdots q^{-(n-b+c+2)}[m+1-c+d]}{[c-d]![d-2]!}\\
	&\quad\quad\times[\alpha_{b,c}+1][\alpha_{b,c}+3]\cdots [\alpha_{b,c}+2d-5]\\
	&\quad\quad\times\Big(q^{-m+n-b+2c+1-d}[\alpha_{b,c}+d-2]+1\Big)\\
	&=\frac{q^{-m-1+c}}{[c]}\Omega^{(b,c)}_{m+1,n+1,c-1}[n-b+c+1]+\frac{q^{-m+n-b+2c}}{[c]}[m-n+b-2c+1]\Omega^{(b,c-2)}_{m,n+3}\\
	&\quad+\frac{q^{-2m+2n-2b+4c}}{[c]}\sum_{d=3}^{c}(-1)^{d}\frac{q^{-(n+1-b+2c-d)}[m]\cdots q^{-(n-b+c+2)}[m+1-c+d]}{[c-d]![d-2]!}\\
	&\quad\quad\times[\alpha_{c}+1][\alpha_{c}+3]\cdots [\alpha_{c}+2d-5](q^{-2}+q^{-4}+\cdots+q^{-2(d-2)})\\
	&=\frac{q^{-m-1+c}}{[c]}\Omega^{(b,c)}_{m+1,n+1,c-1}[n-b+c+1]+\frac{q^{-m+n-b+2c}}{[c]}[m-n+b-2c+1]\Omega^{(b,c-2)}_{m,n+3}\\
	&\quad+\frac{q^{-2m+2n-2b+4c}}{[c]}q^{-c+1}\sum_{d=3}^{c}(-1)^{d}\frac{q^{-(n+1-b+2c-d)}[m]\cdots q^{-(n-b+c+2)}[m+1-c+d]}{[c-d]![d-3]!}\\
	&\quad\quad\times[\alpha_{b,c}+1][\alpha_{b,c}+3]\cdots [\alpha_{b,c}+2d-5]q^{c-d}\\
	&=\frac{q^{-m-1+c}}{[c]}\Omega^{(b,c)}_{m+1,n+1,c-1}[n-b+c+1]+\frac{q^{-m+n-b+2c}}{[c]}[m-n+b-2c+1]\Omega^{(b,c-2)}_{m,n+3}\\
	&\quad-\frac{q^{-2m+2n-2b+3c+1}}{[c]}[m-n+b-2c+1]\Omega_{m+3,n,b,c,c-3}.
\end{align*}
This completes the proof of Lemma \ref{lemrecur3}. 
%\end{proof}

\subsection{Proof of Lemma \ref{lemrecursiveo}}
\label{app6}
We prove \eqref{oandomega} only. We can write 
\begin{align*}
	O^{a,b,c}_{m,n,\bm{\varsigma}}&=q^{-\frac{c(-1+2n+3c-2b)}{2}}\qbinom{m}{c}\sum_{x+k=a}(-1)^{k}\frac{[\beta_{a,b,c,\bm{\varsigma}}][\beta_{a,b,c,\bm{\varsigma}}+2]\cdots [\beta_{a,b,c,\bm{\varsigma}}+2k-2]}{[k]!}q^{\varsigma_{1}(x-c)}\nonumber\\
	&\quad\quad\times q^{-\frac{x(x+1+2m+2b-4c-2a)}{2}+xy-ck}\qbinom{n-b+c+a}{x}.
\end{align*}
The key observation is that
\begin{align*}
	\Omega^{c,b',a,\bm{\varsigma}'}_{m',n'}&=q^{-\frac{c(-1+2m'+2b'-4a-c)}{2}}\qbinom{n'+a-b'+c}{c}\sum_{x+k=a}(-1)^{k}\frac{[\alpha_{b',a,\bm{\varsigma}'}][\alpha_{b',a,\bm{\varsigma}'}+2]\cdots [\alpha_{b',a,\bm{\varsigma}'}+2k-2]}{[k]!}\nonumber\\
	&\quad\quad\times q^{-\frac{x(x+1+2n'+2a-2b')}{2}+xy-ck}q^{\varsigma'_{2}(x-c)}\qbinom{m'}{x}\\
	&=q^{-\frac{c(-1+2n+3c-2b)}{2}}\qbinom{m+b-a-c}{c}\sum_{x+k=a}(-1)^{k}\frac{[\beta_{a,b,c,\bm{\varsigma}}][\beta_{a,b,c,\bm{\varsigma}}+2]\cdots [\beta_{a,b,c,\bm{\varsigma}}+2k-2]}{[k]!}\nonumber\\
	&\quad\quad\times q^{-\frac{x(x+1+2m+2b-4c-2a)}{2}+xy-ck}q^{\varsigma_{1}(x-c)}\qbinom{n-b+c+a}{x},
\end{align*}
where $b'=a+c$, $m'=n-b+c+a$, $n'=m+b-a-c$, $\bm{\varsigma}'=1-\bm{\varsigma}$, and $\alpha_{a+c,a,\bm{\varsigma}'}=m'-n'+b-2a+(1-\varsigma_{2})=n-b+c+a-(m+b-a-c)+a+c-2a+\varsigma_{1}=n-m+a-2b+3c+\varsigma_{1}=\beta_{a,b,c,\bm{\varsigma}}$.

Hence, we get
\begin{align*}
	\qbinom{m+b-a-c}{c}O^{a,b,c}_{m,n}=\qbinom{m}{c}\Omega^{c,b',a}_{m',n'}.
\end{align*}
This completes the proof of Lemma \ref{lemrecursiveo}.

\subsection{Proof of Proposition \ref{thm811}} \label{app8}

First, \eqref{833} follows from Lemmas \ref{lemrecur} and \ref{lemoddddd}. When $\alpha_{c}$ is even, notice that \eqref{844} can be rewritten as
\begin{align}\label{855}
	\Omega^{b,c}_{m,n}&=\sum_{x=0}^{\beta_{c}}\qbinom{\gamma_{c}+x-1}{x}_{q^2}\qbinom{n-b+2c-2x}{c-2x}q^{\frac{(2x+2-c)(c-1-2m)}{2}}q^{x(-m+n-b+c+1)}\nonumber\\
	&=\sum_{x=0}^{\beta_{c}}\frac{q^{-m+c-1-x}\cdots q^{-m+x}}{[2x]\cdots [4][2][c-2x]!}[n-b+c+1]\cdots [n-b+2c-2x]\nonumber\\
	&\quad \quad \quad \quad  \times [m-n+b-2c]\cdots [m-n+b-2c+2(x-1)]q^{x(-m+n-b+c+1)}\nonumber\\
	&=\frac{q^{-m+c-1}q^{-m+c-2}\cdots q^{-m}}{[c]!}[n-b+c+1][n-b+c+2]\cdots [n-b+2c]\nonumber\\
	&\quad \quad +\frac{q^{-m+c-2}q^{-m+c-3}\cdots q^{-m+1}}{[2][c-2]!}[n-b+c+1][n-b+c+2]\cdots [n-b+2c-2]\nonumber\\
	&\quad \quad \quad \quad  \times [m-n+b-2c]q^{-m+n-b+c+1}\nonumber\\
	&\quad \quad +\frac{q^{-m+c-3}q^{-m+c-4}\cdots q^{-m+2}}{[4][2][c-4]!}[n-b+c+1][n-b+c+2]\cdots [n-b+2c-4]\nonumber\\
	&\quad \quad \quad \quad \times [m-n+b-2c][m-n+b-2c+2]q^{2(-m+n-b+c+1)}\nonumber\\
	&\quad \quad +\cdots\nonumber\\
	&\quad \quad +\frac{q^{-m+c-1-\beta_{c}}\cdots q^{-m+\beta_{c}}}{[2\beta_{c}]\cdots [4][2][c-2\beta_{c}]!}[n-b+c+1]\cdots [n-b+2c-2\beta_{c}]\nonumber\\
	&\quad \quad \quad \quad  \times [m-n+b-2c]\cdots [m-n+b-2c+2(\beta_{c}-1)]q^{\beta_{c}(-m+n-b+c+1)}.
\end{align}
We prove \eqref{855} by induction on $c$. First, assume that the statement is true for any $k\leq c-1$, for some even $c-1$. Now for $c$, by Lemma \ref{lemrecur2}, (here $\beta_{c}=\beta_{c-1}=\beta_{c-2}+1$) we have
\begin{align*}
	\Omega^{b,c}_{m,n}&=\frac{q^{-m-1+c}}{[c]}[n-b+c+1]\Omega_{m+1,n+1,b,c,c-1}+\frac{q^{-m+n-b+c+1}}{[c]}[\alpha_{c}]\Omega_{m+2,n,b,c,c-2}\\
	&=\frac{q^{-m-1+c}}{[c]}[n-b+c+1]\Omega_{m,n+2,b,c-1,c-1}+\frac{q^{-m+n-b+c+1}}{[c]}[\alpha_{c}]\Omega_{m,n+2,b,c-2,c-2}\\
	&=\frac{q^{-m-1+c}}{[c]}[n-b+c+1]\sum_{x=0}^{\beta_{c}}\frac{q^{-m+c-2-x}\cdots q^{-m+x}}{[2x]\cdots [4][2][c-1-2x]!}[n-b+c+2]\cdots [n-b+2c-2x]\\\\
	&\quad \quad \quad \quad  \times [m-n+b-2c]\cdots [m-n+b-2c+2(x-1)]q^{x(-m+n-b+c+2)}\\
	&\quad \quad +\frac{q^{-m+n-b+c+1}}{[c]}[\alpha_{c}]\sum_{x=0}^{\beta_{c}-1}\frac{q^{-m+c-3-x}\cdots q^{-m+x}}{[2x]\cdots [4][2][c-2-2x]!}[n-b+c+1]\cdots [n-b+2c-2x-2]\\\\
	&\quad \quad \quad \quad  \times [m-n+b-2c+2]\cdots [m-n+b-2c+2x]q^{x(-m+n-b+c+1)}\\
	&=\frac{q^{-m-1+c}}{[c]}\sum_{x=0}^{\beta_{c}}\frac{q^{-m+c-2-x}\cdots q^{-m+x}}{[2x]\cdots [4][2][c-1-2x]!}[n-b+c+1][n-b+c+2]\cdots [n-b+2c-2x]\\\\
	&\quad \quad \quad \quad  \times [m-n+b-2c]\cdots [m-n+b-2c+2(x-1)]q^{x(-m+n-b+c+2)}\\
	&\quad \quad +\frac{1}{[c]}\sum_{x=0}^{\beta_{c}-1}\frac{q^{-m+c-3-x}\cdots q^{-m+x}}{[2x]\cdots [4][2][c-2-2x]!}[n-b+c+1]\cdots [n-b+2c-2x-2]\\\\
	&\quad \quad \quad \quad  \times [m-n+b-2c] [m-n+b-2c+2]\cdots [m-n+b-2c+2x]q^{(x+1)(-m+n-b+c+1)}\\
\end{align*}
Now we pair the $x+1$-th term in the first summand with the $x$-th term in the second summand, and we get
\begin{align*}
	&\frac{q^{-m-1+c}}{[c]}\frac{q^{-m+c-3-x}\cdots q^{-m+x+1}}{[2x+2]\cdots [4][2][c-3-2x]!}[n-b+c+1]\cdots [n-b+2c-2x-2]\\\\
	&\quad \quad \quad \quad  \times [m-n+b-2c]\cdots [m-n+b-2c+2x]q^{(x+1)(-m+n-b+c+2)}\\
	&\quad \quad +\frac{1}{[c]}\frac{q^{-m+c-3-x}\cdots q^{-m+x}}{[2x]\cdots [4][2][c-2-2x]!}[n-b+c+1]\cdots [n-b+2c-2x-2]\\\\
	&\quad \quad \quad \quad  \times [m-n+b-2c]\cdots [m-n+b-2c+2x]q^{(x+1)(-m+n-b+c+1)}\\
	&\quad=\frac{1}{[c]}\frac{q^{-m+c-3-x}\cdots q^{-m+x+1}}{[2x]\cdots [4][2][c-3-2x]!}[n-b+c+1]\cdots [n-b+2c-2x-2]\\
	&\quad \quad \quad \quad  \times [m-n+b-2c]\cdots [m-n+b-2c+2x]q^{(x+1)(-m+n-b+c+1)}\\
	&\quad \quad \quad \quad  \times \bigg( \frac{q^{-m+x+c}}{[2x+2]}+\frac{q^{-m+x}}{[c-2-2x]}\bigg)\\
	&\quad=\frac{1}{[c]}\frac{q^{-m+c-3-x}\cdots q^{-m+x+1}}{[2x+2][2x]\cdots [4][2][c-2-2x]!}[n-b+c+1]\cdots [n-b+2c-2x-2]\\
	&\quad \quad \quad \quad  \times [m-n+b-2c]\cdots [m-n+b-2c+2x]q^{(x+1)(-m+n-b+c+1)}\\
	&\quad \quad \quad \quad  \times \big([c-2-2x]q^{-m+x+c}+[2x+2]q^{-m+x}\big)\\
	&\quad=\frac{q^{-m+c-2-x}q^{-m+c-3-x}\cdots q^{-m+x+1}}{[2x+2][2x]\cdots [4][2][c-2-2x]!}[n-b+c+1]\cdots [n-b+2c-2x-2]\\
	&\quad \quad \quad \quad  \times [m-n+b-2c]\cdots [m-n+b-2c+2x]q^{(x+1)(-m+n-b+c+1)},
\end{align*}
which is the $x+1$-th term in the expression of $\Omega^{b,c}_{m,n}$ for $0\leq x\leq \beta_{c}-1$. In the last step, we used the fact that
\begin{align*}
	[c-2-2x]q^{-m+x+c}+[2x+2]q^{-m+x}=[c]q^{-m+c-2-x}.
\end{align*}

When it comes to the case where $c-1$ is odd, we get $\beta_{c}=\beta_{c-1}+1=\beta_{c-2}+1$ and now we still pair the $x+1$-th term in the first summand with the $x$-th term in the second summand for $0\leq x\leq \beta_{c}-2$ and by the same argument as above we still get the $x+1$-th term in the expression of $\Omega^{b,c}_{m,n}$. The first and last terms in the expression of $\Omega^{b,c}_{m,n}$ come from the first term in the first summand and the last term in the second summand, respectively. This completes the proof of Proposition \ref{thm811}.

\end{document}